%% file: main.tex
\documentclass{article}
\input{./header.tex}
\title{\bfseries Motivic Higman's Conjecture}
\author{Jesse Vogel}
\affil{\footnotesize Mathematical Institute, Leiden University, j.t.vogel@math.leidenuniv.nl}
\date{}

\begin{document}

\maketitle
\vspace{-0.75cm}
\input{abstract}

\input{introduction}
\input{grothendieck}

\input{tqft_method}
\input{arithmetic_method}

\input{conclusion}


\clearpage

\printbibliography


\end{document}

%% file: header.tex
\usepackage[utf8]{inputenc}
\usepackage{amsmath}
\usepackage{amsthm}
\usepackage{amssymb}
\usepackage{amsfonts}
\usepackage{authblk}
\usepackage{mathtools}
\usepackage{stmaryrd}
\usepackage{tikz-cd}
\usepackage{enumerate}
\usepackage{hyperref} \hypersetup{colorlinks=true,citecolor=[rgb]{0.153,0.255,0.545},urlcolor=black,colorlinks=true,linkcolor=[rgb]{0.745,0.098,0.031}}
\usepackage{bbm}
\usepackage{multirow}
\usepackage{multicol}

\usepackage[style=numeric,doi=false,url=false,isbn=false]{biblatex}
\renewbibmacro{in:}{}

\usepackage[top=1in, bottom=1.25in, left=1.25in, right=1.25in]{geometry}

\newcommand{\ZZ}{\mathbb{Z}}

\newcommand{\CC}{\mathbb{C}}
\newcommand{\FF}{\mathbb{F}}

\newcommand{\TT}{\mathbb{T}}
\newcommand{\GG}{\mathbb{G}}

\newcommand{\K}{\textup{\normalfont K}}
\renewcommand{\AA}{\mathbb{A}}

\newcommand{\tr}{\textup{tr}}

\newcommand{\id}{\textup{id}}
\newcommand{\SL}{\textup{SL}}
\newcommand{\GL}{\textup{GL}}
\newcommand{\AGL}{\textup{AGL}}
\newcommand{\PGL}{\textup{PGL}}
\newcommand{\UU}{\mathbb{U}}
\newcommand{\Stab}{\textup{Stab}}
\newcommand{\Orbit}{\textup{Orbit}}

\newcommand{\Var}{\textup{\bfseries Var}}
\newcommand{\Mod}{\textup{\bfseries Mod}}
\DeclareMathOperator{\Spec}{Spec}

\newcommand{\Hom}{\text{Hom}}

\DeclareMathAlphabet\mathbfcal{OMS}{cmsy}{b}{n}

\newcommand{\smatrix}[1]{\scalebox{1.1}{\renewcommand{\arraystretch}{0.9}\setlength\arraycolsep{2pt}\scriptsize$\begin{pmatrix} #1 \end{pmatrix}$}}

\newtheorem{theorem}{Theorem}[section]
\newtheorem{conjecture}{Conjecture}[section]
\newtheorem{proposition}[theorem]{Proposition}
\newtheorem{lemma}[theorem]{Lemma}
\newtheorem{corollary}[theorem]{Corollary}
\theoremstyle{definition}
\newtheorem{example}[theorem]{Example}
\newtheorem{notation}[theorem]{Notation}
\newtheorem{remark}[theorem]{Remark}
\newtheorem{algorithm}[theorem]{Algorithm}
\newtheorem{definition}[theorem]{Definition}
\numberwithin{theorem}{section}
\numberwithin{equation}{section}

\input{./bordisms.tex}


%% file: bordisms.tex
\newcommand{\bdscale}{0.33}

\def\bdcupleft {
\begin{tikzpicture}[semithick, scale=\bdscale, baseline=-0.5ex]
\begin{scope}
    \draw (0,0) ellipse (0.2cm and 0.4cm);
    \draw (0,-0.4) arc (-90:90:0.75cm and 0.4cm);
\end{scope}
\end{tikzpicture}
}

\def\bdcupright {
\begin{tikzpicture}[semithick, scale=\bdscale, baseline=-0.5ex]
\begin{scope}
    \draw (0,0) ellipse (0.2cm and 0.4cm);
    \draw (0,0.4) arc (90:270:0.75cm and 0.4cm);
\end{scope}
\end{tikzpicture}
}

\newcommand{\bdgenus}[1][\bdscale] {
\begin{tikzpicture}[semithick, scale=#1, baseline=-0.5ex]
\begin{scope}
    \draw (-1,0) ellipse (0.2cm and 0.4cm);
    \draw (-1,0.4) .. controls (-0.5,0.6) and (0.5,0.6) .. (1,0.4);
    \draw (-1,-0.4) .. controls (-0.5,-0.6) and (0.5,-0.6) .. (1,-0.4);
    \draw (-0.5,0.1) .. controls (-0.5,-0.125) and (0.5,-0.125) .. (0.5,0.1);
    \draw (-0.4,0.0) .. controls (-0.4,0.0625) and (0.4,0.0625) .. (0.4,0.0);
    \draw (1,0) ellipse (0.2cm and 0.4cm);
\end{scope}
\end{tikzpicture}
}

\newcommand\bdpantsleft[1][\bdscale] {
\begin{tikzpicture}[semithick, scale=\bdscale, baseline=-0.5ex]
\begin{scope}
    \draw (-1,0.5) ellipse (0.2cm and 0.4cm);
    \draw (-1,-0.5) ellipse (0.2cm and 0.4cm);
    \draw (1,0) ellipse (0.2cm and 0.4cm);
    \draw (-1,0.9) .. controls (0,0.9) and (0,0.4) .. (1,0.4);
    \draw (-1,-0.9) .. controls (0,-0.9) and (0,-0.4) .. (1,-0.4);
    \draw (-1,0.1) .. controls (-0.2,0.1) and (-0.2,-0.1) .. (-1,-0.1);
\end{scope}
\end{tikzpicture}
}

\def\bdgenerictube#1 {
\begin{tikzpicture}[semithick, scale=\bdscale, baseline=-0.5ex]
\begin{scope}
    \draw (-1,0) ellipse (0.2cm and 0.4cm);
    \draw (-1,0.4) cos (-0.875, 0.5) sin (-0.75,0.6) cos (-0.625,0.5) sin (-0.5,0.4) cos (-0.375,0.5) sin (-0.25, 0.6) cos (-0.125,0.5) sin (0, 0.4) cos (0.125,0.5) sin (0.25,0.6) cos (0.375,0.5) sin (0.5,0.4) cos (0.625,0.5) sin (0.75,0.6) cos (0.875,0.5) sin (1,0.4);
    \draw (-1,-0.4) cos (-0.875, -0.5) sin (-0.75,-0.6) cos (-0.625,-0.5) sin (-0.5,-0.4) cos (-0.375,-0.5) sin (-0.25, -0.6) cos (-0.125,-0.5) sin (0,-0.4) cos (0.125,-0.5) sin (0.25,-0.6) cos (0.375,-0.5) sin (0.5,-0.4) cos (0.625,-0.5) sin (0.75,-0.6) cos (0.875,-0.5) sin (1,-0.4);
    \draw (1,0) ellipse (0.2cm and 0.4cm);
    \draw (0,0) node {$#1$};
\end{scope}
\end{tikzpicture}
}

\newcommand\bdparabolic {
\begin{tikzpicture}[semithick, scale=1.25*\bdscale, baseline=-0.5ex]
\begin{scope}
    \draw (-1,0) ellipse (0.2cm and 0.4cm);
    \draw (-1,0.4) -- (1,0.4);
    \draw (-1,-0.4) -- (1,-0.4);
    \draw (1,0.4) arc (90:-90:0.2cm and 0.4cm);
    \draw (1,0.4) arc (90:270:0.2cm and 0.4cm);
    \draw[black,fill=black] (0,-0.10) circle (.4ex);
\end{scope}
\end{tikzpicture}
}

%% file: abstract.tex
\begin{abstract}
    The $G$-representation variety $R_G(\Sigma_g)$ parametrizes the representations of the fundamental groups of surfaces $\pi_1(\Sigma_g)$ into an algebraic group $G$.
    Taking $G$ to be the groups of $n \times n$ upper triangular or unipotent matrices, we compare two methods for computing algebraic invariants of $R_G(\Sigma_G)$.
    Using the geometric method initiated by González-Prieto, Logares and Muñoz, based on a Topological Quantum Field Theory (TQFT), we compute the virtual classes of $R_G(\Sigma_g)$ in the Grothendieck ring of varieties for $n = 1, \ldots, 5$, extending the results of \cite{HablicsekVogel2020}. Introducing the notion of algebraic representatives we are able to efficiently compute the TQFT.
    Using the arithmetic method initiated by Hausel and Rodriguez-Villegas, we compute the $E$-polynomials of $R_G(\Sigma_g)$ for $n = 1, \ldots, 10$.
    For both methods, we describe how the computations can be performed algorithmically.
    Furthermore, we discuss the relation between the representation varieties of the group of unipotent matrices and Higman's conjecture. The computations of this paper can be seen as positive evidence towards a generalized motivic version of the conjecture.
\end{abstract}

%% file: introduction.tex
\section{Introduction}

Let $X$ be a closed connected manifold, and $G$ an algebraic group over a field $k$. The \emph{$G$-representation variety} of $X$ is the set of group homomorphisms
\[ R_G(X) = \Hom(\pi_1(X), G) . \]
As $\pi_1(X)$ is finitely generated, $R_G(X)$ can be seen as a closed subvariety of $G^n$ for some $n \ge 0$. For $X = \Sigma_g$ a closed surface of genus $g$, the $G$-representation variety takes the explicit form
\begin{equation}
    \label{eq:explicit_representation_variety_surface}
    R_G(\Sigma_g) = \left\{ (A_1, B_1, \ldots, A_g, B_g) \in G^{2g} \mid \prod_{i = 1}^{g} [A_i, B_i] = 1 \right\} .
\end{equation}
Closely related is the \emph{$G$-character variety} of $\Sigma_g$, given by the GIT quotient $\chi_G(\Sigma_g) = R_G(\Sigma_g) \sslash G$, where $G$ acts on $R_G(\Sigma_g)$ by conjugation. The $G$-character variety, also known as the \emph{Betti moduli space}, plays a central role in non-abelian Hodge theory \cite{Simpson1994a, Simpson1994b}. In particular, for $C$ a complex smooth projective curve, the Betti moduli space is isomorphic to the moduli space of $G$-flat connections on $C$ by the Riemann--Hilbert correspondence, and for $G = \GL_n$ it is diffeomorphic to the moduli space of polystable $G$-Higgs bundles of rank $n$ and degree $0$ by the non-abelian Hodge correspondence.

\textbf{Invariants}.
Various algebraic invariants of the $G$-representation variety and $G$-character variety have been studied in the literature, for instance \cite{GonzalezPrietoLogaresMunoz2020,HauselRodriguezVillegas2008,LogaresMunozNewstead2013} amongst many, and in the recently resolved $P = W$ conjecture \cite{HauselMellitMinetsSchiffman2022,MaulikShen2022}.
One of these invariants is the \emph{$E$-polynomial}, also known as the \emph{Serre polynomial}, which for a complex variety $X$ is given by
\[ e(X) = \sum_{k, p, q} (-1)^k \, h_c^{k; p, q}(X) \, u^p v^q \in \ZZ[u, v] , \]
where $h_c^{k; p, q}(X)$ are mixed Hodge numbers of the compactly supported cohomology of $X$. The $E$-polynomial can be shown \cite{HauselRodriguezVillegas2008,LogaresMunozNewstead2013} to satisfy
\[ e(X) = e(Z) + e(X \setminus Z) \quad \textup { and } \quad e(X \times_\CC Y) = e(X) \, e(Y) , \]
for complex varieties $X$ and $Y$, where $Z \subset X$ is a closed subvariety with open complement $X \setminus Z$. The \emph{Grothendieck ring of varieties} $\K(\Var_\CC)$, which will be defined in Section \ref{sec:grothendieck_ring}, is the universal ring for invariants with these properties, so in particular there is a morphism
\begin{equation}
    \label{eq:grothendieck_ring_to_Zuv}
    e : \K(\Var_\CC) \to \ZZ[u, v]
\end{equation}
which sends the class $[X]$ of a complex variety $X$ to its $E$-polynomial $e(X)$. In this sense, the class of a variety $X$ in $\K(\Var_\CC)$, also known as the \emph{virtual class} of $X$, is a more refined invariant than the $E$-polynomial.

Another such invariant is the point count over a finite field, that is, there is a morphism
\[ \# : \K(\Var_{\FF_q}) \to \ZZ \]
which sends the class $[X]$ of a variety $X$ over $\FF_q$ to $|X(\FF_q)|$. A remarkable theorem by Katz \cite[Theorem 6.1.2]{HauselRodriguezVillegas2008} states that
if $X$ is a complex variety with a spreading-out $\tilde{X}$ over a finitely generated $\ZZ$-algebra $R \subset \CC$, such that $|(\tilde{X} \times_R \FF_q)(\FF_q)|$ is a polynomial in $q$ for all ring morphisms $R \to \FF_q$, then the $E$-polynomial of $X$ is precisely this polynomial in $q = uv$.
For example, for the affine line $X = \AA^1_\CC$, we have $|\AA^1_{\FF_q}(\FF_q)| = q$ and $e(X) = uv$, and for this reason we usually write $q = uv$ and $q = [\AA^1_k]$, depending on the context.











    
\textbf{Previous work}.
Over the last years, various methods have been developed and used to compute algebraic invariants of the $G$-representation varieties of $\Sigma_g$.
In \cite{HauselRodriguezVillegas2008}, Hausel and Rodriguez--Villegas initiated the \emph{arithmetic method} in order to compute the $E$-polynomials of the $\GL_n$-representation varieties and twisted $\GL_n$-character varieties. This method makes use of Katz' theorem to reduce the problem to counting the points of $R_G(\Sigma_g)$ over finite fields $\FF_q$.
Frobenius' formula, as proven in \cite[Equation 2.3.8]{HauselRodriguezVillegas2008}, gives an expression for the point count of the representation variety for finite groups $G$,
\[ |R_G(\Sigma_g)| = |G| \sum_{\chi \in \widehat{G}} \left(\frac{|G|}{\chi(1)}\right)^{2g - 2} \]
reducing the problem to the study of the irreducible (complex) characters $\chi \in \widehat{G}$ of $G$ over finite fields $\FF_q$.
In \cite{Mereb2018}, this method was used to compute the $E$-polynomials of twisted $\SL_2$-character varieties, and in \cite{BaragliaHekmati2017} it was used to compute the $E$-polynomials of the $\GL_3$- and $\SL_3$-character varieties, and those of the $\GL_2$- and $\SL_2$-character varieties of non-orientable surfaces.

Logares, Martínez, Muñoz and Newstead used geometric arguments in order to compute $E$-polynomials of the $\SL_2$-character varieties \cite{LogaresMunozNewstead2013,MartinezMunoz2016} and $\PGL_2$-character varieties \cite{Martinez2017}.
This approach was reformulated and generalized by González-Prieto, Logares and Muñoz \cite{GonzalezPrietoLogaresMunoz2020}, who constructed a \emph{Topological Quantum Field Theory} (TQFT) that computes the virtual class of $R_G(\Sigma_g)$ in $\K(\Var_k)$. Using this TQFT, they computed the virtual class of the $\AGL_1$-representation variety, and in \cite{GonzalezPrieto2020} it was used to compute the virtual class of the $\SL_2$-representation variety and the twisted $\SL_2$-character variety. Concretely, this method defines a lax monoidal functor $Z : \textbf{Bord}_2 \to \K(\Var_k)\textup{-}\Mod$, from the category of (pointed) $2$-bordisms to the category of $\K(\Var_k)$-modules, such that $Z(\Sigma_g)(1) = [R_G(\Sigma_g)]$. Using functoriality, it suffices to compute
\[ Z(\bdcupright) : \K(\Var/G) \to \K(\Var_k), \quad Z(\bdgenus) : \K(\Var/G) \to \K(\Var/G), \]
\[ \textup{ and } \quad Z(\bdcupleft) : \K(\Var_k) \to \K(\Var/G) \]
to obtain $Z(\Sigma_g)$ by composition.

\textbf{Outline}. The goal of this paper is to compute the $E$-polynomials and virtual classes of the $G$-representation varieties of $\Sigma_g$ for $G$ equal to the group $\TT_n$ of upper triangular $n \times n$ matrices or the group $\UU_n$ of unipotent $n \times n$ matrices, making use of both the TQFT method and the arithmetic method. Therefore, this work is an extension of \cite{HablicsekVogel2020}, in which the virtual class of $R_{\TT_n}(\Sigma_g)$ was computed for $G = \TT_n$ and $1 \le n \le 4$ using the TQFT method. In particular, in Section \ref{sec:tqft_method} we will compute the virtual classes of $R_{\TT_n}(\Sigma_g)$ and $R_{\UU_n}(\Sigma_g)$ for $1 \le n \le 5$.

In this paper, we aim for an organized and concise framework. Rather than working with the explicit description of the TQFT as done in \cite{GonzalezPrieto2020,GonzalezPrietoHablicsekVogel2022,GonzalezPrietoLogaresMunoz2020}, we will be working with the essential maps as in Definition \ref{def:tqft_maps}, avoiding the need for a `reduction' or `simplification' of the TQFT, as appears in \cite{GonzalezPrieto2020,HablicsekVogel2020}. Introducing the notion of \emph{algebraic representatives} in Subsection \ref{subsec:algebraic_representatives}, we can take maximal advantage of the lack of monodromy in the conjugacy classes of $\TT_n$ and $\UU_n$ in order to simplify the computations. Moreover, the computations are organized in such a way that most of the computations for $\TT_n$ can be reused for $\UU_n$.

Furthermore, in Section \ref{sec:arithmetic_method}, we will apply the arithmetic method to compute $E$-polynomials of $R_{\TT_n}(\Sigma_g)$ and $R_{\UU_n}(\Sigma_g)$ for $1 \le n \le 10$. In particular, we compute the representation theory of $\TT_n$ and $\UU_n$ over finite fields $\FF_q$, encoded in the \emph{representation zeta function}. These zeta functions can be computed using a recursive algorithm, Algorithm \ref{alg:representation_zeta_functions}.

\textbf{Results}. In Theorem \ref{thm:virtual_class_T5_representation_variety}, the virtual class of the $\TT_5$-representation variety $R_{\TT_5}(\Sigma_g)$ is given, and in Theorem \ref{thm:virtual_class_Un_representation_varieties}, the virtual classes of the $\UU_n$-representation varieties $R_{\UU_n}(\Sigma_g)$ are given for $1 \le n \le 5$.

In Theorem \ref{thm:representation_zeta_functions_Un}, the representation zeta functions of $\UU_n$ are given for $1 \le n \le 10$, and in Theorem \ref{thm:representation_zeta_functions_Tn} those of $\TT_n$ for $1 \le n \le 10$. These directly determine the $E$-polynomials of $R_{\UU_n}(\Sigma_g)$ and $R_{\TT_n}(\Sigma_G)$, and for $1 \le n \le 5$ these $E$-polynomials can be seen to agree with the virtual classes through the map \eqref{eq:grothendieck_ring_to_Zuv}.



The computation of the representation zeta functions of $\TT_n$ and $\UU_n$ can be seen as a generalization of \cite{VeraLopezArregi2003} and \cite{PakSoffer2015}, where the number of conjugacy classes of $\UU_n(\FF_q)$, denoted $k(\UU_n(\FF_q))$, was computed as a polynomial in $q$ for $1 \le n \le 13$ and $1 \le n \le 16$, respectively. Namely, since the number of irreducible representations of a finite group equals its number of conjugacy classes, evaluating the representation zeta function $\zeta_G(s)$ in $s = 0$ yields $k(G)$. This relates the zeta functions $\zeta_{\UU_n(\FF_q)}(s)$ to a conjecture by Higman first posed in \cite{Higman1960a}.

\begin{conjecture}[Higman]
    \label{conj:higman}
    For every $n \ge 1$, the number of conjugacy classes of $\UU_n(\FF_q)$ is a polynomial in $q$.
\end{conjecture}

As mentioned, this conjecture has been verified for $1 \le n \le 16$. However, as shown in \cite{HalasiPalfy2011}, there exist \emph{pattern subgroups} of $\UU_n$ whose number of conjugacy classes over $\FF_q$ is not polynomial in $q$.
In \cite{PakSoffer2015}, a concrete such pattern subgroup $P \subset \UU_{13}$ is given, and moreover it is shown that $k(\UU_{59})$ can be expressed as a $\ZZ[q]$-linear combination of $k(P)$ and $k(P')$ for other pattern subgroups $P'$. This suggests the conjecture might fail for $n = 59$, unless the non-polynomial terms in the linear combination cancel.

Looking at the results of Theorem \ref{thm:representation_zeta_functions_Un} and \ref{thm:representation_zeta_functions_Tn}, we pose a generalized version of Higman's conjecture.

\begin{conjecture}[Generalized Higman]
    \label{conj:generalized_higman}
    For any $n \ge 1$, the representation zeta function $\zeta_{\UU_n(\FF_q)}(s)$ is a polynomial in $q$ and $q^{-s}$, and $\zeta_{\TT_n(\FF_q)}(s)$ is a polynomial in $q, q^{-s}$ and $(q - 1)^{-s}$.
\end{conjecture}

Similarly, looking at the virtual classes of $R_{\TT_n}(\Sigma_g)$ and $R_{\UU_n}(\Sigma_g)$, we pose a motivic version of the conjecture. Note that Conjecture \ref{conj:motivic_higman} does not imply Conjecture \ref{conj:generalized_higman}, since a variety $X$ could have virtual class in $\ZZ[q]$ without having polynomial count, e.g. $X = \{ (x, y) \mid x^2 + y^2 = 1 \}$. 

\begin{conjecture}[Motivic Higman]
    \label{conj:motivic_higman}
    For any $g \ge 0$ and $n \ge 1$, the virtual classes of the $\TT_n$- and $\UU_n$-representation variety of $\Sigma_g$ are polynomial in $q = [\AA^1_\CC]$.
\end{conjecture}


%% file: grothendieck.tex
\section{Grothendieck ring of varieties}
\label{sec:grothendieck_ring}

Let $k$ be a field. By a \emph{variety} over $k$ we understand a reduced separated scheme of finite type over a field $k$, not necessarily irreducible. All varieties are assumed to be over $k$.


\begin{definition}
    Let $S$ be a variety. The \emph{Grothendieck ring of varieties} over $S$, denoted $\K(\Var/S)$, is the quotient of the free abelian group on isomorphism classes of varieties over $S$, by relations of the form
    \[ [X] = [Z] + [U] , \]
    for all closed subvarieties $Z \subset X$ with open complement $U = X \setminus Z$. Multiplication on $\K(\Var/S)$ is defined by the fiber product over $S$,
    \[ [X] \cdot [Y] = [X \times_S Y] , \]
    and this is easily seen to give a ring structure on $\K(\Var/S)$. In particular, the unit element is the class $\textbf{1}_S = [S]$ of $S$ over itself, and the zero element is the class $\textbf{0}_S = [\varnothing]$. The class $[X]$ of a variety $X$ in $\K(\Var/S)$ is also called its \emph{virtual class}.
\end{definition}

In order to distinguish between classes over different bases, we write $[X]_S$ for the class of $X$ in $\K(\Var/S)$, and when the base is $S = \Spec k$, we simply write $[X]$. Furthermore, we denote by $q = [\AA^1_k] \in \K(\Var_k)$ the class of the affine line.

For any variety $S$, the ring $\K(\Var/S)$ naturally has the structure of a $\K(\Var_k)$-module, where scalar multiplication is given by
\[ [T] \cdot [X]_S = [T \times X]_S \]
for all varieties $T$ and varieties $X$ over $S$, and extended linearly.

For any morphism $f : T \to S$ of varieties, there are induced maps
\begin{align*}
    f_! &: \K(\Var/T) \to \K(\Var/S), \quad [X]_T \mapsto [X]_S, \\
    f^* &: \K(\Var/S) \to \K(\Var/T), \quad [X]_S \mapsto [X \times_S T] ,
\end{align*}
which are seen to be $\K(\Var_k)$-module morphisms. Moreover, the map $f^*$ is a ring morphism.

For any variety $S$, let $c : S \to \Spec k$ denote the final morphism. In particular, for any variety $X$ over $S$, we have
\[ c_! [X]_S = [X] . \]

\subsection{Stratifications}

\begin{definition}
    Let $X$ be a variety. A \emph{stratification} of $X$ is a collection of locally closed subvarieties $\{ X_i \}_{i \in I}$ such that $\bigcup_{i \in I} X_i = X$ and $X_i \cap X_j = \varnothing$ for $i \ne j$. 
\end{definition}

\begin{lemma}
    \label{lemma:stratifications_are_finite}
    Let $X$ be a variety over $S$, with stratification $\{ X_i \}_{i \in I}$. Then only finitely many of the $X_i$ are non-empty, and $\sum_{i \in I} [X_i]_S = [X]_S$.
\end{lemma}
\begin{proof}
    Proof by induction on the dimension of $X$. If $\dim X = 0$, then $X$ is a finite set of points, and the result is clear. Now assume that $\dim X > 0$ and that the result holds for all varieties of dimension less than $\dim X$.
    
    First consider the case where $X$ is irreducible. Some $U = X_i$ contains the generic point of $X$, and is therefore open. The complement $Z = X \setminus U$ is of smaller dimension than $X$ and is stratified by the other $X_i$. Since $[X]_S = [Z]_S + [U]_S$, the result follows from the induction hypothesis.
    
    Now consider the case where $X$ is reducible. Take an irreducible component and remove the intersections with the other irreducible components, which gives an irreducible open subset $U \subset X$. The complement $Z = X \setminus U$ is a closed subvariety with fewer irreducible components than $X$. Since $\{ Z \cap X_i \}_{i \in I}$ is a stratification of $Z$ and $\{ U \cap X_i \}_{i \in I}$ a stratification of $U$, apply induction on the number of irreducible components of $X$ to find that only finitely many are non-empty, and we have
    \[ [U]_S = \sum_{i \in I} [U \cap X_i]_S \quad \textup{ and } \quad [Z]_S = \sum_{i \in I} [Z \cap X_i]_S . \]
    Since $[X_i]_S = [U \cap X_i]_S + [Z \cap X_i]_S$ for all $i$, it follows that $[X]_S = [U]_S + [Z]_S = \sum_{i \in I} [X_i]_S$.
\end{proof}

A stratification $\{ S_i \}_{i \in I}$ of $S$ induces a decomposition of the Grothendieck ring $\K(\Var/S)$. This is a decomposition as $\K(\Var_k)$-algebras.


\begin{proposition}
    Let $S$ be a variety with stratification $\{ S_i \}_{i \in I}$. Then
    \[ \K(\Var/S) \cong \prod_{i \in I} \K(\Var/S_i) \]
    as $\K(\Var_k)$-algebras.
\end{proposition}
\begin{proof}
    Writing $f_i : S_i \to S$ for the inclusions, the isomorphism is explicitly given by
    \[ X \mapsto (f_i^* X)_{i \in I} \quad \textup{ with inverse } \quad \sum_{i \in I} (f_i)_* X_i \mapsfrom (X_i)_{i \in I} . \]
    Note that the sum is well-defined since by Lemma \ref{lemma:stratifications_are_finite} only finitely many $S_i$ are non-empty. These maps are easily seen to be inverse to each other, and morphisms of $\K(\Var_k)$-algebras.
\end{proof}

\begin{notation}
    For any $X \in \K(\Var/S)$, we will write $X|_{S_i} \in \K(\Var/S_i)$ for the components of the image of $X$ under this isomorphism.
\end{notation}

\subsection{Special algebraic groups}

Special algebraic groups were first introduced by Serre \cite{Serre1958}.

\begin{definition}
    An algebraic group $G$ over $k$ is \emph{special} if any étale $G$-torsor is locally trivial in the Zariski topology.
    Equivalently, $G$ is special if for all $k$-schemes $X$, the natural map $H^1_\textup{Zar}(X, G) \to H^1_\textup{ét}(X, G)$ is an isomorphism.
\end{definition}

\begin{lemma}
    \label{lemma:virtual_class_special_G_torsor}
    Let $G$ be a special algebraic group. Then for every $G$-torsor of varieties $X \to S$, we have $[X]_S = [G] \cdot \textup{\bf 1}_S$ in $\K(\Var/S)$.
\end{lemma}
\begin{proof}
    Since $G$ is special, the $G$-torsor $X \to S$ is Zariski-locally trivial, so there exists a stratification $\{ S_i \}_{i \in I}$ of $S$ such that $X \times_S S_i \cong G \times S_i$. Now, $[X]_S = \sum_{i \in I} [G \times S_i]_S = [G] \cdot \sum_{i \in I} [S_i]_S = [G] \cdot \textbf{1}_S$.
\end{proof}

\begin{proposition}
    \label{prop:extension_special_groups}
    Let $1 \to N \to G \to H \to 1$ be an exact sequence of algebraic groups. If $N$ and $H$ are special, then so is $G$.
\end{proposition}
\begin{proof}
    Since $G$ being special is equivalent to $H^1_\textup{ét}(X, G) = H^1_\textup{Zar}(X, G)$ for all $k$-schemes $X$, the result follows from long exact sequences in cohomology, and the five lemma.
    
    Alternatively, any $G$-torsor $X \to S$ can be written as the composition of the $N$-torsor $X \to X/N$ and the $H$-torsor $X/N \to X/G \cong S$. As $H$ is special, there exist opens $S_i \subset S$ such that $(X/N) \times_S S_i \cong H \times S_i$. Pulling back the $N$-torsor $X \times_S S_i \to H \times S_i$ along $S_i \xrightarrow{(1, \id)} H \times S_i$ gives an $N$-torsor $Y_i \to S_i$, which is also Zariski-locally trivial as $N$ is special. Hence, there exist opens $S_{ij} \subset S_i$ such that $Y_i \times_{S_i} S_{ij} \cong N \times S_{ij}$. There is now a natural morphism $G \times S_{ij} \to X \times_S S_{ij}$ of $G$-torsors over $S_{ij}$, which must be an isomorphism. Therefore, $X \to S$ is Zariski-locally trivial.
\end{proof}

\begin{example}
    \begin{itemize}
        \item By Hilbert's Theorem 90, the general linear groups $\GL_n$ are special \cite[Proposition III.4.9, Lemma III.4.10]{Milne1980}.
        \item From the exact sequence $0 \to \SL_n \to \GL_n \xrightarrow{\det} \GG_m \to 1$ follows by Proposition \ref{prop:extension_special_groups} that $\SL_n$ is also special.
        \item The group $\ZZ/2\ZZ$ is not special. For example, for $\operatorname{char}(k) \ne 2$, the $\ZZ/2\ZZ$-torsor $\AA^1_k \setminus \{ 0 \} \to \AA^1_k \setminus \{ 0 \}$ given by $x \mapsto x^2$ is not locally trivial in the Zariski topology.
    \end{itemize}
\end{example}

\begin{lemma}
    \label{lemma:unipotent_groups_are_special}
    Let $k$ be an algebraically closed field. Any unipotent group $U \subset \UU_n$ over $k$ is special.
\end{lemma}
\begin{proof}
    By \cite[Theorem 17.24]{Milne2015}, $U$ can be realized as an extension of copies of $\GG_a$. From \cite[Proposition III.3.7]{Milne1980} follows that $\GG_a$ is special, so the result follows from Proposition \ref{prop:extension_special_groups}.
\end{proof}


\begin{lemma}[Motivic Orbit-Stabilizer Theorem]
    \label{lemma:motivic_orbit_stabilizer}
    Let $G$ be an algebraic group acting on a variety $X$ over a field $k$. If $\xi \in X$ is a point such that $\Stab(\xi)$ is special, then
    \[ [G] = [\Stab(\xi)] \cdot [\Orbit(\xi)] . \]
\end{lemma}
\begin{proof}
    The map $G \to \Orbit(\xi)$ given by $g \mapsto g \xi g^{-1}$ is a $\Stab(\xi)$-torsor, hence Zariski-locally trivial, so the result follows from Lemma \ref{lemma:virtual_class_special_G_torsor}.
\end{proof}

\begin{example}
    The above lemma does not hold without the condition that $\Stab(\xi)$ is special. Namely, for $\operatorname{char}(k) \ne 2$, consider $G = \GG_m$ acting on $X = \AA^1_k$ via $t \cdot x = t^2 x$. Then, $\Orbit(1) = \AA^1_k \setminus \{ 0 \}$ and $\Stab(1) = \ZZ/2\ZZ$, but $[G] \ne 2 (q - 1)$.
\end{example}



\subsection{Algebraic representatives}
\label{subsec:algebraic_representatives}

\begin{definition}
    Let $G$ be an algebraic group over $k$, and let $X$ be a variety with a transitive $G$-action. A point $\xi \in X$ is an \emph{algebraic representative} for $X$ if, Zariski-locally on $X$, there exists a morphism of varieties $\gamma : X \to G$ such that $x = \gamma(x) \cdot \xi$ for all $x \in X$.
    Equivalently, $\xi \in X$ is an algebraic representative if the $\Stab(\xi)$-torsor 
    \[ P_\xi = \left\{ (x, g) \in X \times G \mid x = g \cdot \xi \right\} \xrightarrow{\pi_X} X \]
    is Zariski-locally trivial.
\end{definition}

\begin{remark}
    If there exists an algebraic representative $\xi \in X$, then all $\xi' \in X$ are algebraic representatives, with $\gamma'$ equal to $\gamma$ post-composed with some translation on $G$.
\end{remark}

\begin{example}
    Algebraic representatives need not always exist. Suppose $\operatorname{char}(k) \ne 2$ and consider $G = \GG_m$ acting on $X = \AA^1_k \setminus \{ 0 \}$ by $t \cdot x = t^2 x$. Then $X$ does not have an algebraic representative. Namely, without loss of generality we may take $\xi = 1$, but then the $\ZZ/2\ZZ$-torsor
    \[ P_\xi = \{ (x, t) \in X \times G \mid x = t^2 \} \xrightarrow{\pi_X} X \]
    is non-trivial.
\end{example}

\begin{example}
    Suppose $\operatorname{char}(k) \ne 2$ and consider the element $A = \smatrix{-1 & 0 \\ 0 & 1}$ in $\PGL_2$, whose stabilizer $\Big\{ \smatrix{x & 0 \\ 0 & 1} \Big\} \sqcup \Big\{\smatrix{0 & x \\ 1 & 0}\Big\} \cong \GG_m \rtimes \ZZ/2\ZZ$ is not special. Namely, the conjugacy class of $A$ consists of all elements with trace $0$, and since $[\PGL_2] = q^3 - q$ is not divisible by $[\{ A \in \PGL_2 \mid \tr(A) = 0 \}] = q^2$, the stabilizer of $A$ is not special by Lemma \ref{lemma:motivic_orbit_stabilizer}.
\end{example}

\begin{definition}
    Let $G$ be an algebraic group over $k$, and let $X$ be a variety with a $G$-action. A \emph{family of algebraic representatives} for $X$ is a morphism of varieties $\rho : X \to T$ with a section $\xi : T \to X$ such that, Zariski-locally on $X$, there exists a morphism of varieties $\gamma : X \to G$ such that $x = \gamma(x) \cdot \xi(\rho(x))$ for all $x \in X$.
\end{definition}

\begin{example}
    Consider $G = \TT_2$ acting on $X = \Big\{ \smatrix{a & b \\ 0 & 1} \mid a \ne 0, 1 \Big\}$ by conjugation. Then $X$ has a family of representatives given by $T = \Big\{ \smatrix{a & 0 \\ 0 & 1} \mid a \ne 0, 1 \Big\}$, with $\xi$ and $\rho$ the inclusion and projection, and
    \[ \gamma \begin{pmatrix} a & b \\ 0 & 1 \end{pmatrix} = \begin{pmatrix} 1 & \frac{b}{1 - a} \\ 0 & 1 \end{pmatrix} . \]
\end{example}

\begin{proposition}
    \label{prop:conjugacy_class_has_representative}
    Let $k$ be an algebraically closed field. Consider $G = \TT_n$ for some $n \ge 1$, acting by conjugation on a unipotent conjugacy class $\mathcal{U} \subset \TT_n$. Then $\mathcal{U}$ has an algebraic representative.
\end{proposition}
\begin{proof}
    Pick any point $\xi \in \mathcal{U}$. We need show that the $\Stab(\xi)$-torsor
    \[ P = \{ (u, g) \in \mathcal{U} \times G : u = g \xi g^{-1} \} \xrightarrow{\pi_\mathcal{U}} \mathcal{U} \]
    is locally trivial in the Zariski topology. In particular, it suffices to show that $\Stab(\xi)$ is a special group. Since $\TT_n$ is triangularizable, the subgroup $\Stab(\xi)$ is so as well, hence we can write
    \[ 1 \to U \to \Stab(\xi) \to D \to 1 , \]
    where $U$ is the maximal normal unipotent subgroup, which is special by Lemma \ref{lemma:unipotent_groups_are_special}. Furthermore, there is a natural splitting $D \to \Stab(\xi)$, and any $d \in D$ must satisfy $d_{ii} = d_{jj}$ whenever $\xi_{ij} \ne 0$. Hence, $D$ is a product of copies of $\GG_m$, which is also special. Now the result follows from Proposition \ref{prop:extension_special_groups}.
\end{proof}

The following lemma shows why it is useful to have algebraic representatives.

\begin{proposition}
    \label{prop:fiber_product_over_family_representatives}
    Let $S$ be a variety with a $G$-action and a family of algebraic representatives $\xi : T \to S$ and $\rho : S \to T$. Then for any morphism $f : Y \to S$ and $G$-equivariant morphism $g : X \to S$, we have
    \[ [X \times_S Y]_S = [(X \times_S T) \times_T Y]_S \in \K(\Var/S) , \]
    where $(X \times_S T) \times_T Y$ is seen as a variety over $S$ via the composition $f \circ \pi_Y$.
\end{proposition}
\begin{proof}
    Zariski-locally on $X$, there is a commutative diagram
    \[ \begin{tikzcd}
        X \times_S Y \arrow[shift left=0.25em]{rr}{\varphi} \arrow{dr} & & (X \times_S T) \times_T Y \arrow[shift left=0.25em]{ll}{\psi} \arrow{dl}{f \circ \pi_Y} \\
        & S &
    \end{tikzcd} \]
    where $\varphi(x, y) = (\gamma(f(y)) \cdot x, \rho(f(y)), y)$ and $\psi(x, t, y) = (\gamma(f(y)) \cdot x, y)$. Note that $\varphi$ and $\psi$ are easily seen to be well-defined over $S$ and inverse to each other.
\end{proof}

In the case of algebraic representatives, i.e. when $T$ is a point, we obtain the following corollaries.

\begin{corollary}
    \label{cor:fiber_product_over_representative}
    Let $S$ be a variety with a $G$-action and an algebraic representative $\xi \in S$. Then for any morphism $f : Y \to S$ and $G$-equivariant morphism $g : X \to S$, we have
    \[ [X \times_S Y]_S = [X|_\xi] \cdot [Y]_S \in \K(\Var/S) . \tag*{\qed} \]
\end{corollary}

\begin{corollary}
    \label{cor:fibration_over_representative}
    Let $S$ be a variety with $G$-action and algebraic representative $\xi \in S$. Then for any $G$-equivariant map $f : X \to S$, we have
    \[ [X]_S = \frac{[X]}{[S]} \cdot \textup{\textbf{1}}_S \in \K(\Var/S) . \]
\end{corollary}
\begin{proof}
    Apply the previous corollary with $f = \id_S$ to find that
    \[ [X]_S = [X|_\xi] \cdot \textbf{1}_S . \]
    Looking at this equality in $\K(\Var_k)$, we see that $[X|_\xi] = [X] / [S]$.
\end{proof}

\subsection[Algorithmic computations in K(Var\_k)]{Algorithmic computations in $\K(\Var_k)$}

Let $A = \{ a_1, \ldots, a_n \}$ be a finite set, and let $F$ and $G$ be finite subsets of $k[A]$. Then we write
\[ X(A, F, G) \]
for the reduced locally closed variety of $\AA^n_k$ given by $f = 0$ for all $f \in F$ and $g \ne 0$ for all $g \in G$.

In this subsection we will describe strategies for computing the virtual class of varieties of this form in $\K(\Var_k)$. These strategies are combined into a recursive algorithm, Algorithm \ref{alg:virtual_classes}, which will be used in Section \ref{sec:tqft_method} to compute the virtual class of the representation variety $R_G(\Sigma_g)$. We remark already that the algorithm will not be a general recipe for computing virtual classes, since it is allowed to fail. In fact, whenever the algorithm does not fail, it will return the virtual class of $X(A, F, G)$ as a polynomial in $q = [\AA^1_k]$, and it is clear that not all virtual classes are of this form. However, it turns out that the algorithm is sufficiently general for the purposes of this paper.

The same algorithm was used in \cite[Appendix A]{HablicsekVogel2020} to compute the virtual class of $R_{\TT_n}(\Sigma_g)$ for $1 \le n \le 4$. An implementation of this algorithm can be found at \cite{GitHubMathCode}.

\begin{notation}
    For $a \in A$, $f \in k[A]$ and $u \in k[A \setminus \{ a \}]$, we write $\textup{ev}_a(f, u)$ for the evaluation of $f$ in $a = u$. When $u, v \in k[A \setminus \{ a \}]$, we denote by $\textup{ev}_a(f, u / v)$ the evaluation of $f$ in $a = u/v$ multiplied by $v^{\deg_a(f)}$, so that $\textup{ev}_a(f, u / v) \in k[A \setminus \{ a \}]$. For subsets $F \subset k[A]$, we write $\textup{ev}_a(F, -) = \{ \textup{ev}_a(f, -) : f \in F \}$.
\end{notation}

\begin{algorithm}
    \label{alg:virtual_classes}
     \textbf{Input}: Finite sets $A, F$ and $G$ as above.
     
     \textbf{Output}: The virtual class of $X(A, F, G)$ as a polynomial in $q = [\AA^1_k]$.
    \begin{enumerate}
        \item If $F$ contains a non-zero constant or if $0 \in G$, then $X = \varnothing$, so return $[X] = 0$.
        
        \item If $F = G = \varnothing$ or $A = \varnothing$, then $X = \AA_k^{\# A}$, so return $[X] = q^{\# A}$.
        
        \item If all $f \in F$ and $g \in G$ are independent of some $a \in A$, then $X \cong \AA^1_k \times X'$ with $X' = X(A \setminus \{ a \}, F, G)$, so return $[X] = q [X']$.
        
        \item If $f = u^n$ (with $n > 1$) for some $f \in F$ and $u \in k[A]$, then we can replace $f$ with $u$, not changing $X$. That is, $X = X(A, (F \setminus \{ f \}) \cup \{ u \}, G)$. Similarly, if $g = u^n$ (with $n > 1$) for some $g \in G$ and $u \in k[A]$, then $X = X(S, F, (G \setminus \{ g \}) \cup \{ u \})$.
    
        \item If some $f \in F$ is univariate in $a \in A$, and factors as $f = (a - \alpha_1) \cdots (a - \alpha_m)$ for some $\alpha_i \in k$, then return $[X] = \sum_{i = 1}^{m} [X_i]$ with
        \[ X_i = X(A \setminus \{ a \}, \text{ev}_a(F \setminus \{ f \}, \alpha_i), \text{ev}_a(G, \alpha_i)) . \]
        
        \item Suppose $f = u v$ for some $f \in F$ and non-constant $u, v \in k[A]$. Then $X$ can be stratified by $X \cap \{ u = 0 \}$ and $X \cap \{ u \ne 0, v = 0 \}$, so return $[X] = [X_1] + [X_2]$ with
        \begin{align*}
            X_1 &= X(A, (F \setminus \{ f \}) \cup \{ u \}, G) , \\
            X_2 &= X(A, (F \setminus \{ f \}) \cup \{ v \}, G \cup \{ u \}) .
        \end{align*}
        
        \item Suppose $f = a u + v$ for some $f \in F$, $a \in A$ and $u, v \in k[A]$ with $u$ and $v$ independent of $a$ and $u$ non-zero. Then $X$ can be stratified by $X \cap \{ u = 0, v = 0 \}$ and $X \cap \{ u \ne 0, a = - v / u \}$, so return $[X] = [X_1] + [X_2]$ with
        \begin{align*}
            X_1 &= X(A, (F \setminus \{ f \}) \cup \{ u, v \}, G) , \\
            X_2 &= X(A, \textup{ev}_a(F \setminus \{ f \}, - v / u), \textup{ev}_a(G, - v / u) \cup \{ u \}) .
        \end{align*}
        
        \item If $\operatorname{char}(k) \ne 2$, suppose $f = a^2 u + a v + w$ for some $f \in F$, $a \in S$ and $u, v, w \in k[A]$ with $u, v$ and $w$ independent of $a$, and $u$ non-zero. Moreover, suppose that the discriminant $D = v^2 - 4uw$ is a square, that is, $D = d^2$ for some $d \in k[A]$. Then return $[X] = [X_1] + [X_2] + [X_3] + [X_4]$, with
        \begin{align*}
            X_1 &= X(A, (F \setminus \{ f \}) \cup \{ u, a v + w \}, G) , \\
            X_2 &= X(A, \text{ev}_a(F \setminus \{ f \}, -v / 2u) \cup \{ D \}, \text{ev}_a(G, -v / 2u) \cup \{ u \}) , \\
            X_3 &= X(A, \text{ev}_a(F \setminus \{ f \}, (- v - d) / 2u), \text{ev}_a(G, (- v - d) / 2u) \cup \{ u, D \}) , \\
            X_4 &= X(A, \text{ev}_a(F \setminus \{ f \}, (- v + d) / 2u), \text{ev}_a(G, (- v + d) / 2u) \cup \{ u, D \}) .
        \end{align*}
        
        \item If $G \ne \varnothing$, pick any $g \in G$, and return $[X] = [X_1] - [X_2]$ with
        \begin{align*}
            X_1 &= X(A, F, G \setminus \{ g \}) , \\
            X_2 &= X(A, F \cup \{ g \}, G) .
        \end{align*}
        
        \item Fail.
    \end{enumerate}
\end{algorithm}


%% file: tqft_method.tex
\section{TQFT method}
\label{sec:tqft_method}

The idea of the \emph{TQFT method} is to define a Topological Quantum Field Theory (TQFT), that is, a lax monoidal functor $Z : \textbf{Bord}_n \to \K(\Var_k)$-$\Mod$ from the category of (pointed) $n$-bordisms to the category of modules over $\K(\Var_k)$, such that the invariant $Z(M)(1) \in \K(\Var_k)$ associated to a closed manifold $X$ is precisely the class $[R_G(X)]$. Then, functoriality allows for the closed manifold $X$ to be seen as a composition of simpler bordisms, so that computing the image under $Z$ of the simpler bordisms gives the invariant of $X$.

For the purpose of this paper, it suffices to take $n = 2$ and define maps $Z(\bdgenus), Z(\bdcupleft)$ and $Z(\bdcupright)$ that mimic the TQFT, rather than giving the full description of the TQFT. Although these maps do not define an actual TQFT, they compute the same invariant (Theorem \ref{thm:tqft_method}) and have the advantage of being easier to define and to deal with.
For an elaborate discussion on the TQFT method, we refer to \cite{GonzalezPrieto2020,LogaresMunozNewstead2013}.

\begin{definition}
    \label{def:tqft_maps}
    Let $G$ be an algebraic group over $k$. Define the following $\K(\Var_k)$-module morphisms,
    \[ \arraycolsep=2pt
       \begin{array}{rlcl}
        Z(\bdcupleft): & \K(\Var_k) \to \K(\Var/G), &\quad 1 &\mapsto \left[\begin{tikzcd} \{ 1 \} \arrow{d} \\ G \end{tikzcd}\right] , \\
        Z(\bdcupright): & \K(\Var/G) \to \K(\Var_k), &\quad \left[\begin{tikzcd} X \arrow{d}{f} \\ G \end{tikzcd}\right] &\mapsto \Big[ f^{-1}(1) \Big] , \\
        Z(\bdgenus): & \K(\Var/G) \to \K(\Var/G), &\quad \left[\begin{tikzcd} X \arrow{d}{f} \\ G \end{tikzcd}\right] &\mapsto \left[\begin{tikzcd}[column sep=0em] X \times G^2 \arrow{d} & (x, A, B) \arrow[mapsto]{d} \\ G & f(x) [A, B] \end{tikzcd}\right] .
    \end{array} \]
\end{definition}

\begin{theorem}
    \label{thm:tqft_method}
    For any algebraic group $G$ and $g \ge 0$, we have
    \[ [R_G(\Sigma_g)] = Z(\bdcupright) \circ Z(\bdgenus)^g \circ Z(\bdcupleft)(1) . \]
\end{theorem}
\begin{proof}
    This follows directly from the explicit form of the $G$-representation variety of $\Sigma_g$ \eqref{eq:explicit_representation_variety_surface}.
\end{proof}

This theorem shows that, in order to compute $[R_G(\Sigma_g)]$, we must understand the map $Z(\bdgenus)$.
In the following subsections, we will take $G = \TT_n$ with $1 \le n \le 5$ and describe how to compute the matrix associated to $Z(\bdgenus)$ when restricted to the $\K(\Var_k)$-submodule of $\K(\Var/G)$ generated by the unipotent conjugacy classes.
To be precise, we will take $G$ equal to
\[ \tilde{\TT}_n = \left\{ \begin{pmatrix}
    a_{1, 1} & a_{1, 2} & \cdots & a_{1, n - 1} & a_{1, n} \\
    0 & a_{2, 2} & \cdots & a_{2, n - 1} & a_{2, n} \\
    \vdots & \vdots & \ddots & \vdots & \vdots \\
    0 & 0 & \cdots & a_{n - 1, n - 1} & a_{n - 1, n} \\
    0 & 0 & \cdots & 0 & 1
\end{pmatrix} \mid a_{i,i} \ne 0 \right\} , \]
and using the isomorphism $\TT_n \cong \GG_m \times \tilde{\TT}_n$, it follows that
\begin{equation}
    \label{eq:Tn_from_tilde_Tn}
    [R_{\TT_n}(\Sigma_g)] = [R_{\GG_m}(\Sigma_g)] [R_{\tilde{\TT}_n}(\Sigma_g)] = (q - 1)^{2g} [R_{\tilde{\TT}_n}(\Sigma_g)] .
\end{equation}
The reason to focus on $\tilde{\TT}_n$ rather than $\TT_n$ is to simplify the equations a little bit by reducing the number of variables.

In Subsection \ref{subsec:tqft_T2} we will give an example with $G = \tilde{\TT}_2$ in order to demonstrate the method. For higher $n$, the process will be analogous. In particular, we will work with a $\K(\Var_k)$-submodule of $\K(\Var/G)$ generated by the unipotent conjugacy classes of $G$. Then the matrix representing $Z(\bdgenus)$ is computed with respect to these generators.

To accomplish this, we need to determine the unipotent conjugacy classes of $G$, equations describing them, and an algebraic representatives for each one. This data can be computed automatically, and is done in Subsection \ref{subsec:tqft_conjugacy_classes}.

In Subsection \ref{subsec:tqft_computing_Z} we will compute the matrix representing $Z(\bdgenus)$. The final results will be presented in Subsection \ref{subsec:tqft_results}.

\subsection[Example G = T\_2]{Example $G = \tilde{\TT}_2$}
\label{subsec:tqft_T2}

In order to demonstrate the method, we will compute the TQFT, that is, the maps in Definition \ref{def:tqft_maps}, for the group $G = \tilde{\TT}_2$ by hand. This group is also known as the group of affine transformations of the line
\[ \textup{AGL}_1 = \left\{ \begin{pmatrix} a & b \\ 0 & 1 \end{pmatrix} \mid a \ne 0 \right\} . \]
In particular, the corresponding representation variety $R_{\textup{AGL}_1}(\Sigma_g)$ parametrizes flat rank one affine bundles over $\Sigma_g$.

As the computations below will show, the map $Z(\bdgenus)$ restricts to an endomorphism of the $\K(\Var_k)$-submodule of $\K(\Var/G)$ generated by the classes of
\[ \mathcal{I} = \{ 1 \} \to G \quad \textup{ and } \quad \mathcal{J} = \left\{ \begin{pmatrix} 1 & x \\ 0 & 1 \end{pmatrix} \mid x \ne 0 \right\} \to G . \]
Therefore, we will represent the maps as matrices with respect to these generators. In particular, it is clear that
\[ Z(\bdcupleft) = \begin{pmatrix} 1 \\ 0 \end{pmatrix} \quad \textup{ and } \quad Z(\bdcupright) = \begin{pmatrix} 1 & 0 \end{pmatrix} . \]

\begin{proposition}
    The map $Z(\bdgenus)$ is, with respect to the generators $\mathcal{I}$ and $\mathcal{J}$, represented by the matrix
    \[ Z(\bdgenus) = q^2 \begin{pmatrix}
        q - 1 & (q - 2)(q - 1) \\
        q - 2 & q^2 - 3q + 3
    \end{pmatrix} . \]
\end{proposition}
\begin{proof}
    First, let us compute $Z(\bdgenus)(\mathcal{I})|_\mathcal{I} = \left[\left\{ (A, B) \in \tilde{\TT}_2^2 \mid [A, B] = 1 \right\}\right]$. Writing $A = \begin{pmatrix} a & b \\ 0 & 1 \end{pmatrix}$ and $B = \begin{pmatrix} x & y \\ 0 & 1 \end{pmatrix}$, the equation $[A, B] = 1$ translates to $a y - b x + b - y = 0$. We distinguish the following cases:
    \begin{itemize}
        \item If $a = 1$, then $b (x - 1) = 0$, so either $b = 0$ or $x = 1$. Hence, the virtual class of this stratum is
        \[ \underset{(y)}{q} (\underset{(b = 0)}{(q - 1)} + \underset{(x = 1)}{q} - \underset{\left(\substack{b = 0 \\ x = 1}\right)}{1}) = 2q (q - 1) . \]
        \item If $a \ne 1$, then $y = b (x - 1) / (a - 1)$, yielding a virtual class of $q (q - 2) (q - 1)$.
    \end{itemize}
    Together, $Z(\bdgenus)(\mathcal{I})|_\mathcal{I} = q^2 (q - 1) \mathcal{I}$.
    Next, note that
    \[ c_! Z(\bdgenus)(\mathcal{I})|_\mathcal{I} + [\mathcal{J}] c_! Z(\bdgenus)(\mathcal{I})|_\mathcal{J} = [\tilde{\TT}_2]^2 = q^2 (q - 1)^2 , \]
    from which follows that $Z(\bdgenus)(\mathcal{I})|_\mathcal{J} = q^2 (q - 2) \mathcal{J}$, using Corollary \ref{cor:fibration_over_representative} and the fact that $\mathcal{J}$ has an algebraic representative.
    
    The next coefficient can be computed from the previous one as 
    \[ Z(\bdgenus)(\mathcal{J})|_\mathcal{I} = ([\mathcal{J}] c_! Z(\bdgenus)(\mathcal{I})|_\mathcal{J}) \mathcal{I} = q^2 (q - 2) (q - 1) \mathcal{I} . \]
    Finally, for the last coefficient, note that
    \[ c_! Z(\bdgenus)(\mathcal{J})|_\mathcal{I} + [\mathcal{J}] c_! Z(\bdgenus)(\mathcal{J})|_\mathcal{J} = [\mathcal{J}] [\tilde{\TT}_2]^2 = q^2 (q - 1)^3 , \]
    from which follows that $Z(\bdgenus)(\mathcal{J})|_{\mathcal{J}} = q^2 (q^2 - 3q + 3) \mathcal{J}$.
\end{proof}
In order to apply Theorem \ref{thm:tqft_method}, it is convenient to diagonalize the map $Z(\bdgenus)$ as $PDP^{-1}$ with
\[ P = \begin{pmatrix} 1 - q & 1 \\ 1 & 1 \end{pmatrix} \quad \textup{ and } \quad D = \begin{pmatrix} q^2 & 0 \\ 0 & q^2 (q - 1)^2 \end{pmatrix} . \]
Now from Theorem \ref{thm:tqft_method}, and straightforward matrix multiplication, it follows that
\[ [R_{\tilde{\TT}_2}(\Sigma_g)] = q^{2g - 1} (q - 1)^{2g} + q^{2g - 1} (q - 1) . \]



\subsection{Describing conjugacy classes}
\label{subsec:tqft_conjugacy_classes}

The aim of this subsection is to describe the conjugacy classes of $\tilde{\TT}_n$, with a focus on the unipotent conjugacy classes. In particular, we will find algebraic representatives for the unipotent conjugacy classes, and families of algebraic representatives for the non-unipotent conjugacy classes. Furthermore, we will find equations describing the unipotent conjugacy classes, and finally, compute the virtual class of the unipotent conjugacy classes as well as that of the stabilizers of their representatives. All of this data is needed in the later subsections in order to compute the matrix of $Z(\bdgenus)$ with respect to the generators given by the unipotent classes.


\textbf{Representatives}.
In order to find the number of unipotent conjugacy classes of $\tilde{\TT}_n$, and representatives for each one, one can use Belitskii's algorithm as described in \cite{Bhunia2019}, which works for $1 \le n \le 5$. Note that the representatives will be automatically algebraic by Proposition \ref{prop:conjugacy_class_has_representative}. Denote the resulting unipotent conjugacy classes by $\mathcal{U}_1, \ldots, \mathcal{U}_M$ and the respective algebraic representatives by $\xi_1, \ldots, \xi_M$. The number $M$ of unipotent conjugacy classes, depending on $n$, is given by the following table.
\[ \begin{tabular}{c|c|c|c|c|c}
     $n$ & 1 & 2 & 3 & 4 & 5 \\ \hline
     $M$ & 1 & 2 & 5 & 16 & 61
\end{tabular} \]
Alternatively, one can use the qualitative result of Belitskii's algorithm that for $1 \le n \le 5$ every unipotent matrix can be conjugated in $\tilde{\TT}_n$ to a matrix containing only $0$'s and $1$'s. This gives a finite number of matrices ($2^{n (n - 1) / 2}$) which can easily be partitioned based on whether or not they are conjugate. Now take one representative out of each conjugacy class.

Next, we describe the non-unipotent conjugacy classes of $\tilde{\TT}_n$ in terms of families depending on their diagonal. Define a \emph{diagonal pattern} to be a partition of the set $\{ 1, 2, \ldots, n \}$. Then, for any matrix $A \in \tilde{\TT}_n$, the diagonal pattern $\delta_A$ of $A$ is the partition such that $i$ and $j$ are equivalent if $A_{ii} = A_{jj}$. 
Note that two matrices $A$ and $B$ in $\tilde{\TT}_n$ are conjugate only if their diagonals coincide, but not necessarily if.
Now, we look at the following families of conjugacy classes:
\[ \mathcal{C}_{\delta, i} = \{ A \in \tilde{\TT}_n \mid \delta_A = \delta \textup{ and } A \sim{} \textup{diag}(A) + \xi_i - 1 \} , \]
for any diagonal pattern $\delta$ and $1 \le i \le M$, where $\textup{diag}(A)$ denotes the diagonal part of $A$. Any such $\mathcal{C}_{\delta, i}$ has a family of representatives
\[ \xi_{\delta, i} : C_{\delta, i} = \{ A \in \tilde{\TT}_n \mid \delta_A = \delta \textup{ and } A = \textup{diag}(A) + \xi_i - 1 \} \to \mathcal{C}_{\delta, i} \]
given by inclusion.
Of course, for some diagonal matrices $D$ it may happen that $D + \xi_i - 1$ and $D + \xi_j - 1$ are conjugate while $i \ne j$, but such duplicates are easily filtered out.
After this filtering, we obtain families of conjugacy classes $\mathcal{C}_1, \ldots, \mathcal{C}_N$, where the number $N$ is given, depending on $n$, by the following table.
\[ \begin{tabular}{c|c|c|c|c|c}
     $n$ & 1 & 2 & 3 & 4 & 5 \\ \hline
     $N$ & 2 & 3 & 12 & 61 & 372
\end{tabular} \]
Note that we can choose our indices in such a way that the $\xi_i$ coincide with the unipotent representatives for $1 \le i \le M$.

\textbf{Equations}. Next, we want to find equations describing the unipotent conjugacy classes $\mathcal{U}_i$ for $1 \le i \le M$. For simplicity we will compute equations for the closure $\overline{\mathcal{U}}_i$ rather than $\mathcal{U}_i$. In Subsection \ref{subsec:tqft_transition_matrix} it will be shown that this is sufficient. The closure $\overline{\mathcal{U}}_i$ is the closure of the image of the map
\[ f_i : G \to G, \quad g \mapsto g \xi_i g^{-1} . \]
Since $G$ is affine, $f_i$ can equivalently be described by a morphism of rings
\[ f_i^\# : \mathcal{O}_G(G) \to \mathcal{O}_G(G) . \]
In particular, the closure $\overline{\mathcal{U}}_i$ corresponds to the ideal $I_i \subset \mathcal{O}_G(G)$ which is the kernel of $f_i^\#$. Generators for such these ideals can be computed using Gröbner basis \cite{AdamsLoustaunau1994}, and this gives us the desired equations. In particular, we use \cite[Theorem 2.4.2]{AdamsLoustaunau1994} in order to compute the kernel of $f_i^\#$.




\textbf{Orbits and stabilizers}.
For any $A \in \tilde{\TT}_n$, in order to compute the virtual class of the orbit of $A$, we can use Lemma \ref{lemma:motivic_orbit_stabilizer}. Indeed, as we have seen in Proposition \ref{prop:conjugacy_class_has_representative}, the stabilizer $\Stab(A)$ of any $A \in \tilde{\TT}_n$ is special.
In order to compute the virtual class of the stabilizer of $A$, we use Algorithm \ref{alg:virtual_classes}, since we have an explicit description
\[ \Stab(A) = \{ B \in \tilde{\TT}_n \mid AB - BA = 0 \} . \]
Regarding the families of conjugacy classes $\mathcal{C}_i$, we will make use of the following lemma.
\begin{lemma}
    \label{lemma:splitting_conjugacy_classes}
    For every $1 \le i \le N$, there is an isomorphism
    \[ \mathcal{C}_i \cong C_i \times \Orbit(A) \]
    for every $A \in C_i$.
\end{lemma}
\begin{proof}
    Note that the fibers of $\mathcal{C}_i \to C_i$ are given by $\Orbit(A) \cong G / \Stab(A)$, so it suffices to show that $\Stab(A)$ is constant for all $A \in C_i$. As there are only a finite number of cases to consider, this can easily be verified explicitly.
    Alternatively, note that $B \in \Stab(A)$ if and only if for all $1 \le i \le j \le n$,
    \[ B_{ij} (A_{ii} - A_{jj}) + \sum_{k = i + 1}^{j} B_{kj} A_{ik} - \sum_{k = i}^{j - 1} B_{ik} A_{kj} = 0 . \tag{$*$} \]
    We claim that $B_{ij} = 0$ for all $i \le j$ such that $A_{ii} \ne A_{jj}$. The result follows from this claim, as $A_{ij} \in \{ 0, 1 \}$ for $i \ne j$, so the solutions to $(*)$ will be independent of $A$. We proof the claim by induction on $j - i$, the case $j - i = 0$ being trivial. For the general case, take $i \le j$ such that $A_{ii} \ne A_{jj}$. if there exists a $k \in \{ i + 1, \ldots, j \}$ such that $A_{ki} \ne 0$, then $A_{kk} = A_{ii} \ne A_{jj}$, so $B_{kj} = 0$. Similarly, if there exists a $k \in \{ i, \ldots, j - 1 \}$ such that $A_{kj} \ne 0$, then $A_{kk} = A_{jj} \ne A_{ii}$ so $B_{ki} = 0$. Therefore, $(*)$ reduces to $B_{ij} = 0$.
\end{proof}

\subsection{Transition matrix}
\label{subsec:tqft_transition_matrix}

Let $X$ be a variety with stratification $X_i$, and let $Y$ be a variety over $X$. The goal of this subsection is to show that, in order to compute the classes $[Y \cap X_i]_X$, it is sufficient to compute the classes $[Y \cap \overline{X}_i]_X$ instead, making use of an inclusion--exclusion principle. 
%

\begin{example}
    Suppose $X$ is stratified by $X_0 \subset X$ closed and its open complement $X_1 = X \setminus X_0$. If we want to compute $[Y \cap X_0]$ and $[Y \cap X_1]$, then computing the latter would likely result in the computation $[Y \cap X] - [Y \cap X_0]$, which shows that the result of the computation of $[Y \cap X_0]$ can be reused. Therefore, instead of computing $[Y \cap X_0]$ and $[Y \cap X_1]$, one can compute $[Y \cap \overline{X}_0] = [Y \cap X_0]$ and $[Y \cap \overline{X}_1] = [Y \cap X] = [Y]$, from which formally follows that $[Y \cap X_1] = [Y \cap \overline{X}_1] - [Y \cap \overline{X}_0]$.
\end{example}


\begin{lemma}
    Let $X$ be a topological space with stratification $\{ X_i \}_{i \in I}$. Then $\overline{X}_i = \overline{X_j}$ if and only if $i = j$.
\end{lemma}
\begin{proof}
    For each $i \in I$, write $X_i = Z_i \cap U_i$ for some closed $Z_i \subset X$ and open $U_i \subset X$. Without loss of generality, we may assume $Z_i = \overline{X}_i$. Now, if $\overline{X}_i = \overline{X}_j$ for some $i, j \in I$, then both $X_i$ and $X_j$ are open and dense in $\overline{X}_i = \overline{X}_j$, so they must intersect. But this contradicts the assumption that $X_i$ and $X_j$ are disjoint, since they are part of the stratification.
\end{proof}

\begin{definition}
    Let $X$ be a variety over $S$ with a stratification $\{ X_i \}_{i \in I}$. Put a partial ordering on $I$ where $i \le j$ if and only if $\overline{X}_i \subset \overline{X}_j$. Reflexivity and transitivity are clear, and anti-symmetry follows from the above lemma. Hence, the equalities
    \[ [X_i]_S = [\overline{X}_i]_S - \sum_{\substack{j < i}} [X_j]_S \]
    define a $\ZZ$-linear map $C : \ZZ^I \to \ZZ^I$, called the \emph{transition matrix} of the stratification, satisfying
    \[ [X_i]_S = \sum_{j \in I} C_{ij} [\overline{X}_j]_S \]
    in $\K(\Var/S)$.
\end{definition}

\begin{corollary}
    \label{cor:transition_matrix}
    Let $X$ be a variety over $S$ with finite stratification $\{ X_i \}_{i \in I}$ and corresponding transition matrix $C$. Then for any variety $Y$ over $X$, we have
    \[ [Y \cap X_i]_S = \sum_{j \in I} C_{ij} \ [Y \cap \overline{X}_j]_S , \]
    where it is understood that $Y \cap X_i = Y \times_X X_i$. \qed
\end{corollary}

\begin{example}
    \label{ex:transition_matrix_unipotent}
    Let $G = \tilde{\TT}_n$ and let $U \subset G$ be the closed subvariety of unipotent matrices. Then $U$ has a stratification given by the $\mathcal{U}_1, \ldots, \mathcal{U}_M$. Since the equations corresponding to every $\overline{\mathcal{U}}_i$ can be computed, the partial ordering on $\{ 1, 2, \ldots, M \}$ can easily be decided, and hence the transition matrix $C_{ij}$ of this stratification can be computed. Now, for any variety $Y$ over $U$, we have by the above corollary that
    \[ [Y \cap \mathcal{U}_i]_G = \sum_{j = 1}^{M} C_{ij} \ [Y \cap \overline{\mathcal{U}}_j]_G . \]
    In particular, it suffices to compute all $[Y \cap \overline{\mathcal{U}}_i]_G$ in order to compute the $[Y \cap \mathcal{U}_i]_G$.
\end{example}

\subsection[Computing Z]{Computing $Z(\bdgenus)$}
\label{subsec:tqft_computing_Z}

Let us return to the problem of computing the matrix associated to $Z(\bdgenus)$ with respect to the generators given by the unipotent conjugacy classes. 

In order to compute the first column of this matrix, we can write
\begin{align*}
    c_! Z(\bdgenus)(\textbf{1}_\id) |_{\mathcal{U}_i}
        &= [\{ (A, B) \in G^2 \mid [A, B] \in \mathcal{U}_i \}] \\
        &= \sum_{j = 1}^{N} [\{ (A, B) \in G \times \mathcal{C}_j \mid [A, B] \in \mathcal{U}_i \}] \\
        &= \sum_{j = 1}^{N} [\{ (A, t) \in G \times C_j \mid [A, \xi_j(t)] \in \mathcal{U}_i \}] \times [\Orbit(\xi_j(\star))] \\
        &= \sum_{j = 1}^{N} E_{ij} \, [ \Orbit(\xi_j(\star)) ],
\end{align*}
where $E_{ij} = [\{ (A, t) \in G \times C_j \mid [A, \xi_j(t)] \in \mathcal{U}_i \}]$ and $\star$ is any point in $C_j$. For the third equality, we used Proposition \ref{prop:fiber_product_over_family_representatives} with $Y = S = \mathcal{C}_j$ and $X = \{ (A, B) \in G \times \mathcal{C}_j \mid [A, B] \in \mathcal{U}_i \}$, in combination with Lemma \ref{lemma:splitting_conjugacy_classes}.

Since we have computed equations describing the closures $\overline{\mathcal{U}}_i$, it is in fact easier to compute the classes $\overline{E}_{ij} = [\{ (A, t) \in G \times C_j : [A, \xi_j(t)] \in \overline{\mathcal{U}}_i \}]$ rather than the $E_{ij}$. By Corollary \ref{cor:transition_matrix}, they are related by the transition matrix of the stratification
\[ E_{ij} = \sum_{k = 1}^{M} C_{ik} \overline{E}_{kj} . \]
The coefficients $\overline{E}_{ij}$ can be computed using Algorithm \ref{alg:virtual_classes}.
Then, using Corollary \ref{cor:fibration_over_representative}, we obtain
\[ Z(\bdgenus)(\textbf{1}_\id) = \sum_{i, k = 1}^{M} \sum_{j = 1}^{N} C_{ik} \overline{E}_{kj} [\Orbit(\xi_j(\star))] / [\mathcal{U}_i] \cdot \textbf{1}_i . \]



Next, we can use the first column we have just computed to compute all other columns of the matrix associated to $Z(\bdgenus)$.
In particular, we can write
\begin{align*}
    c_! Z(\bdgenus)(\textbf{1}_j)|_{\mathcal{U}_i}
        &= [\{ (g, A, B) \in \mathcal{U}_j \times G^2 \mid g[A, B] \in \mathcal{U}_i \}] \\
        &= \sum_{k = 1}^{M} [\{ (g, A, B) \in \mathcal{U}_j \times G^2 \mid g [A, B] \in \mathcal{U}_i, [A, B] \in \mathcal{U}_k \}] \\
        &= \sum_{k = 1}^{M} [\{ g \in \mathcal{U}_j \mid g \xi_k \in \mathcal{U}_i \}] [\{ (A, B) \in G^2 \mid [A, B] \in \mathcal{U}_k \}] \\
        &= \sum_{k = 1}^{M} F_{ijk} \cdot c_! Z(\bdgenus)(\textbf{1}_\id)|_{\mathcal{U}_k} ,
\end{align*}
where $F_{ijk} = [\{ g \in \mathcal{U}_j \mid g \xi_k \in \mathcal{U}_i \}]$. The third equality follows from Corollary \ref{cor:fiber_product_over_representative} applied to $S = \mathcal{U}_k$ and $X = \{ (g, h) \in \mathcal{U}_j \times \mathcal{U}_k \mid g h \in \mathcal{U}_i \}$ and $Y = \{ (A, B) \in G^2 \mid [A, B] \in \mathcal{U}_k \}$.

As was the case for the coefficients $E_{ij}$, it is easier to compute $\overline{F}_{ijk} = [\{ g \in \overline{\mathcal{U}}_{j} \mid g \xi_k \in \overline{\mathcal{U}}_i \}]$ rather than $F_{ijk}$, and they are related by the transition matrix of the stratification
\[ F_{ijk} = \sum_{m, \ell = 1}^{M} C_{i m} C_{j \ell} \overline{F}_{m \ell k} . \]
The coefficients $\overline{F}_{ijk}$ can be computed using Algorithm \ref{alg:virtual_classes}. Finally, using Corollary \ref{cor:fibration_over_representative} we obtain
\[ Z(\bdgenus)(\textbf{1}_j) = \sum_{i, k, \ell, m = 1}^{M} C_{i m} C_{j \ell} \overline{F}_{m \ell k} c_! Z(\bdgenus)(\textbf{1}_\id)|_{\mathcal{U}_k} / [\mathcal{U}_i] \cdot \textbf{1}_i . \]

\begin{remark}
    If one would naively try to compute the coefficients of the matrix representing $Z(\bdgenus)$, they would need to compute the virtual class of $M^2$ varieties, each of which being a subvariety of $G^3$, with equations being mostly quadratic due to the commutator $[A, B]$. With this new setup, one needs to compute the virtual class of $MN + M^3$ varieties to obtain the coefficients $\overline{E}_{ij}$ and $\overline{F}_{ijk}$. However, the main advantage of this approach is that these varieties will now be subvarieties of $G \times C_j$ and $\mathcal{U}_j$, respectively, with equations being mostly linear! In practice, the simplification of these systems of equations far outweighs the number of such systems. It is due to the (families of) algebraic representatives of the conjugacy classes $\mathcal{C}_i$, or in other words, the lack of monodromy of $\mathcal{C}_i \to C_i$, that these simplifications can be made.
\end{remark}

\begin{remark}
    Let us make a few computational remarks. First, to speed up the computation of the coefficients $\overline{E}_{ij}$ and $\overline{F}_{ijk}$, which are done by Algorithm \ref{alg:virtual_classes}, note that these can be performed in parallel as they are independent.
    Second, there are a few checks one can perform to detect obvious errors. In particular, one can assert that the following equalities hold:
    \begin{itemize}
        \item $\sum_{i = 1}^{M} c_! Z(\bdgenus)(\textbf{1}_j)|_{\mathcal{U}_i} = [G]^2 [\mathcal{U}_j]$ for all $j$,
        \item $\sum_{i = 1}^{M} F_{ijk} = [\mathcal{U}_j]$ for all $j, k$,
        \item $\sum_{j = 1}^{M} F_{ijk} = [\mathcal{U}_i]$ for all $i, k$,
        \item $\sum_{i = 1}^{M} E_{ij} = [G] [C_j]$ for all $j$.
    \end{itemize}
\end{remark}
 
\subsection{Results}
\label{subsec:tqft_results}

The code performing the computations as described in this section can be found at \cite{GitHubMathCode}.
The results for $1 \le n \le 4$ were already given in \cite{HablicsekVogel2020}.

The matrix associated to $Z(\bdgenus)$, with respect to the generators given by the unipotent conjugacy classes, can be diagonalized in order to apply Theorem \ref{thm:tqft_method}. For $G = \tilde{\TT}_5$, the eigenvalues with multiplicity of this matrix are given by
\[ \arraycolsep=1cm\def\arraystretch{1.25} \begin{array}{lll}
    q^{12} (q - 1)^{4}, &
    q^{14} (q - 1)^{2} \quad \textup{(mult. 2)}, &
    q^{16} \quad \textup{(mult. 2)}, \\
    q^{14} (q - 1)^{4} \quad \textup{(mult. 3)}, &
    q^{16} (q - 1)^{2} \quad \textup{(mult. 7)}, &
    q^{18} \quad \textup{(mult. 2)}, \\
    q^{14} (q - 1)^{6}, &
    q^{16} (q - 1)^{4} \quad \textup{(mult. 7)}, &
    q^{18} (q - 1)^{2} \quad \textup{(mult. 7)}, \\
    q^{20}, &
    q^{16} (q - 1)^{6} \quad \textup{(mult. 2)}, &
    q^{18} (q - 1)^{4} \quad \textup{(mult. 8)}, \\
    q^{20} (q - 1)^{2} \quad \textup{(mult. 4)}, &
    q^{18} (q - 1)^{6} \quad \textup{(mult. 3)}, &
    q^{20} (q - 1)^{4} \quad \textup{(mult. 6)}, \\
    q^{20} (q - 1)^{6} \quad \textup{(mult. 4)}, &
    q^{20} (q - 1)^{8} , &
\end{array} \]
with respective eigenvectors
\[ \scalebox{0.30}{$\footnotesize \left[\begin{array}{ccccccccccccccccccccccccccc}q \left(q - 1\right)^{2} & - q^{2} \left(q - 1\right)^{3} & q \left(q - 1\right)^{3} & q \left(q - 1\right)^{4} & q \left(q - 1\right)^{4} & 0 & q^{2} \left(q - 1\right)^{2} & - q \left(q - 1\right)^{2} & 0 & - q \left(q - 1\right)^{3} & 0 & - \left(q - 1\right)^{3} & 0 & 0 & - q^{2} \left(q - 1\right)^{3} & \left(q - 1\right)^{4} & - \left(q - 1\right)^{4} & - q^{2} \left(q - 1\right) & 0 & q \left(q - 1\right)^{2} & 0 & \left(q - 1\right)^{2} & 0 & 0 & q^{2} \left(q - 1\right)^{2} & - \left(q - 1\right)^{3} & 0\\0 & 0 & 0 & 0 & 0 & 0 & 0 & 0 & - q \left(q - 1\right)^{2} & - q \left(q - 2\right) \left(q - 1\right)^{2} & \left(q - 1\right)^{2} & 0 & 0 & 0 & 0 & - \left(q - 1\right)^{3} & 0 & 0 & q \left(q - 1\right) & q \left(q - 2\right) \left(q - 1\right) & \left(q - 1\right)^{2} & 0 & 0 & 0 & 0 & \left(q - 1\right)^{2} & 0\\0 & 0 & 0 & - q \left(q - 1\right)^{3} & 0 & 0 & 0 & 0 & q^{2} \left(q - 1\right)^{2} & q \left(q - 1\right)^{4} & - q \left(q - 1\right)^{2} & \left(q - 1\right)^{2} & 0 & 0 & 0 & \left(q - 1\right)^{4} & \left(q - 1\right)^{3} & 0 & 0 & q \left(q - 1\right)^{2} & q \left(q - 1\right) & 1 - q & 0 & 0 & 0 & - \left(q - 1\right)^{3} & 0\\0 & 0 & - q \left(q - 1\right)^{2} & q \left(q - 1\right)^{4} & - q \left(q - 1\right)^{3} & 0 & 0 & q \left(q - 1\right) & 0 & - q \left(q - 1\right)^{3} & 0 & - \left(q - 1\right)^{3} & 0 & 0 & q^{2} \left(q - 1\right)^{2} & \left(q - 1\right)^{4} & - \left(q - 1\right)^{4} & 0 & 0 & q \left(q - 1\right)^{2} & 0 & \left(q - 1\right)^{2} & 0 & 0 & - q^{2} \left(q - 1\right) & - \left(q - 1\right)^{3} & 0\\- q \left(q - 1\right) & q^{2} \left(q - 1\right)^{2} & q \left(q - 1\right)^{3} & q \left(q - 1\right)^{4} & q \left(q - 1\right)^{4} & 0 & - q^{2} \left(q - 1\right) & - q \left(q - 1\right)^{2} & 0 & - q \left(q - 1\right)^{3} & 0 & - \left(q - 1\right)^{3} & 0 & 0 & - q^{2} \left(q - 1\right)^{3} & \left(q - 1\right)^{4} & - \left(q - 1\right)^{4} & q^{2} & 0 & q \left(q - 1\right)^{2} & 0 & \left(q - 1\right)^{2} & 0 & 0 & q^{2} \left(q - 1\right)^{2} & - \left(q - 1\right)^{3} & 0\\0 & q \left(q - 1\right)^{2} & 0 & 0 & - \left(q - 1\right)^{3} & q - 1 & q \left(q - 2\right) \left(q - 1\right) & 0 & 0 & 0 & 0 & 0 & \left(q - 1\right)^{2} & \left(q - 1\right)^{3} & \left(q - 1\right)^{2} & 0 & 0 & - q \left(q - 1\right) & 0 & 0 & 0 & 0 & q - 1 & 1 - q & \left(q - 1\right)^{3} & \left(q - 1\right)^{2} & - \left(q - 1\right)^{2}\\0 & 0 & 0 & 0 & 0 & 0 & 0 & 0 & 0 & 0 & 0 & 0 & 0 & 0 & 0 & 0 & 0 & 0 & 0 & 0 & 0 & 0 & 0 & 0 & 0 & 1 - q & q - 1\\0 & 0 & 0 & 0 & \left(q - 1\right)^{2} & 0 & 0 & 0 & 0 & 0 & 0 & 0 & 1 - q & - \left(q - 1\right)^{2} & 1 - q & 0 & 0 & 0 & 0 & 0 & 0 & 0 & -1 & 1 & - \left(q - 1\right)^{2} & \left(q - 1\right)^{2} & - \left(q - 1\right)^{2}\\0 & - q \left(q - 1\right) & 0 & 0 & - \left(q - 1\right)^{3} & -1 & - q \left(q - 2\right) & 0 & 0 & 0 & 0 & 0 & \left(q - 1\right)^{2} & \left(q - 1\right)^{3} & \left(q - 1\right)^{2} & 0 & 0 & q & 0 & 0 & 0 & 0 & q - 1 & 1 - q & \left(q - 1\right)^{3} & \left(q - 1\right)^{2} & - \left(q - 1\right)^{2}\\1 - q & - q \left(q - 1\right)^{3} & 0 & 0 & 0 & 0 & q \left(q - 1\right)^{2} & 0 & 0 & 0 & 0 & \left(q - 1\right)^{2} & 0 & 0 & 0 & - \left(q - 1\right)^{3} & \left(q - 1\right)^{3} & - q \left(q - 1\right) & 0 & 0 & 0 & \left(q - 1\right)^{2} & 0 & 0 & 0 & - \left(q - 1\right)^{3} & q \left(q - 1\right)^{2}\\0 & 0 & 0 & 0 & 0 & 0 & 0 & 0 & 0 & 0 & 0 & 0 & 0 & 0 & 0 & \left(q - 1\right)^{2} & 0 & 0 & 0 & 0 & 0 & 0 & 0 & 0 & 0 & \left(q - 1\right)^{2} & - q \left(q - 1\right)\\0 & 0 & 0 & 0 & 0 & 0 & 0 & 0 & 0 & 0 & 0 & 1 - q & 0 & 0 & 0 & - \left(q - 1\right)^{3} & - \left(q - 1\right)^{2} & 0 & 0 & 0 & 0 & 1 - q & 0 & 0 & 0 & - \left(q - 1\right)^{3} & q \left(q - 1\right)^{2}\\1 & q \left(q - 1\right)^{2} & 0 & 0 & 0 & 0 & - q \left(q - 1\right) & 0 & 0 & 0 & 0 & \left(q - 1\right)^{2} & 0 & 0 & 0 & - \left(q - 1\right)^{3} & \left(q - 1\right)^{3} & q & 0 & 0 & 0 & \left(q - 1\right)^{2} & 0 & 0 & 0 & - \left(q - 1\right)^{3} & q \left(q - 1\right)^{2}\\0 & 0 & - q \left(q - 1\right)^{2} & - q \left(q - 1\right)^{3} & q \left(q - 1\right)^{4} & 0 & 0 & q \left(q - 1\right) & 0 & q \left(q - 1\right)^{2} & 0 & - \left(q - 1\right)^{3} & 0 & 0 & - q^{2} \left(q - 1\right)^{3} & \left(q - 1\right)^{4} & - \left(q - 1\right)^{4} & 0 & 0 & - q \left(q - 1\right) & 0 & \left(q - 1\right)^{2} & 0 & 0 & q^{2} \left(q - 1\right)^{2} & - \left(q - 1\right)^{3} & 0\\0 & 0 & 0 & 0 & 0 & 0 & 0 & 0 & q \left(q - 1\right) & q \left(q - 2\right) \left(q - 1\right) & 1 - q & 0 & 0 & 0 & 0 & - \left(q - 1\right)^{3} & 0 & 0 & - q & - q \left(q - 2\right) & 1 - q & 0 & 0 & 0 & 0 & \left(q - 1\right)^{2} & 0\\0 & 0 & 0 & q \left(q - 1\right)^{2} & 0 & 0 & 0 & 0 & - q^{2} \left(q - 1\right) & - q \left(q - 1\right)^{3} & q \left(q - 1\right) & \left(q - 1\right)^{2} & 0 & 0 & 0 & \left(q - 1\right)^{4} & \left(q - 1\right)^{3} & 0 & 0 & - q \left(q - 1\right) & - q & 1 - q & 0 & 0 & 0 & - \left(q - 1\right)^{3} & 0\\0 & 0 & q \left(q - 1\right) & - q \left(q - 1\right)^{3} & - q \left(q - 1\right)^{3} & 0 & 0 & - q & 0 & q \left(q - 1\right)^{2} & 0 & - \left(q - 1\right)^{3} & 0 & 0 & q^{2} \left(q - 1\right)^{2} & \left(q - 1\right)^{4} & - \left(q - 1\right)^{4} & 0 & 0 & - q \left(q - 1\right) & 0 & \left(q - 1\right)^{2} & 0 & 0 & - q^{2} \left(q - 1\right) & - \left(q - 1\right)^{3} & 0\\0 & q \left(q - 1\right)^{2} & 0 & - \left(q - 1\right)^{3} & 0 & 1 - q & 0 & 0 & - \left(q - 1\right)^{3} & - q \left(q - 2\right) \left(q - 1\right)^{2} & \left(q - 1\right)^{2} & 0 & 0 & 0 & 0 & 0 & 0 & - q \left(q - 1\right) & 1 - q & 0 & 1 - q & 0 & 0 & 0 & 0 & 0 & 0\\0 & 0 & 0 & \left(q - 1\right)^{2} & 0 & 0 & 0 & 0 & q \left(q - 1\right) & \left(q - 1\right) \left(q^{2} - q - 1\right) & 1 - q & 0 & 0 & 0 & 0 & 0 & \left(q - 1\right)^{2} & 0 & 0 & 1 - q & 1 - q & 0 & 0 & 0 & 0 & 0 & - \left(q - 1\right)^{2}\\0 & 0 & 0 & \left(q - 1\right)^{2} & 0 & 0 & 0 & 0 & 1 - q & 0 & 0 & 0 & 0 & 0 & 0 & 0 & 0 & 0 & 1 - q & 0 & 0 & 0 & 0 & 0 & 0 & 0 & 0\\0 & 0 & 0 & 1 - q & 0 & 0 & 0 & 0 & 0 & 1 - q & 0 & 0 & 0 & 0 & 0 & 0 & 1 - q & 0 & 0 & 1 - q & 0 & 0 & 0 & 0 & 0 & 0 & q - 1\\0 & - q \left(q - 1\right) & 0 & - \left(q - 1\right)^{3} & 0 & 1 & 0 & 0 & - \left(q - 1\right)^{3} & - q \left(q - 2\right) \left(q - 1\right)^{2} & \left(q - 1\right)^{2} & 0 & 0 & 0 & 0 & 0 & 0 & q & 1 - q & 0 & 1 - q & 0 & 0 & 0 & 0 & 0 & 0\\0 & 1 - q & 0 & 0 & 0 & 0 & 1 - q & 0 & 0 & 0 & 0 & 0 & 0 & 0 & 0 & 0 & 0 & 1 - q & 0 & 0 & 0 & 0 & 0 & 0 & 0 & 0 & 0\\0 & 0 & 0 & 0 & 0 & 0 & 0 & 0 & 0 & 0 & 0 & 0 & 0 & 0 & 0 & 0 & 0 & 0 & 0 & 0 & 0 & 0 & 0 & 0 & 0 & 0 & 0\\0 & 1 & 0 & 0 & 0 & 0 & 1 & 0 & 0 & 0 & 0 & 0 & 0 & 0 & 0 & 0 & 0 & 1 & 0 & 0 & 0 & 0 & 0 & 0 & 0 & 0 & 0\\0 & 0 & 0 & \left(q - 1\right)^{2} & 0 & 0 & 0 & 0 & \left(q - 1\right)^{2} & q \left(q - 2\right) \left(q - 1\right) & 1 - q & 0 & 0 & 0 & 0 & 0 & 0 & 0 & 1 & 0 & 1 & 0 & 0 & 0 & 0 & 0 & 0\\0 & 0 & 0 & 1 - q & 0 & 0 & 0 & 0 & - q & - q^{2} + q + 1 & 1 & 0 & 0 & 0 & 0 & 0 & \left(q - 1\right)^{2} & 0 & 0 & 1 & 1 & 0 & 0 & 0 & 0 & 0 & - \left(q - 1\right)^{2}\\0 & 0 & 0 & 1 - q & 0 & 0 & 0 & 0 & 1 & 0 & 0 & 0 & 0 & 0 & 0 & 0 & 0 & 0 & 1 & 0 & 0 & 0 & 0 & 0 & 0 & 0 & 0\\0 & 0 & 0 & 1 & 0 & 0 & 0 & 0 & 0 & 1 & 0 & 0 & 0 & 0 & 0 & 0 & 1 - q & 0 & 0 & 1 & 0 & 0 & 0 & 0 & 0 & 0 & q - 1\\0 & 0 & 0 & 0 & - q \left(q - 1\right)^{3} & 0 & 0 & 0 & 0 & 0 & 0 & \left(q - 1\right)^{2} & 0 & - q^{2} \left(q - 1\right)^{2} & - q^{2} \left(q - 2\right) \left(q - 1\right)^{2} & - \left(q - 1\right)^{3} & - \left(q - 1\right)^{4} & 0 & 0 & 0 & 0 & 1 - q & 0 & q^{2} \left(q - 1\right) & q^{2} \left(q - 2\right) \left(q - 1\right) & \left(q - 1\right)^{2} & - q \left(q - 1\right)^{2}\\0 & 0 & 0 & - q \left(q - 1\right)^{2} & 0 & 0 & 0 & 0 & 0 & 0 & 1 - q & 0 & 0 & 0 & 0 & \left(q - 1\right)^{2} & 0 & 0 & 0 & 0 & 1 - q & 0 & 0 & 0 & 0 & 1 - q & q \left(q - 1\right)\\0 & 0 & 0 & 0 & 0 & 0 & 0 & 0 & 0 & 0 & 0 & 1 - q & 0 & 0 & 0 & - \left(q - 1\right)^{3} & \left(q - 1\right)^{3} & 0 & 0 & 0 & 0 & 1 & 0 & 0 & 0 & \left(q - 1\right)^{2} & - q \left(q - 1\right)^{2}\\0 & 0 & 0 & q \left(q - 1\right) & 0 & 0 & 0 & 0 & 0 & 0 & 1 & 0 & 0 & 0 & 0 & \left(q - 1\right)^{2} & 0 & 0 & 0 & 0 & 1 & 0 & 0 & 0 & 0 & 1 - q & q \left(q - 1\right)\\0 & 0 & 0 & 0 & q \left(q - 1\right)^{2} & 0 & 0 & 0 & 0 & 0 & 0 & \left(q - 1\right)^{2} & 0 & q^{2} \left(q - 1\right) & q^{2} \left(q - 2\right) \left(q - 1\right) & - \left(q - 1\right)^{3} & - \left(q - 1\right)^{4} & 0 & 0 & 0 & 0 & 1 - q & 0 & - q^{2} & - q^{2} \left(q - 2\right) & \left(q - 1\right)^{2} & - q \left(q - 1\right)^{2}\\0 & 0 & 0 & 0 & \left(q - 1\right)^{2} & 0 & 0 & 0 & 0 & 0 & 0 & 0 & 1 - q & q - 1 & \left(q - 1\right) \left(q^{2} - q - 1\right) & 0 & 0 & 0 & 0 & 0 & 0 & 0 & q - 1 & \left(q - 1\right)^{2} & \left(q - 1\right) \left(q^{2} - 3 q + 1\right) & 1 - q & q - 1\\0 & 0 & 0 & 0 & 0 & 0 & 0 & 0 & 0 & 0 & 0 & 0 & 0 & 0 & 0 & 0 & 0 & 0 & 0 & 0 & 0 & 0 & 0 & 0 & 0 & 1 & -1\\0 & 0 & 0 & 0 & 1 - q & 0 & 0 & 0 & 0 & 0 & 0 & 0 & 1 & -1 & - q^{2} + q + 1 & 0 & 0 & 0 & 0 & 0 & 0 & 0 & -1 & 1 - q & - q^{2} + 3 q - 1 & 1 - q & q - 1\\0 & 0 & 0 & 0 & 0 & 0 & 0 & 0 & 0 & 0 & 0 & 1 - q & 0 & 0 & 0 & \left(q - 1\right)^{2} & \left(q - 1\right)^{3} & 0 & 0 & 0 & 0 & 1 - q & 0 & 0 & 0 & \left(q - 1\right)^{2} & 0\\0 & 0 & 0 & 0 & 0 & 0 & 0 & 0 & 0 & 0 & 0 & 0 & 0 & 0 & 0 & 1 - q & 0 & 0 & 0 & 0 & 0 & 0 & 0 & 0 & 0 & 1 - q & 0\\0 & 0 & 0 & 0 & 0 & 0 & 0 & 0 & 0 & 0 & 0 & 1 & 0 & 0 & 0 & \left(q - 1\right)^{2} & - \left(q - 1\right)^{2} & 0 & 0 & 0 & 0 & 1 & 0 & 0 & 0 & \left(q - 1\right)^{2} & 0\\0 & 0 & 0 & 0 & 0 & 0 & 0 & 0 & 0 & 0 & 0 & 0 & 0 & q \left(q - 1\right)^{2} & q \left(q - 1\right)^{3} & 0 & \left(q - 1\right)^{3} & 0 & 0 & 0 & 0 & 0 & - q \left(q - 1\right) & 0 & 0 & 0 & q \left(q - 1\right)^{2}\\0 & 0 & 1 - q & 0 & 0 & 0 & 0 & 1 - q & 0 & 0 & 0 & 0 & 0 & 0 & 0 & 0 & 0 & 0 & 0 & 0 & 0 & 0 & 0 & 0 & 0 & 0 & 1 - q\\0 & 0 & 0 & 0 & - q \left(q - 1\right)^{2} & 0 & 0 & 0 & 0 & 0 & 0 & 0 & - q \left(q - 1\right) & 0 & 0 & 0 & - \left(q - 1\right)^{2} & 0 & 0 & 0 & 0 & 0 & 0 & - q \left(q - 1\right) & q \left(q - 1\right) & 0 & 0\\0 & 0 & 0 & 0 & 0 & 0 & 0 & 0 & 0 & 0 & 0 & 0 & 0 & - q \left(q - 1\right) & - q \left(q - 1\right)^{2} & 0 & \left(q - 1\right)^{3} & 0 & 0 & 0 & 0 & 0 & q & 0 & 0 & 0 & q \left(q - 1\right)^{2}\\0 & 0 & 1 & 0 & 0 & 0 & 0 & 1 & 0 & 0 & 0 & 0 & 0 & 0 & 0 & 0 & 0 & 0 & 0 & 0 & 0 & 0 & 0 & 0 & 0 & 0 & 1 - q\\0 & 0 & 0 & 0 & q \left(q - 1\right) & 0 & 0 & 0 & 0 & 0 & 0 & 0 & q & 0 & 0 & 0 & - \left(q - 1\right)^{2} & 0 & 0 & 0 & 0 & 0 & 0 & q & - q & 0 & 0\\0 & 0 & 0 & 0 & \left(q - 1\right)^{2} & 0 & 0 & 0 & 0 & 0 & 0 & 0 & 0 & - q \left(q - 1\right) & 0 & - \left(q - 1\right)^{2} & 0 & 0 & 0 & 0 & 0 & 0 & 0 & - q \left(q - 1\right) & 0 & 0 & 0\\0 & 0 & 0 & 0 & 0 & 0 & 0 & 0 & 0 & 0 & 0 & 0 & 0 & 0 & 0 & q - 1 & 0 & 0 & 0 & 0 & 0 & 0 & 0 & 0 & 0 & 0 & 0\\0 & 0 & 0 & 0 & 1 - q & 0 & 0 & 0 & 0 & 0 & 0 & 0 & 0 & q & 0 & - \left(q - 1\right)^{2} & 0 & 0 & 0 & 0 & 0 & 0 & 0 & q & 0 & 0 & 0\\0 & 0 & 0 & 0 & 0 & 0 & 0 & 0 & 0 & 0 & 0 & 0 & 0 & 0 & 0 & 0 & - \left(q - 1\right)^{2} & 0 & 0 & 0 & 0 & 0 & 0 & 0 & 0 & 0 & - q \left(q - 1\right)\\0 & 0 & 0 & 0 & 0 & 0 & 0 & 0 & 0 & 0 & 0 & 0 & 0 & 0 & 0 & 0 & 0 & 0 & 0 & 0 & 0 & 0 & 0 & 0 & 0 & 0 & 1\\0 & 0 & 0 & 0 & 0 & 0 & 0 & 0 & 0 & 0 & 0 & 0 & 0 & 0 & 0 & 0 & q - 1 & 0 & 0 & 0 & 0 & 0 & 0 & 0 & 0 & 0 & 0\\0 & 0 & 0 & 0 & 1 - q & 0 & 0 & 0 & 0 & 0 & 0 & 0 & 0 & 0 & - q \left(q - 1\right) & q - 1 & 0 & 0 & 0 & 0 & 0 & 0 & 0 & 0 & - q \left(q - 1\right) & 0 & 0\\0 & 0 & 0 & 0 & 0 & 0 & 0 & 0 & 0 & 0 & 0 & 0 & 0 & 0 & 0 & -1 & 0 & 0 & 0 & 0 & 0 & 0 & 0 & 0 & 0 & 0 & 0\\0 & 0 & 0 & 0 & 1 & 0 & 0 & 0 & 0 & 0 & 0 & 0 & 0 & 0 & q & q - 1 & 0 & 0 & 0 & 0 & 0 & 0 & 0 & 0 & q & 0 & 0\\0 & 0 & 0 & 0 & 0 & 0 & 0 & 0 & 0 & 0 & 0 & 0 & 0 & 0 & 0 & 0 & 0 & 0 & 0 & 0 & 0 & 0 & 0 & 0 & 0 & 0 & 0\\0 & 0 & 0 & 0 & 0 & 0 & 0 & 0 & 0 & 0 & 0 & 0 & 0 & 0 & 0 & 0 & 1 - q & 0 & 0 & 0 & 0 & 0 & 0 & 0 & 0 & 0 & 0\\0 & 0 & 0 & 0 & 0 & 0 & 0 & 0 & 0 & 0 & 0 & 0 & 0 & 0 & 0 & 0 & 0 & 0 & 0 & 0 & 0 & 0 & 0 & 0 & 0 & 0 & 0\\0 & 0 & 0 & 0 & 0 & 0 & 0 & 0 & 0 & 0 & 0 & 0 & 0 & 0 & 0 & 0 & 1 & 0 & 0 & 0 & 0 & 0 & 0 & 0 & 0 & 0 & 0\\0 & 0 & 0 & 0 & 0 & 0 & 0 & 0 & 0 & 0 & 0 & 0 & 0 & 0 & 0 & 0 & 0 & 0 & 0 & 0 & 0 & 0 & 0 & 0 & 0 & 0 & 0\\0 & 0 & 0 & 0 & 0 & 0 & 0 & 0 & 0 & 0 & 0 & 0 & 0 & 0 & 0 & 0 & 0 & 0 & 0 & 0 & 0 & 0 & 0 & 0 & 0 & 0 & 0\end{array}\right.$} \]
\[ \scalebox{0.30}{$\footnotesize
\left.\begin{array}{cccccccccccccccccccccccccccccccccc}0 & q \left(q - 1\right)^{3} & 0 & 0 & q \left(q - 1\right)^{3} & \left(q - 1\right)^{4} & - q \left(q - 1\right) & - q \left(q - 1\right) & 0 & - \left(q - 1\right)^{2} & 0 & 0 & 0 & - q \left(q - 1\right)^{2} & 0 & - q \left(q - 1\right)^{2} & 0 & 0 & 0 & - \left(q - 1\right)^{3} & 1 - q & 0 & 0 & 0 & 0 & 0 & 0 & 0 & \left(q - 1\right)^{2} & 0 & 0 & 0 & 1 - q & 1\\0 & \left(q - 1\right)^{2} \left(q^{2} - 3 q + 1\right) & \left(q - 1\right)^{2} & - \left(q - 1\right)^{2} & - \left(q - 1\right)^{2} & - \left(q - 1\right)^{3} & - q \left(q - 1\right) & 0 & 1 - q & 0 & 0 & 0 & 0 & - q \left(q - 1\right)^{2} & 0 & 0 & 0 & 0 & - \left(q - 1\right)^{2} & \left(2 - q\right) \left(q - 1\right)^{2} & 1 - q & 0 & q - 1 & 0 & 0 & 0 & 0 & q - 1 & \left(q - 2\right) \left(q - 1\right) & 0 & 0 & -1 & 2 - q & 1\\0 & q \left(q - 1\right)^{3} & 0 & q \left(q - 1\right)^{2} & 0 & \left(q - 1\right)^{4} & - q \left(q - 1\right) & 0 & 0 & - \left(q - 1\right)^{2} & 0 & 0 & 0 & - q \left(q - 1\right)^{2} & - q \left(q - 1\right) & 0 & 0 & 0 & 0 & - \left(q - 1\right)^{3} & 1 - q & 0 & q & 0 & 0 & 0 & 0 & 0 & \left(q - 1\right)^{2} & 0 & 0 & 0 & 1 - q & 1\\0 & q \left(q - 1\right)^{3} & 0 & 0 & q \left(q - 1\right)^{3} & \left(q - 1\right)^{4} & - q \left(q - 1\right) & q & 0 & - \left(q - 1\right)^{2} & 0 & 0 & 0 & - q \left(q - 1\right)^{2} & 0 & - q \left(q - 1\right)^{2} & 0 & 0 & 0 & - \left(q - 1\right)^{3} & 1 - q & 0 & 0 & 0 & 0 & 0 & 0 & 0 & \left(q - 1\right)^{2} & 0 & 0 & 0 & 1 - q & 1\\0 & q \left(q - 1\right)^{3} & 0 & 0 & q \left(q - 1\right)^{3} & \left(q - 1\right)^{4} & - q \left(q - 1\right) & - q \left(q - 1\right) & 0 & - \left(q - 1\right)^{2} & 0 & 0 & 0 & - q \left(q - 1\right)^{2} & 0 & - q \left(q - 1\right)^{2} & 0 & 0 & 0 & - \left(q - 1\right)^{3} & 1 - q & 0 & 0 & 0 & 0 & 0 & 0 & 0 & \left(q - 1\right)^{2} & 0 & 0 & 0 & 1 - q & 1\\0 & - \left(q - 1\right)^{2} & 0 & - \left(q - 1\right)^{2} & - \left(q - 1\right)^{2} & - \left(q - 1\right)^{3} & 0 & 1 - q & q - 1 & \left(2 - q\right) \left(q - 1\right) & q - 1 & 0 & \left(q - 1\right)^{2} & q - 1 & \left(2 - q\right) \left(q - 1\right) & \left(2 - q\right) \left(q - 1\right) & 0 & - \left(q - 1\right)^{2} & 0 & \left(2 - q\right) \left(q - 1\right)^{2} & 1 - q & q - 1 & 0 & 0 & 0 & 0 & q - 1 & 0 & \left(q - 2\right) \left(q - 1\right) & 0 & -1 & 0 & 2 - q & 1\\0 & q - 1 & 0 & q - 1 & q - 1 & \left(q - 1\right)^{2} & 0 & 0 & 0 & q - 1 & 0 & q - 1 & 0 & q - 1 & q - 1 & q - 1 & 0 & q - 1 & q - 1 & \left(q - 1\right) \left(2 q - 3\right) & 1 - q & q - 1 & 0 & 0 & 0 & 1 & q - 2 & q - 2 & \left(q - 2\right)^{2} & 0 & -1 & -1 & 3 - q & 1\\0 & - \left(q - 1\right)^{2} & 0 & - \left(q - 1\right)^{2} & - \left(q - 1\right)^{2} & - \left(q - 1\right)^{3} & 0 & 1 & q - 1 & \left(2 - q\right) \left(q - 1\right) & q - 1 & 0 & \left(q - 1\right)^{2} & q - 1 & \left(2 - q\right) \left(q - 1\right) & \left(2 - q\right) \left(q - 1\right) & 0 & - \left(q - 1\right)^{2} & 0 & \left(2 - q\right) \left(q - 1\right)^{2} & 1 - q & q - 1 & 0 & 0 & 0 & 0 & q - 1 & 0 & \left(q - 2\right) \left(q - 1\right) & 0 & -1 & 0 & 2 - q & 1\\0 & - \left(q - 1\right)^{2} & 0 & - \left(q - 1\right)^{2} & - \left(q - 1\right)^{2} & - \left(q - 1\right)^{3} & 0 & 1 - q & q - 1 & \left(2 - q\right) \left(q - 1\right) & q - 1 & 0 & \left(q - 1\right)^{2} & q - 1 & \left(2 - q\right) \left(q - 1\right) & \left(2 - q\right) \left(q - 1\right) & 0 & - \left(q - 1\right)^{2} & 0 & \left(2 - q\right) \left(q - 1\right)^{2} & 1 - q & q - 1 & 0 & 0 & 0 & 0 & q - 1 & 0 & \left(q - 2\right) \left(q - 1\right) & 0 & -1 & 0 & 2 - q & 1\\0 & 0 & 0 & 0 & 0 & \left(q - 1\right)^{4} & 0 & 0 & 0 & - \left(q - 1\right)^{2} & 0 & - q \left(q - 1\right) & 0 & 0 & 0 & 0 & 0 & 0 & 0 & - \left(q - 1\right)^{3} & 1 - q & q & - q & 0 & 0 & 0 & 0 & 0 & \left(q - 1\right)^{2} & 0 & 0 & 0 & 1 - q & 1\\0 & - \left(q - 1\right)^{3} & 1 - q & - \left(q - 1\right)^{2} & - \left(q - 1\right)^{2} & - \left(q - 1\right)^{3} & 0 & 0 & 1 - q & 0 & - q & - q \left(q - 2\right) & - q \left(q - 1\right) & 0 & 0 & 0 & 0 & 0 & - \left(q - 1\right)^{2} & \left(2 - q\right) \left(q - 1\right)^{2} & 1 - q & q & -1 & 0 & 0 & 0 & 0 & q - 1 & \left(q - 2\right) \left(q - 1\right) & 0 & 0 & -1 & 2 - q & 1\\0 & 0 & 0 & q \left(q - 1\right)^{2} & q \left(q - 1\right)^{2} & \left(q - 1\right)^{4} & 0 & 0 & 0 & - \left(q - 1\right)^{2} & 0 & - q \left(q - 1\right) & 0 & 0 & - q \left(q - 1\right) & - q \left(q - 1\right) & 0 & 0 & 0 & - \left(q - 1\right)^{3} & 1 - q & q & 0 & 0 & 0 & 0 & 0 & 0 & \left(q - 1\right)^{2} & 0 & 0 & 0 & 1 - q & 1\\0 & 0 & 0 & 0 & 0 & \left(q - 1\right)^{4} & 0 & 0 & 0 & - \left(q - 1\right)^{2} & 0 & - q \left(q - 1\right) & 0 & 0 & 0 & 0 & 0 & 0 & 0 & - \left(q - 1\right)^{3} & 1 - q & q & - q & 0 & 0 & 0 & 0 & 0 & \left(q - 1\right)^{2} & 0 & 0 & 0 & 1 - q & 1\\0 & q \left(q - 1\right)^{3} & 0 & 0 & q \left(q - 1\right)^{3} & \left(q - 1\right)^{4} & q & - q \left(q - 1\right) & 0 & - \left(q - 1\right)^{2} & 0 & 0 & 0 & - q \left(q - 1\right)^{2} & 0 & - q \left(q - 1\right)^{2} & 0 & 0 & 0 & - \left(q - 1\right)^{3} & 1 - q & 0 & 0 & 0 & 0 & 0 & 0 & 0 & \left(q - 1\right)^{2} & 0 & 0 & 0 & 1 - q & 1\\0 & \left(q - 1\right)^{2} \left(q^{2} - 3 q + 1\right) & \left(q - 1\right)^{2} & - \left(q - 1\right)^{2} & - \left(q - 1\right)^{2} & - \left(q - 1\right)^{3} & q & 0 & 1 - q & 0 & 0 & 0 & 0 & - q \left(q - 1\right)^{2} & 0 & 0 & 0 & 0 & - \left(q - 1\right)^{2} & \left(2 - q\right) \left(q - 1\right)^{2} & 1 - q & 0 & q - 1 & 0 & 0 & 0 & 0 & q - 1 & \left(q - 2\right) \left(q - 1\right) & 0 & 0 & -1 & 2 - q & 1\\0 & q \left(q - 1\right)^{3} & 0 & q \left(q - 1\right)^{2} & 0 & \left(q - 1\right)^{4} & q & 0 & 0 & - \left(q - 1\right)^{2} & 0 & 0 & 0 & - q \left(q - 1\right)^{2} & - q \left(q - 1\right) & 0 & 0 & 0 & 0 & - \left(q - 1\right)^{3} & 1 - q & 0 & q & 0 & 0 & 0 & 0 & 0 & \left(q - 1\right)^{2} & 0 & 0 & 0 & 1 - q & 1\\0 & q \left(q - 1\right)^{3} & 0 & 0 & q \left(q - 1\right)^{3} & \left(q - 1\right)^{4} & q & q & 0 & - \left(q - 1\right)^{2} & 0 & 0 & 0 & - q \left(q - 1\right)^{2} & 0 & - q \left(q - 1\right)^{2} & 0 & 0 & 0 & - \left(q - 1\right)^{3} & 1 - q & 0 & 0 & 0 & 0 & 0 & 0 & 0 & \left(q - 1\right)^{2} & 0 & 0 & 0 & 1 - q & 1\\0 & 0 & 0 & \left(q - 1\right)^{2} & 0 & - \left(q - 1\right)^{3} & 1 - q & 0 & 0 & 0 & 1 - q & 0 & q - 1 & 0 & \left(q - 2\right) \left(q - 1\right) & \left(q - 1\right)^{2} & \left(q - 1\right)^{2} & 0 & \left(q - 1\right)^{2} & \left(q - 1\right)^{2} & 0 & 0 & 1 - q & q - 1 & 0 & 0 & 0 & 0 & \left(q - 2\right) \left(q - 1\right) & -1 & 0 & 1 & 1 - q & 1\\1 - q & 0 & 0 & - \left(q - 1\right)^{2} & \left(q - 2\right) \left(q - 1\right)^{2} & \left(q - 1\right)^{2} & 1 - q & 0 & 0 & 0 & 0 & 0 & 0 & 0 & 0 & - \left(q - 1\right)^{2} & 1 - q & 0 & 0 & \left(q - 2\right) \left(q - 1\right) & 0 & 0 & 0 & q - 2 & 1 & 0 & 0 & q - 2 & \left(q - 2\right)^{2} & -1 & 0 & 0 & 2 - q & 1\\0 & 0 & 0 & 1 - q & 0 & - \left(q - 1\right)^{3} & 1 - q & 0 & 0 & 0 & 1 & 0 & -1 & 0 & 2 - q & 1 - q & \left(q - 1\right)^{2} & 0 & \left(q - 1\right)^{2} & \left(q - 1\right)^{2} & 0 & 0 & 1 & q - 1 & 0 & 0 & 0 & 0 & \left(q - 2\right) \left(q - 1\right) & -1 & 0 & 1 & 1 - q & 1\\1 & 0 & 0 & q - 1 & \left(2 - q\right) \left(q - 1\right) & \left(q - 1\right)^{2} & 1 - q & 0 & 0 & 0 & 0 & 0 & 0 & 0 & 0 & q - 1 & 1 - q & 0 & 0 & \left(q - 2\right) \left(q - 1\right) & 0 & 0 & 0 & q - 2 & 1 & 0 & 0 & q - 2 & \left(q - 2\right)^{2} & -1 & 0 & 0 & 2 - q & 1\\0 & 0 & 0 & \left(q - 1\right)^{2} & 0 & - \left(q - 1\right)^{3} & 1 - q & 0 & 0 & 0 & 1 - q & 0 & q - 1 & 0 & \left(q - 2\right) \left(q - 1\right) & \left(q - 1\right)^{2} & \left(q - 1\right)^{2} & 0 & \left(q - 1\right)^{2} & \left(q - 1\right)^{2} & 0 & 0 & 1 - q & q - 1 & 0 & 0 & 0 & 0 & \left(q - 2\right) \left(q - 1\right) & -1 & 0 & 1 & 1 - q & 1\\0 & 0 & 0 & 0 & 0 & \left(q - 1\right)^{2} & 0 & 0 & 0 & 0 & 0 & 0 & 0 & 0 & 0 & 0 & 1 - q & q - 1 & 1 - q & \left(q - 2\right) \left(q - 1\right) & 0 & 0 & 0 & 0 & -1 & -1 & 0 & 2 - q & 1 - q & -1 & -1 & 1 & 2 - q & 1\\0 & 0 & 0 & 0 & 0 & 1 - q & 0 & 0 & 0 & 0 & 0 & 0 & 0 & 0 & 0 & 0 & 1 & -1 & 0 & 3 - 2 q & 0 & 0 & 0 & -1 & 0 & 0 & -1 & -1 & 4 - 2 q & -1 & -1 & 0 & 3 - q & 1\\0 & 0 & 0 & 0 & 0 & \left(q - 1\right)^{2} & 0 & 0 & 0 & 0 & 0 & 0 & 0 & 0 & 0 & 0 & 1 - q & q - 1 & 1 - q & \left(q - 2\right) \left(q - 1\right) & 0 & 0 & 0 & 0 & -1 & -1 & 0 & 2 - q & 1 - q & -1 & -1 & 1 & 2 - q & 1\\0 & 0 & 0 & \left(q - 1\right)^{2} & 0 & - \left(q - 1\right)^{3} & 1 & 0 & 0 & 0 & 1 - q & 0 & q - 1 & 0 & \left(q - 2\right) \left(q - 1\right) & \left(q - 1\right)^{2} & \left(q - 1\right)^{2} & 0 & \left(q - 1\right)^{2} & \left(q - 1\right)^{2} & 0 & 0 & 1 - q & q - 1 & 0 & 0 & 0 & 0 & \left(q - 2\right) \left(q - 1\right) & -1 & 0 & 1 & 1 - q & 1\\1 - q & 0 & 0 & - \left(q - 1\right)^{2} & \left(q - 2\right) \left(q - 1\right)^{2} & \left(q - 1\right)^{2} & 1 & 0 & 0 & 0 & 0 & 0 & 0 & 0 & 0 & - \left(q - 1\right)^{2} & 1 - q & 0 & 0 & \left(q - 2\right) \left(q - 1\right) & 0 & 0 & 0 & q - 2 & 1 & 0 & 0 & q - 2 & \left(q - 2\right)^{2} & -1 & 0 & 0 & 2 - q & 1\\0 & 0 & 0 & 1 - q & 0 & - \left(q - 1\right)^{3} & 1 & 0 & 0 & 0 & 1 & 0 & -1 & 0 & 2 - q & 1 - q & \left(q - 1\right)^{2} & 0 & \left(q - 1\right)^{2} & \left(q - 1\right)^{2} & 0 & 0 & 1 & q - 1 & 0 & 0 & 0 & 0 & \left(q - 2\right) \left(q - 1\right) & -1 & 0 & 1 & 1 - q & 1\\1 & 0 & 0 & q - 1 & \left(2 - q\right) \left(q - 1\right) & \left(q - 1\right)^{2} & 1 & 0 & 0 & 0 & 0 & 0 & 0 & 0 & 0 & q - 1 & 1 - q & 0 & 0 & \left(q - 2\right) \left(q - 1\right) & 0 & 0 & 0 & q - 2 & 1 & 0 & 0 & q - 2 & \left(q - 2\right)^{2} & -1 & 0 & 0 & 2 - q & 1\\0 & - q \left(q - 1\right)^{2} & 0 & - q \left(q - 1\right)^{2} & q \left(q - 2\right) \left(q - 1\right)^{2} & \left(q - 1\right)^{4} & 0 & - q \left(q - 1\right) & 0 & q - 1 & 0 & q \left(q - 1\right) & 0 & q \left(q - 1\right) & q \left(q - 1\right) & - q \left(q - 2\right) \left(q - 1\right) & 0 & 0 & 0 & - \left(q - 1\right)^{3} & 1 & - q & 0 & 0 & 0 & 0 & 0 & 0 & \left(q - 1\right)^{2} & 0 & 0 & 0 & 1 - q & 1\\0 & \left(1 - q\right) \left(q^{2} - 3 q + 1\right) & 1 - q & q - 1 & q - 1 & - \left(q - 1\right)^{3} & 0 & 0 & 1 & 0 & q & q \left(q - 2\right) & q \left(q - 1\right) & q \left(q - 1\right) & 0 & 0 & 0 & 0 & - \left(q - 1\right)^{2} & \left(2 - q\right) \left(q - 1\right)^{2} & 1 & - q & q - 1 & 0 & 0 & 0 & 0 & q - 1 & \left(q - 2\right) \left(q - 1\right) & 0 & 0 & -1 & 2 - q & 1\\0 & - q \left(q - 1\right)^{2} & 0 & 0 & - q \left(q - 1\right)^{2} & \left(q - 1\right)^{4} & 0 & 0 & 0 & q - 1 & 0 & q \left(q - 1\right) & 0 & q \left(q - 1\right) & 0 & q \left(q - 1\right) & 0 & 0 & 0 & - \left(q - 1\right)^{3} & 1 & - q & q & 0 & 0 & 0 & 0 & 0 & \left(q - 1\right)^{2} & 0 & 0 & 0 & 1 - q & 1\\0 & \left(1 - q\right) \left(q^{2} - 3 q + 1\right) & 1 - q & q - 1 & q - 1 & - \left(q - 1\right)^{3} & 0 & 0 & 1 & 0 & q & q \left(q - 2\right) & q \left(q - 1\right) & q \left(q - 1\right) & 0 & 0 & 0 & 0 & - \left(q - 1\right)^{2} & \left(2 - q\right) \left(q - 1\right)^{2} & 1 & - q & q - 1 & 0 & 0 & 0 & 0 & q - 1 & \left(q - 2\right) \left(q - 1\right) & 0 & 0 & -1 & 2 - q & 1\\0 & - q \left(q - 1\right)^{2} & 0 & - q \left(q - 1\right)^{2} & q \left(q - 2\right) \left(q - 1\right)^{2} & \left(q - 1\right)^{4} & 0 & q & 0 & q - 1 & 0 & q \left(q - 1\right) & 0 & q \left(q - 1\right) & q \left(q - 1\right) & - q \left(q - 2\right) \left(q - 1\right) & 0 & 0 & 0 & - \left(q - 1\right)^{3} & 1 & - q & 0 & 0 & 0 & 0 & 0 & 0 & \left(q - 1\right)^{2} & 0 & 0 & 0 & 1 - q & 1\\0 & q - 1 & 0 & q - 1 & q - 1 & - \left(q - 1\right)^{3} & 0 & 1 - q & -1 & q - 2 & -1 & 0 & 1 - q & -1 & q - 2 & q - 2 & 0 & - \left(q - 1\right)^{2} & 0 & \left(2 - q\right) \left(q - 1\right)^{2} & 1 & -1 & 0 & 0 & 0 & 0 & q - 1 & 0 & \left(q - 2\right) \left(q - 1\right) & 0 & -1 & 0 & 2 - q & 1\\0 & -1 & 0 & -1 & -1 & \left(q - 1\right)^{2} & 0 & 0 & 0 & -1 & 0 & -1 & 0 & -1 & -1 & -1 & 0 & q - 1 & q - 1 & \left(q - 1\right) \left(2 q - 3\right) & 1 & -1 & 0 & 0 & 0 & 1 & q - 2 & q - 2 & \left(q - 2\right)^{2} & 0 & -1 & -1 & 3 - q & 1\\0 & q - 1 & 0 & q - 1 & q - 1 & - \left(q - 1\right)^{3} & 0 & 1 & -1 & q - 2 & -1 & 0 & 1 - q & -1 & q - 2 & q - 2 & 0 & - \left(q - 1\right)^{2} & 0 & \left(2 - q\right) \left(q - 1\right)^{2} & 1 & -1 & 0 & 0 & 0 & 0 & q - 1 & 0 & \left(q - 2\right) \left(q - 1\right) & 0 & -1 & 0 & 2 - q & 1\\0 & 0 & 0 & - q \left(q - 1\right)^{2} & - q \left(q - 1\right)^{2} & \left(q - 1\right)^{4} & 0 & 0 & 0 & q - 1 & 0 & 0 & 0 & 0 & q \left(q - 1\right) & q \left(q - 1\right) & 0 & 0 & 0 & - \left(q - 1\right)^{3} & 1 & 0 & - q & 0 & 0 & 0 & 0 & 0 & \left(q - 1\right)^{2} & 0 & 0 & 0 & 1 - q & 1\\0 & \left(q - 1\right)^{2} & 1 & q - 1 & q - 1 & - \left(q - 1\right)^{3} & 0 & 0 & 1 & 0 & 0 & 0 & 0 & 0 & 0 & 0 & 0 & 0 & - \left(q - 1\right)^{2} & \left(2 - q\right) \left(q - 1\right)^{2} & 1 & 0 & -1 & 0 & 0 & 0 & 0 & q - 1 & \left(q - 2\right) \left(q - 1\right) & 0 & 0 & -1 & 2 - q & 1\\0 & 0 & 0 & 0 & 0 & \left(q - 1\right)^{4} & 0 & 0 & 0 & q - 1 & 0 & 0 & 0 & 0 & 0 & 0 & 0 & 0 & 0 & - \left(q - 1\right)^{3} & 1 & 0 & 0 & 0 & 0 & 0 & 0 & 0 & \left(q - 1\right)^{2} & 0 & 0 & 0 & 1 - q & 1\\\left(q - 1\right)^{2} & 0 & 0 & q \left(q - 1\right)^{2} & 0 & - \left(q - 1\right)^{3} & 0 & - q \left(q - 1\right) & 0 & 0 & 0 & 0 & - q \left(q - 1\right) & 0 & 0 & - q \left(q - 1\right)^{2} & - \left(q - 1\right)^{2} & \left(q - 1\right)^{2} & 0 & \left(q - 1\right)^{3} & 0 & 1 - q & 0 & 1 - q & 0 & 0 & 1 - q & 1 - q & \left(2 - q\right) \left(q - 1\right) & 1 & 1 & 0 & -1 & 1\\0 & 0 & 0 & 0 & 0 & \left(q - 1\right)^{2} & 0 & 0 & 0 & 0 & 0 & 1 - q & 0 & 0 & 0 & 0 & q - 1 & 1 - q & q - 1 & 0 & 0 & 1 - q & q - 1 & 2 - q & -1 & -1 & 2 - q & 2 - q & - q^{2} + 3 q - 3 & 1 & 1 & -1 & 0 & 1\\1 - q & 0 & 0 & - q \left(q - 1\right) & 0 & - \left(q - 1\right)^{3} & 0 & 0 & 0 & 0 & - q & 0 & - q \left(q - 2\right) & 0 & 0 & q \left(q - 1\right) & - \left(q - 1\right)^{2} & \left(q - 1\right)^{2} & 0 & \left(q - 1\right)^{3} & 0 & 1 - q & q & 1 - q & 0 & 0 & 1 - q & 1 - q & \left(2 - q\right) \left(q - 1\right) & 1 & 1 & 0 & -1 & 1\\\left(q - 1\right)^{2} & 0 & 0 & q \left(q - 1\right)^{2} & 0 & - \left(q - 1\right)^{3} & 0 & q & 0 & 0 & 0 & 0 & - q \left(q - 1\right) & 0 & 0 & - q \left(q - 1\right)^{2} & - \left(q - 1\right)^{2} & \left(q - 1\right)^{2} & 0 & \left(q - 1\right)^{3} & 0 & 1 - q & 0 & 1 - q & 0 & 0 & 1 - q & 1 - q & \left(2 - q\right) \left(q - 1\right) & 1 & 1 & 0 & -1 & 1\\0 & 0 & 0 & 0 & 0 & \left(q - 1\right)^{2} & 0 & 0 & 0 & 0 & 0 & 1 - q & 0 & 0 & 0 & 0 & q - 1 & 1 - q & q - 1 & 0 & 0 & 1 - q & q - 1 & 2 - q & -1 & -1 & 2 - q & 2 - q & - q^{2} + 3 q - 3 & 1 & 1 & -1 & 0 & 1\\1 - q & 0 & 0 & - q \left(q - 1\right) & 0 & - \left(q - 1\right)^{3} & 0 & 0 & 0 & 0 & - q & 0 & - q \left(q - 2\right) & 0 & 0 & q \left(q - 1\right) & - \left(q - 1\right)^{2} & \left(q - 1\right)^{2} & 0 & \left(q - 1\right)^{3} & 0 & 1 - q & q & 1 - q & 0 & 0 & 1 - q & 1 - q & \left(2 - q\right) \left(q - 1\right) & 1 & 1 & 0 & -1 & 1\\0 & 0 & 1 - q & 0 & 0 & \left(q - 1\right)^{2} & 0 & 1 - q & 0 & 0 & 0 & 0 & 0 & - \left(q - 1\right)^{2} & 0 & 0 & q - 1 & 0 & 0 & 0 & 0 & 0 & 0 & 0 & 1 & 0 & 0 & 0 & 0 & 1 & 0 & 0 & 0 & 1\\0 & - \left(q - 1\right)^{2} & 0 & 0 & 0 & 1 - q & 0 & 0 & 0 & 0 & 0 & 0 & 0 & - \left(q - 1\right)^{2} & 0 & 0 & -1 & 0 & -1 & 1 - q & 0 & 0 & 0 & 1 & 0 & 0 & 0 & 0 & 0 & 1 & 0 & -1 & 1 & 1\\0 & 0 & 1 - q & 0 & 0 & \left(q - 1\right)^{2} & 0 & 1 & 0 & 0 & 0 & 0 & 0 & - \left(q - 1\right)^{2} & 0 & 0 & q - 1 & 0 & 0 & 0 & 0 & 0 & 0 & 0 & 1 & 0 & 0 & 0 & 0 & 1 & 0 & 0 & 0 & 1\\1 - q & 0 & 0 & 0 & 0 & - \left(q - 1\right)^{3} & 0 & 0 & 0 & 0 & q & 0 & 0 & 0 & 0 & 0 & - \left(q - 1\right)^{2} & \left(q - 1\right)^{2} & 0 & \left(q - 1\right)^{3} & 0 & 1 & - q & 1 - q & 0 & 0 & 1 - q & 1 - q & \left(2 - q\right) \left(q - 1\right) & 1 & 1 & 0 & -1 & 1\\0 & 0 & 0 & 0 & 0 & \left(q - 1\right)^{2} & 0 & 0 & 0 & 0 & 0 & 1 & 0 & 0 & 0 & 0 & q - 1 & 1 - q & q - 1 & 0 & 0 & 1 & -1 & 2 - q & -1 & -1 & 2 - q & 2 - q & - q^{2} + 3 q - 3 & 1 & 1 & -1 & 0 & 1\\1 & 0 & 0 & 0 & 0 & - \left(q - 1\right)^{3} & 0 & 0 & 0 & 0 & 0 & 0 & q & 0 & 0 & 0 & - \left(q - 1\right)^{2} & \left(q - 1\right)^{2} & 0 & \left(q - 1\right)^{3} & 0 & 1 & 0 & 1 - q & 0 & 0 & 1 - q & 1 - q & \left(2 - q\right) \left(q - 1\right) & 1 & 1 & 0 & -1 & 1\\0 & 0 & 1 & 0 & 0 & \left(q - 1\right)^{2} & 0 & 1 - q & 0 & 0 & 0 & 0 & 0 & q - 1 & 0 & 0 & q - 1 & 0 & 0 & 0 & 0 & 0 & 0 & 0 & 1 & 0 & 0 & 0 & 0 & 1 & 0 & 0 & 0 & 1\\0 & q - 1 & 0 & 0 & 0 & 1 - q & 0 & 0 & 0 & 0 & 0 & 0 & 0 & q - 1 & 0 & 0 & -1 & 0 & -1 & 1 - q & 0 & 0 & 0 & 1 & 0 & 0 & 0 & 0 & 0 & 1 & 0 & -1 & 1 & 1\\0 & 0 & 1 & 0 & 0 & \left(q - 1\right)^{2} & 0 & 1 & 0 & 0 & 0 & 0 & 0 & q - 1 & 0 & 0 & q - 1 & 0 & 0 & 0 & 0 & 0 & 0 & 0 & 1 & 0 & 0 & 0 & 0 & 1 & 0 & 0 & 0 & 1\\0 & 0 & 0 & 1 - q & 0 & \left(q - 1\right)^{2} & 0 & 0 & 0 & 0 & 0 & 0 & 0 & 0 & 1 - q & 0 & 0 & 1 - q & 1 - q & - \left(q - 1\right)^{2} & 0 & 0 & 1 - q & 0 & 0 & 1 & 0 & 0 & 0 & 0 & 1 & 1 & -1 & 1\\0 & 0 & 0 & 0 & - \left(q - 1\right)^{2} & 1 - q & 0 & 0 & 0 & 0 & 0 & 0 & 0 & 0 & 0 & - \left(q - 1\right)^{2} & 0 & 1 & 0 & 0 & 0 & 0 & 0 & 0 & 0 & 0 & 1 & 0 & 0 & 0 & 1 & 0 & 0 & 1\\0 & 0 & 0 & 1 & 0 & \left(q - 1\right)^{2} & 0 & 0 & 0 & 0 & 0 & 0 & 0 & 0 & 1 & 0 & 0 & 1 - q & 1 - q & - \left(q - 1\right)^{2} & 0 & 0 & 1 & 0 & 0 & 1 & 0 & 0 & 0 & 0 & 1 & 1 & -1 & 1\\0 & 0 & 0 & 0 & q - 1 & 1 - q & 0 & 0 & 0 & 0 & 0 & 0 & 0 & 0 & 0 & q - 1 & 0 & 1 & 0 & 0 & 0 & 0 & 0 & 0 & 0 & 0 & 1 & 0 & 0 & 0 & 1 & 0 & 0 & 1\\0 & 0 & 0 & 0 & 0 & 1 - q & 0 & 0 & 0 & 0 & 0 & 0 & 0 & 0 & 0 & 0 & 0 & 0 & 1 & 0 & 0 & 0 & 0 & 0 & 0 & 0 & 0 & 1 & 0 & 0 & 0 & 1 & 0 & 1\\0 & 0 & 0 & 0 & 0 & 1 & 0 & 0 & 0 & 0 & 0 & 0 & 0 & 0 & 0 & 0 & 0 & 0 & 0 & 1 & 0 & 0 & 0 & 0 & 0 & 0 & 0 & 0 & 1 & 0 & 0 & 0 & 1 & 1\end{array}\right]$} \]

Theorem \ref{thm:tqft_method} and \eqref{eq:Tn_from_tilde_Tn} yield the following theorem.

\begin{theorem}
    \label{thm:virtual_class_T5_representation_variety}
    The virtual class of the $\TT_5$-representation variety of $\Sigma_g$ is given by
    \begin{align*}
        [R_{\TT_5}(\Sigma_g)]
            &= q^{12 g - 2} \left(q - 1\right)^{6 g + 2} + 2 q^{14 g - 4} \left(q - 1\right)^{4 g + 3} + 3 q^{14 g - 4} \left(q - 1\right)^{6 g + 2} + q^{14 g - 4} \left(q - 1\right)^{8 g + 1} \\
            &\quad + 2 q^{16 g - 6} \left(q - 1\right)^{2 g + 4} + 7 q^{16 g - 6} \left(q - 1\right)^{4 g + 3} + 7 q^{16 g - 6} \left(q - 1\right)^{6 g + 2} + 2 q^{16 g - 6} \left(q - 1\right)^{8 g + 1} \\
            &\quad + 2 q^{18 g - 8} \left(q - 1\right)^{2 g + 4} + 7 q^{18 g - 8} \left(q - 1\right)^{4 g + 3} + 8 q^{18 g - 8} \left(q - 1\right)^{6 g + 2} + 3 q^{18 g - 8} \left(q - 1\right)^{8 g + 1} \\
            &\quad + q^{20 g - 10} \left(q - 1\right)^{10 g} + q^{20 g - 10} \left(q - 1\right)^{2 g + 4} + 4 q^{20 g - 10} \left(q - 1\right)^{4 g + 3} + 6 q^{20 g - 10} \left(q - 1\right)^{6 g + 2} \\
            &\quad + 4 q^{20 g - 10} \left(q - 1\right)^{8 g + 1} . \tag*{\qed}
    \end{align*}
\end{theorem}

\begin{remark}
    For small $g$, we have
    \begin{align*}
        [R_{\TT_5}(\Sigma_{0})] &= 1 , \\
        [R_{\TT_5}(\Sigma_{1})] &= q^{10} \left(q - 1\right)^{6} \left(q^{4} + 6 q^{3} + q^{2} - 4 q + 1\right) , \\
        [R_{\TT_5}(\Sigma_{2})] &= q^{22} \left(q - 1\right)^{8} \left(q^{20} - 12 q^{19} + 66 q^{18} - 216 q^{17} + 459 q^{16} - 645 q^{15} + 567 q^{14} - 214 q^{13} - 181 q^{12} \right. \\ &\quad \left. + 403 q^{11} - 428 q^{10} + 345 q^{9} - 237 q^{8} + 143 q^{7} - 68 q^{6} + 17 q^{5} + 18 q^{4} - 23 q^{3} + 15 q^{2} - 6 q + 1\right) .
    \end{align*}
\end{remark}

Note that precisely the same method can be applied to the groups $G = \UU_n$ for $n = 1, 2, 3, 4, 5$. In fact, the coefficients $F_{ijk}$ can be reused.

\begin{theorem}
    \label{thm:virtual_class_Un_representation_varieties}
    The virtual classes of the $G$-representation variety of $\Sigma_g$ for $G = \UU_2, \UU_3, \UU_4$ and $\UU_5$ are given by
    \begin{align*}
        [R_{\UU_2}(\Sigma_g)] &= q^{2g} , \\
        [R_{\UU_3}(\Sigma_g)] &= q^{4 g - 1} \left(q - 1\right) + q^{6 g - 1} , \\
        [R_{\UU_4}(\Sigma_g)] &= q^{8 g - 1} \left(q - 1\right) + q^{10 g - 3} \left(q - 1\right) \left(q + 1\right) + q^{12 g - 3} , \\
        [R_{\UU_5}(\Sigma_g)] &= q^{12 g - 2} \left(q - 1\right)^{2} + q^{14 g - 3} \left(q - 1\right) \left(2 q - 1\right) \\ &\quad + q^{16 g - 5} \left(q - 1\right) \left(q + 1\right) \left(2 q - 1\right) + q^{18 g - 6} \left(q - 1\right) \left(2 q + 1\right) + q^{20 g - 6} . \tag*{\qed}
    \end{align*}
\end{theorem}

The computation times\footnote{As performed on an Intel\textregistered Xeon\textregistered CPU E5-4640 0 @ 2.40GHz. Since the computations were performed in parallel (64 cores), both the world time and the CPU time were recorded.} for the TQFT for the groups $G = \tilde{\TT}_n$ are listed in the table below. The diagonalization of $Z(\bdgenus)$ was not taken into account, as this was done manually.
\begin{center}
\begin{tabular}{c|c|c|c|c}
     $n$ & 2 & 3 & 4 & 5 \\ \hline
     world time & 1.92s & 5.17s & 1m19s & 1h38m \\
     CPU time & 2.11s & 31.50s & 28m12s & 50h9m
\end{tabular}
\end{center}

\subsection{Twisted representation varieties}
\label{subsec:twisted_representation_varieties}

One of the advantages of the TQFT method is that it is very flexible. Slightly modifying the setup, one can compute the virtual class of the \emph{twisted representation variety}.

\begin{definition}
    Let $G$ be an algebraic group. A \emph{$G$-parabolic structure} on $\Sigma_g$ is a set $Q$ of tuples $(p_i, \mathcal{C}_i)$ where $p_i \in \Sigma_g$ are distinct points and $\mathcal{C}_i$ are conjugacy classes in $G$. Given a $G$-parabolic structure $Q = \{ (p_1, \mathcal{C}_1), \ldots, (p_r, \mathcal{C}_r) \}$ on $\Sigma_g$, the \emph{twisted $G$-representation variety} of $(\Sigma_g, Q)$ is given by
    \begin{equation}
        \label{eq:explicit_twisted_representation_variety_surface}
        R_G(\Sigma_g, Q) = \left\{ (A_1, B_1, \ldots, A_g, B_g, C_1, \ldots, C_r) \in G^{2g} \times \prod_{i = 1}^{r} \mathcal{C}_i \mid \prod_{i = 1}^{g} [A_i, B_i] \prod_{i = 1}^{r} C_i = 1 \right\} .
\end{equation}
\end{definition}

Indeed, the twisted representation variety parametrizes all group morphisms $\pi_1(\Sigma_g \setminus \{ p_1, \ldots, p_r \}) \to G$ such that a small (positively oriented) loop around a puncture $p_i$ is mapped to the conjugacy class $\mathcal{C}_i$. We introduce a new map of $\K(\Var_k)$-modules that can compute the virtual class of $R_G(\Sigma_g, Q)$.

\begin{definition}
    Let $G$ be an algebraic group. For every conjugacy class $\mathcal{C}$ of $G$, define the $\K(\Var_k)$-module morphism
    \[ Z(\underset{\mathcal{C}}{\bdparabolic}) : \K(\Var/G) \to \K(\Var/G), \quad \left[ \begin{tikzcd} X \arrow{d}{f} \\ G \end{tikzcd} \right] \mapsto \left[\begin{tikzcd}[column sep=0em] X \times \mathcal{C} \arrow{d} & (x, c) \arrow[mapsto]{d} \\ G & f(x) c \end{tikzcd}\right] . \]
\end{definition}

\begin{theorem}
    For any algebraic group $G$, integer $g \ge 0$, and $G$-parabolic structure \\ $Q = \{ (p_1, \mathcal{C}_1), \ldots (p_r, \mathcal{C}_r) \}$ on $\Sigma_g$, we have
    \[ [R_G(\Sigma_g, Q)] = Z(\bdcupright) \circ Z(\underset{\mathcal{C}_r}{\bdparabolic}) \circ \cdots \circ Z(\underset{\mathcal{C}_1}{\bdparabolic}) \circ Z(\bdgenus)^g \circ Z(\bdcupleft)(1) . \]
\end{theorem}

\begin{proof}
    This follows from the explicit form of the twisted $G$-representation variety of $\Sigma_g$ \eqref{eq:explicit_twisted_representation_variety_surface}.
\end{proof}

Interestingly, it turns out that for $G = \tilde{\TT}_n$ the matrix corresponding to $Z(\underset{\mathcal{U}_k}{\bdparabolic})$, with respect to the generators given by the unipotent conjugacy classes, is completely determined by the coefficients $F_{ijk}$ computed in Subsection \ref{subsec:tqft_computing_Z}. Namely,
\begin{align*}
    c_! Z(\underset{\mathcal{U}_k}{\bdparabolic})(\textbf{1}_j)|_{\mathcal{U}_i}
    &= \left[\left\{ (g, h) \in \mathcal{U}_j \times \mathcal{U}_k \mid g h \in \mathcal{U}_i \right\}\right] \\
    &= \sum_{k = 1}^M \left[\left\{ g \in \mathcal{U}_j \mid g \xi_k \in \mathcal{U}_i \right\}\right] [\mathcal{U}_k] \\
    &= \sum_{k = 1}^{M} F_{ijk} \, [\mathcal{U}_k]
\end{align*}
where in the second equality we used Corollary \ref{cor:fiber_product_over_representative}. Using Corollary \ref{cor:fibration_over_representative}, we find
\[ Z(\underset{\mathcal{U}_k}{\bdparabolic})(\textbf{1}_j) = \sum_{i = 1}^{M} F_{ijk} \, [\mathcal{U}_k] / [\mathcal{U}_i] \cdot \textbf{1}_i . \]







    
    


%% file: arithmetic_method.tex

\section{Arithmetic method}
\label{sec:arithmetic_method}

The \emph{arithmetic method} was originally introduced in \cite{HauselRodriguezVillegas2008} in order to compute the $E$-polynomial of the $\GL_n$-representation varieties and the (twisted) $\GL_n$-character varieties of $\Sigma_g$. Basically, the method consists of two ingredients. One of them is Katz' theorem \cite[Theorem 6.1.2]{HauselRodriguezVillegas2008}, which reduces the computation of the $E$-polynomial to the counting of points over finite fields.
\begin{theorem}[Katz' theorem]
    \label{thm:katz_theorem}
    Let $X$ be a complex variety and $\tilde{X}$ a spreading-out of $X$ over a finitely generated $\ZZ$-algebra $R \subset \CC$.
    If there exists a polynomial $P \in \ZZ[q]$ such that $|(\tilde{X} \times_R \FF_q)(\FF_q)| = P(q)$ for all ring morphisms $R \to \FF_q$, then the $E$-polynomial of $X$ is given by $P(uv) \in \ZZ[u, v]$.
\end{theorem}

\begin{remark}
    For all varieties we consider, natural spreading-outs can be chosen over $R = \ZZ$.
\end{remark}

The second ingredient is Frobenius' formula \cite[Equation 2.3.8]{HauselRodriguezVillegas2008}, which relates the number of points of the representation variety of $\Sigma_g$ over finite fields to the representation theory of $G$.
\begin{theorem}[Frobenius' formula]
    \label{thm:frobenius_formula}
    Let $G$ be a finite group. The number of points of the $G$-representation variety of the closed surface $\Sigma_g$ is given by
    \[ |R_G(\Sigma_g)| = |G| \sum_{\chi \in \widehat{G}} \left( \frac{|G|}{\chi(1)} \right)^{2g - 2} , \]
    where $\widehat{G}$ denotes the set of irreducible complex characters of $G$.
\end{theorem}

Therefore, in this section we will study the representation theory of the groups of upper triangular matrices $\TT_n$ and unipotent matrices $\UU_n$ over finite fields $\FF_q$.
In particular, we will describe an algorithm to compute the representation zeta function of these groups, by means of semidirect decompositions.

\subsection{Representation zeta functions}

\begin{definition}
    Let $G$ be a finite group. The \emph{representation zeta function} of $G$ is the function
    \[ \zeta_G(s) = \sum_{\chi \in \widehat{G}} \chi(1)^{-s} . \]
\end{definition}

\begin{example}
    \begin{itemize}
        \item The cyclic group $\ZZ/n\ZZ$ has $n$ irreducible representations of degree $1$, so we have $\zeta_{\ZZ/n\ZZ}(s) = n$.
        \item The symmetric group $S_3$ has two irreducible representations of degree $1$, and one of degree $2$, so $\zeta_{S_3}(s) = 2 + 2^{-s}$.
        \item For finite groups $G$ and $H$, the irreducible representations of the product $G \times H$ are of the form $\rho \otimes \tau$, where $\rho$ and $\tau$ are irreducible representations of $G$ and $H$, respectively. Hence, $\zeta_{G \times H}(s) = \zeta_G(s) \, \zeta_H(s)$.
    \end{itemize}
\end{example}

\begin{example}
    Evaluating $\zeta_G(s)$ in $s = 0$ yields the number of irreducible representations, which is equal to the number of conjugacy classes of $G$.
\end{example}

From Theorem \ref{thm:frobenius_formula} we now immediately obtain the following corollary.
\begin{corollary}
    \label{cor:point_count_from_zeta_function}
    Let $G$ be a finite group. The number of points of the $G$-representation variety of the closed surface $\Sigma_g$ is given by
    \[ |R_G(\Sigma_g)| = |G|^{2g - 1} \zeta_G(\chi(\Sigma_g)) , \]
    where $\chi(\Sigma_g) = 2g - 2$ denotes the Euler characteristic of $\Sigma_g$. \qed
\end{corollary}

\subsection{Semidirect products}

Consider a finite group $G = N \rtimes H$, with $N \subset G$ an abelian normal subgroup. As $N$ is abelian, its irreducible representations are one-dimensional and given by $X = \Hom(N, \CC^*)$. The group $G$ acts on $X$ via
\[ (g \cdot \chi)(n) = \chi(g^{-1} n g) \qquad \textup{ for all } \chi \in X, \; g \in G \textup{ and } n \in N . \]
Let $(\chi_i)_{i \in X/H}$ be a family of representatives for the orbits in $X$ under $H$. For each $i \in X/H$, let $H_i = \{ h \in H \mid h \cdot \chi_i = \chi_i \}$ denote the stabilizer of $\chi_i$, and let $G_i = N \rtimes H_i \subset G$ be the corresponding subgroup of $G$. We can extend $\chi_i$ to $G_i$ by setting
\[ \chi_i(nh) = \chi_i(n) \qquad \textup{ for all } n \in N \textup{ and } h \in H_i . \]
Indeed, this defines a character since $\chi_i(n_1 h_1 n_2 h_2) = \chi_i(n_1 (h_1 n_2 h_1^{-1}) h_1 h_2) = \chi_i(n_1 h_1 n_2 h_1^{-1}) = \chi_i(n_1) \chi_i(n_2)$ for all $n_1, n_2 \in N$ and $h_1, h_2 \in H_i$. Now, any irreducible representation $\rho$ of $H_i$ induces a representation $\tilde{\rho}$ of $G_i$ by composing with the projection $G_i \to G_i / N = H_i$, and we define
\[ \theta_{i, \rho} = \textup{Ind}_{G_i}^G\left( \chi_i \otimes \tilde{\rho} \right) . \]
It turns out that these are precisely all the irreducible representations of $G$.
\begin{proposition}[{\cite[Proposition 25]{Serre1997}}]
    \label{prop:representations_semidirect_product}
    \begin{enumerate}[(i)]
        \item $\theta_{i, \rho}$ is irreducible.
        \item If $\theta_{i, \rho}$ is isomorphic to $\theta_{i', \rho'}$, then $i = i'$ and $\rho$ is isomorphic to $\rho'$.
        \item Every irreducible representation of $G$ is isomorphic to some $\theta_{i, \rho}$.
    \end{enumerate}
\end{proposition}

In terms of representation zeta functions, this proposition translates to the following corollary, using the fact that $\dim \big( \textup{Ind}_H^G(\rho) \big) = \dim(\rho) \cdot [G : H]$.

\begin{corollary}
    \label{cor:zeta_function_semidirect_product}
    The representation zeta function of $G$ is given by
    \[ \zeta_G(s) = \sum_{i \in X/H} \zeta_{H_i}(s) \cdot [G : G_i]^{-s} = \sum_{i \in X/H} \zeta_{H_i}(s) \cdot [H : H_i]^{-s} . \tag*{\qed} \]
\end{corollary}

\subsection{Upper triangular matrices}
\label{subsec:arithmetic_upper_triangular}

Consider the group $\UU_n(\FF_q) = \{ A \in \textup{GL}_n(\FF_q) \mid A_{ii} = 1 \textup{ and } A_{ij} = 0 \textup{ for } i > j \}$ of $n \times n$ unipotent matrices over a finite field $\FF_q$. Define the subgroup $N \subset \UU_n(\FF_q)$ as the kernel of
\[ \UU_n(\FF_q) \to \UU_{n - 1}(\FF_q), \quad A \mapsto (A_{ij})_{i, j = 1}^{n - 1} , \]
so that the quotient $\UU_n(\FF_q) / N$ is isomorphic to $\UU_{n - 1}(\FF_q)$. Now we have a split exact sequence
\[ \begin{tikzcd} 0 \arrow{r} & N \arrow{r} & \UU_n(\FF_q) \arrow{r} & \UU_{n - 1}(\FF_q) \arrow{r} \arrow[bend right = 25]{l} & 0 , \end{tikzcd} \]
which yields a semidirect decomposition $\UU_n(\FF_q) = N \rtimes \UU_{n - 1}(\FF_q)$, where $N$ is abelian.
Moreover, for any unipotent subgroup $U \subset \UU_n(\FF_q)$, the above exact sequence can be intersected with $U$ to obtain
\[ \begin{tikzcd} 0 \arrow{r} & U \cap N \arrow{r} & U \arrow{r} & U \cap \UU_{n - 1}(\FF_q) \arrow{r} \arrow[bend right = 25]{l} & 0 , \end{tikzcd} \]
yielding a semidirect decomposition $U = (U \cap N) \rtimes (U \cap \UU_{n - 1}(\FF_q))$.

Identifying $N \cong \GG_a^{n - 1}(\FF_q)$, the irreducible characters $\chi_\alpha \in X = \Hom(N, \CC^*)$ of $N$ are given by tuples $\alpha = (\alpha_1, \ldots, \alpha_{n - 1}) \in \GG_a^{n - 1}(\FF_q)$, via
\[ \chi_\alpha(x) = \zeta_p^{\langle \alpha, x \rangle} \quad \textup{ for all } x \in N , \]
where $\langle \alpha, x \rangle = \sum_{i, j} \alpha_{ij} x_{ij} \in \FF_p$, with $\alpha_{ij}$ and $x_{ij}$ are the coefficients of $\alpha_i$ and $x_i$, viewing $\GG_a^{n - 1}(\FF_q)$ as vector space over $\FF_p$.
Since $\UU_{n - 1}(\FF_q)$ acts on $N \cong \GG_a^{n - 1}(\FF_q)$ by left multiplication, it acts on $X \cong \GG_a^{n - 1}(\FF_q)$ by right multiplication.

From now on, we will omit the field $\FF_q$ from the group, simply writing $G$ instead of $G(\FF_q)$.


\begin{example}
    \label{ex:zeta_U3}
    Consider $\UU_3 \cong \GG_a^2 \rtimes \UU_2$, for which $X = \Hom(\GG_a^2, \CC^*) \cong \GG_a^2$, and $H$ acts on $\begin{pmatrix} \alpha & \beta \end{pmatrix} \in X$ by right-multiplication, that is,
    \[ \begin{pmatrix} \alpha & \beta \end{pmatrix} \begin{pmatrix} 1 & a \\ 0 & 1 \end{pmatrix} = \begin{pmatrix} \alpha & \beta + a \alpha \end{pmatrix} . \]
    Hence, the orbits in $X$ under $H$ are given by $\left\{ \begin{pmatrix} \alpha & \beta \end{pmatrix} : \beta \in \FF_q \right\}$ for all $\alpha \ne 0$ and $\left\{ \begin{pmatrix} 0 & \beta \end{pmatrix} \right\}$ for all $\beta \in \FF_q$.
    We choose the following representatives:
    \begin{itemize}
        \item $\chi_\alpha = \begin{pmatrix} \alpha & 0 \end{pmatrix}$, which yields $H_\alpha = \{ 1 \}$, so we get a contribution to the zeta function equal to
        \[ (q - 1) \cdot \zeta_{\{ 1 \}}(s) \cdot [H : H_\alpha]^{-s} = (q - 1) q^{-s} . \]
        \item $\chi_\beta = \begin{pmatrix} 0 & \beta \end{pmatrix}$, which yields $H_\beta = \UU_2$, so we get a contribution to the zeta function equal to
        \[ q \cdot \zeta_{\UU_2}(s) \cdot [H : H_\beta]^{-s} = q^2 . \]
    \end{itemize}
    Adding up the contributions, it follows that $\zeta_{\UU_3}(s) = q^2 + (q - 1) q^{-s}$.
\end{example}

\begin{example}
    \label{ex:zeta_U4}
    Consider $\UU_4 \cong \GG_a^3 \rtimes \UU_3$, for which $X = \Hom(\GG_a^3, \CC^*) \cong \GG_a^3$, and $H$ acts on $\begin{pmatrix} \alpha & \beta & \gamma \end{pmatrix} \in X$ by right-multiplication, that is,
    \[ \begin{pmatrix} \alpha & \beta & \gamma \end{pmatrix} \begin{pmatrix} 1 & a & b \\ 0 & 1 & c \\ 0 & 0 & 1 \end{pmatrix} = \begin{pmatrix} \alpha & \beta + a \alpha & \gamma + b \alpha + c \beta \end{pmatrix} . \]
    Hence, the orbits in $X$ under $H$ are given by $\left\{ \begin{pmatrix} \alpha & \beta & \gamma \end{pmatrix} : \beta, \gamma \in \FF_q \right\}$ for all $\alpha \ne 0$, $\left\{ \begin{pmatrix} 0 & \beta & \gamma \end{pmatrix} : \gamma \in \FF_q  \right\}$ for all $\beta \ne 0$, and $\left\{ \begin{pmatrix} 0 & 0 & \gamma \end{pmatrix} \right\}$ for all $\gamma \in \FF_q$.
    We choose the following representatives:
    \begin{itemize}
        \item $\chi_\alpha = \begin{pmatrix} \alpha & 0 & 0 \end{pmatrix}$ yields $H_\alpha \cong \GG_a$, contributing $(q - 1) \cdot \zeta_{\GG_a}(s) \cdot [H : H_\alpha]^{-s} = q^{1 - 2s} (q - 1)$.
        \item $\chi_\beta = \begin{pmatrix} 0 & \beta & 0 \end{pmatrix}$ yields $H_\beta \cong \GG_a^2$, contributing $(q - 1) \cdot \zeta_{\GG_a^2}(s) \cdot [H : H_\beta]^{-s} = q^{2 - s} (q - 1)$.
        \item $\chi_\gamma = \begin{pmatrix} 0 & 0 & \gamma \end{pmatrix}$ yields $H_\gamma = \UU_3$, contributing $q \cdot \zeta_{\UU_3} \cdot [H : H_\gamma]^{-s} = q^3 + (q - 1) q^{1 - s}$.
    \end{itemize}
    In total, $\zeta_{\UU_4}(s) = q^3 + q^{1 - s} (q - 1) (q + 1) + q^{1 - 2s} (q - 1)$.
\end{example}

The construction as described above works more generally for a connected algebraic subgroup $G \subset \TT_n$. Namely, let $G'$ be the image of the map $G \to \tilde{\TT}_n$ given by $A \mapsto A / A_{nn}$. Then either $G \cong G'$ or $G \cong \GG_m \times G'$, because the only connected subgroups of $\GG_m$ are $\{ 1 \}$ and $\GG_m$ itself. Since $\zeta_{\GG_m}(s) = q - 1$ is known, we may assume $G \subset \tilde{\TT}_n$. Now let $H$ and $N$ be the image and kernel, respectively, of the map
\[ G \to \TT_{n - 1}, \quad A \mapsto (A_{ij})_{i, j = 1}^{n - 1} , \]
so that $G = N \rtimes H$ with $N$ abelian and $H \subset \TT_{n - 1}$, and we can repeat the above arguments.

\begin{example}
    Consider $G = \TT_2 \cong \GG_m \times \tilde{\TT}_2$ with $\tilde{\TT}_2 \cong \GG_a \rtimes \GG_m$, for which $X = \Hom(\GG_a, \CC^*) \cong \GG_a$ and $H$ acts on $\alpha \in X$ by multiplication. Hence, the orbits in $X$ under $H$ are given by $\{ 0 \}$ and $\{ \alpha : \alpha \ne 0 \}$. We choose the following representatives:
    \begin{itemize}
        \item $\chi_0 = 0$ yields $H_0 = \GG_m$, contributing $\zeta_{\GG_m}(s) = q - 1$,
        \item $\chi_1 = 1$ yields $H_1 = \{ 1 \}$, contributing $(q - 1)^{-s}$.
    \end{itemize}
    In total, $\zeta_{\TT_2}(s) = \zeta_{\GG_m}(s) \, \zeta_{\tilde{\TT}_2}(s) = (q - 1)((q - 1) + (q - 1)^{-s}) = (q - 1)^2 + (q - 1)^{1 - s}$.
\end{example}

From these examples we see that the representation zeta functions can be computed in a recursive manner using Proposition \ref{prop:representations_semidirect_product}. We will treat one more example, in order to illustrate that the stabilizers $H_i$ need not be constant along families of representatives. This suggests that one needs to consider parametrized families of algebraic groups $H$.

\begin{example}
    \label{ex:zeta_family_stabilizers}
    Consider $G = \GG_a^3 \rtimes H$ with $H = \left\{ \left( \begin{smallmatrix} 1 & 0 & a \\ 0 & 1 & b \\ 0 & 0 & 1 \end{smallmatrix} \right) \right\}$ acting naturally on $\GG_a^3$. Then $H$ acts on $X \cong \GG_a^3$ by
    \[ \begin{pmatrix} \alpha & \beta & \gamma \end{pmatrix} \begin{pmatrix} 1 & 0 & a \\ 0 & 1 & b \\ 0 & 0 & 1 \end{pmatrix} = \begin{pmatrix} \alpha & \beta & \gamma + a \alpha + b \beta \end{pmatrix} . \]
    Hence, the orbits in $X$ under $H$ are given by $\left\{ \begin{pmatrix} 0 & 0 & \gamma \end{pmatrix} \right\}$ for all $\gamma \in \FF_q$ and $\left\{ \begin{pmatrix} \alpha & \beta & \gamma \end{pmatrix} : \gamma \in \FF_q \right\}$ for all $\alpha \ne 0$ or $\beta \ne 0$.
    We choose the following representatives:
    \begin{itemize}
        \item $\chi_\gamma = \begin{pmatrix} 0 & 0 & \gamma \end{pmatrix}$ yields $H_\gamma = H$, contributing $q \cdot \zeta_H(s) = q^3$.
        \item $\chi_{\alpha, \beta} = \begin{pmatrix} \alpha & \beta & 0 \end{pmatrix}$ yields $H_{\alpha, \beta} = \left\{ \smatrix{1 & 0 & x \beta \\ 0 & 1 & -x \alpha \\ 0 & 0 & 1} : x \in \FF_q \right\} \cong \GG_a$, contributing
        \[ (q^2 - 1) \cdot \zeta_{\GG_a}(s) \cdot [H : H_{\alpha, \beta}]^{-s} = q^{1 - s} (q - 1)(q + 1) . \]
    \end{itemize}
    In total, we obtain $\zeta_G(s) = q^3 + q^{1 - s} (q - 1)(q + 1)$.
\end{example}


\subsection[Algorithmically computing ζ\_G(s)]{Algorithmically computing $\zeta_G(s)$}

In this subsection, we will describe an algorithm to compute $\zeta_G(s)$ for connected algebraic groups $G \subset \TT_n$, in the style of examples \ref{ex:zeta_U3}, \ref{ex:zeta_U4} and \ref{ex:zeta_family_stabilizers}
An implementation of this algorithm can be found at \cite{GitHubMathCode}, together with the code for computing $\zeta_{\UU_n}(s)$ and $\zeta_{\TT_n}(s)$ for $1 \le n \le 10$. The resulting zeta functions are given in Theorem \ref{thm:representation_zeta_functions_Tn} and Theorem \ref{thm:representation_zeta_functions_Un}.

Before discussing the algorithm, let us give some remarks.

The algorithm is divided into two parts. The main part, Algorithm \ref{alg:representation_zeta_functions}, finds a semidirect decomposition $G \cong N \rtimes H$ and applies Corollary \ref{cor:zeta_function_semidirect_product} in order to compute $\zeta_G(s)$. Finding representatives for the orbits $X/H$ is a more intricate step, and is described as a separate part, in Algorithm \ref{alg:orbit_representatives}.

As highlighted in Example \ref{ex:zeta_family_stabilizers}, it is possible for the stabilizers $H_i$ to vary along the families of representatives $\chi_i$. Therefore, in order for the algorithm to work recursively, we allow the input of the algorithm to be a family of algebraic groups $G \subset \TT_n$ parametrized by a variety $Y$ over $\FF_q$. We then understand the representation zeta function of $G$ to be
\[ \zeta_G(s) = \sum_{y \in Y(\FF_q)} \zeta_{G_y}(s) . \]
As we want the computations to hold over general a ground field $\FF_q$, we will practically work over $\ZZ$. Then $|Y(\FF_q)|$ can be computed as a polynomial in $q$ whenever $[Y] \in \K(\Var/\ZZ)$ can be computed as a polynomial in $q = [\AA^1_\ZZ]$ using Algorithm \ref{alg:virtual_classes}.

If any decision in the algorithm depends on whether some function $f$ on $Y$ is zero, we distinguish cases and continue working over the closed subvariety $Y' = Y \cap \{ f = 0 \}$ and its open complement $Y'' = Y \cap \{ f \ne 0 \}$.

\begin{algorithm}
    \label{alg:representation_zeta_functions}
    \textbf{Input}: A family of connected algebraic groups $G \subset \TT_n$, parametrized by a variety $Y$.
    
    \textbf{Output}: The representation zeta function $\zeta_G(s)$ as a polynomial in $q$, $q^{-s}$ and $(q - 1)^{-s}$.
    \begin{enumerate}
        \item If $n = 0$, then $G$ is trivial, so that $\zeta_G(s) = |Y(\FF_q)|$. Hence, we can assume $n \ge 1$.
        \item Since $G$ is connected, the image of the map $G \to \GG_m^n$ given by $A \mapsto (A_{ii})_i$ is isomorphic to $\GG_m^d$ for some $0 \le d \le n$. If $d = n$, then there is an isomorphism $G \cong \GG_m \times G'$ with $G' \subset \tilde{\TT}_n$ given by $A \mapsto (A_{nn}, A / A_{nn})$, so that $\zeta_G(s) = (q - 1) \zeta_{G'}(s)$. If $d < n$, then $G \cong G' \subset \tilde{\TT}_n$ via the map $A \mapsto A / A_{nn}$. Either way, we can assume $G \subset \tilde{\TT}_n$. 
        \item Write $G = N \rtimes H$ as discussed in Subsection \ref{subsec:arithmetic_upper_triangular}. The group $H$ can be obtained as the group of minors $H = \left\{ (A_{ij})_{i, j = 1}^{n - 1} : A \in G \right\}$, and $N$ can be obtained as the closed subgroup of $G$ given by $A_{ij} = 0$ for $1 \le i, j \le n - 1$ and $A_{ii} = 1$ for $1 \le i \le n - 1$.
        \item Identify $N \cong \GG_a^r$ for some $0 \le r \le n - 1$, and consider induced action of $H$ on the characters $\Hom(N, \CC^*) \cong \GG_a^r$.
        \item Use Algorithm \ref{alg:orbit_representatives} to find families of representatives $\chi_i$, parametrized by varieties $Z_i$, for the orbits $X/H$, together with their stabilizer $H_i$ and index $[H : H_i]$. Note that the stabilizers $H_i$ are generally parametrized by $X \times Z_i$.
        \item Repeat the algorithm to compute $\zeta_{H_i}(s)$ for all $i$, from which $\zeta_G(s)$ can be computed using Corollary \ref{cor:zeta_function_semidirect_product}.
    \end{enumerate}
\end{algorithm}

\begin{algorithm}
    \label{alg:orbit_representatives}
    \textbf{Input}: Let $H \subset \TT_n$ be an algebraic group, parametrized by a variety $X$, acting linearly on a subvariety $Z \subset \GG_a^r$ of the form $Z = \prod_{i = 1}^{r} Z_i$ with $Z_i \in \{ \{ 0 \}, \{ 1 \}, \GG_a \}$. Write $z_1, \ldots, z_r$ for the coordinates on $Z$. The variety $X$ is assumed to be a variety over $Z$.
    
    \textbf{Output}: Families of representatives $\chi_i$ together with stabilizers $H_i$, and the index $[H : H_i]$ as a polynomial in $q$.
    \begin{enumerate}
        \item Repeat steps $2$ and $3$ until every $z_i$ is invariant under $H$. When this is the case, return a single family of representatives given by $\chi = (z_1, \ldots, z_r)$, parametrized over $Z$, with stabilizer $H$ and index $[H : H] = 1$. If both step $2$ and $3$ do not apply, return failure.
        \item If $z_i \overset{H}{\mapsto} a z_i$ for some coordinate $a$ of $H$, then $a$ must be a diagonal entry of $H$. Distinguish between the following cases:
        \begin{enumerate}
            \item If $z_i = 0$, then continue with the action of $H$ restricted to $Z' = Z \cap \{ z_i = 0 \}$.
            \item If $z_i \ne 0$, then choose representatives with $z_i = 1$ using an appropriate choice of $a$. Continue with the action of $H' = H \cap \{ a = 1 \}$ restricted to $Z' = Z \cap \{ z_i = 1 \}$, remembering the index $[H : H'] = q - 1$.
        \end{enumerate}
        \item Write $z_i \overset{H}{\mapsto} \sum_{j} a_j f_j$, where $a_j$ are coordinates of $H$ and $f_j$ are functions on $X$ which are not identically zero. If some $f_\ell$ is invariant under the action of $H$, then distinguish between the following cases:
        \begin{enumerate}
            \item If $f_\ell = 0$, then continue with the restriction to the closed subvariety $X' = X \cap \{ f_\ell = 0 \}$.
            \item If $f_\ell \ne 0$, then choose representatives with $z_i = 0$ using an appropriate choice of $a_\ell$. Continue with the action of $H' = H \cap \left\{ a_\ell = - f_\ell^{-1} \sum_{j \ne \ell} a_j f_j \right\}$ restricted to $Z' = Z \cap \{ z_i = 0 \}$, remembering the index $[H : H'] = q$.
        \end{enumerate}
    \end{enumerate}
\end{algorithm}


\begin{remark}
    Note that some steps in this algorithm might fail. In fact, if this algorithm were to never fail, then this would show the representation zeta function $\zeta_G(s)$ is always a polynomial in $q$, $q^{-s}$ and $(q - 1)^{-s}$. Then, evaluating at $s = 0$ would imply that the number of conjugacy classes of $G$ is a polynomial in $q$. In particular, this would imply Higman's Conjecture \ref{conj:higman}. For us, the algorithm does not fail when applied to $G = \UU_n$ or $G = \TT_n$ for $1 \le n \le 10$.
\end{remark}

\subsection{Results}
\label{subsec:arithmetic_results}

The representation zeta functions of $\UU_n$ and $\TT_n$, as computed using Algorithm \ref{alg:representation_zeta_functions}, are given in Theorem \ref{thm:representation_zeta_functions_Un} and Theorem \ref{thm:representation_zeta_functions_Tn} below.
 
One can evaluate these zeta functions at $s = 0$ in order to obtain the number of conjugacy classes of the groups over finite fields $\FF_q$. The resulting polynomials in $q$ can be seen to agree with \cite[Appendix A]{PakSoffer2015}, where $t = q - 1$. In this sense, these zeta functions can be viewed as a generalization of the polynomials $k(\UU_n(\FF_q))$ as in \cite{PakSoffer2015}.

Furthermore, the $E$-polynomials of $R_G(\Sigma_g)$ can be obtained through Theorem \ref{thm:katz_theorem} and Corollary \ref{cor:point_count_from_zeta_function}. Indeed, one can verify that for $1 \le n \le 5$ these $E$-polynomials agree with the virtual classes as given by Theorem \ref{thm:virtual_class_Un_representation_varieties} and Theorem \ref{thm:virtual_class_T5_representation_variety}, via the map \eqref{eq:grothendieck_ring_to_Zuv}.

\begin{theorem}
    \label{thm:representation_zeta_functions_Un}
    The representation zeta functions $\zeta_{\UU_n}(s)$ for $1 \le n \le 10$ are given by
    \allowdisplaybreaks
    \footnotesize
    \begin{align*}
\zeta_{\UU_{1}}(s) &= 1 \\
\zeta_{\UU_{2}}(s) &= q \\
\zeta_{\UU_{3}}(s) &= q^{- s} \left(q - 1\right) + q^{2} \\
\zeta_{\UU_{4}}(s) &= q^{1 - s} \left(q - 1\right) \left(q + 1\right) + q^{1 - 2 s} \left(q - 1\right) + q^{3} \\
\zeta_{\UU_{5}}(s) &= q^{1 - 2 s} \left(q - 1\right) \left(q + 1\right) \left(2 q - 1\right) + q^{2 - s} \left(q - 1\right) \left(2 q + 1\right) + q^{1 - 3 s} \left(q - 1\right) \left(2 q - 1\right) \\ &\quad +q^{- 4 s} \left(q - 1\right)^{2} + q^{4} \\
\zeta_{\UU_{6}}(s) &= q^{2 - 2 s} \left(q - 1\right) \left(q + 2\right) \left(q^{2} + q - 1\right) + q^{2 - 3 s} \left(q - 1\right) \left(q + 1\right) \left(4 q - 3\right) \\ &\quad +q^{- 4 s} \left(q - 1\right) \left(2 q^{2} - 1\right) \left(q^{2} + q - 1\right) + q^{3 - s} \left(q - 1\right) \left(3 q + 1\right) + q^{1 - 5 s} \left(q - 1\right)^{2} \left(2 q + 1\right) \\ &\quad +q^{1 - 6 s} \left(q - 1\right)^{2} + q^{5} \\
\zeta_{\UU_{7}}(s) &= q^{3 - 2 s} \left(q - 1\right) \left(q + 1\right) \left(2 q^{2} + 3 q - 3\right) \\ &\quad +q^{1 - 4 s} \left(q - 1\right) \left(2 q - 1\right) \left(q^{4} + 5 q^{3} - 3 q - 1\right) + q^{4 - s} \left(q - 1\right) \left(4 q + 1\right) \\ &\quad +q^{2 - 3 s} \left(q - 1\right) \left(3 q^{4} + 6 q^{3} - 2 q^{2} - 5 q + 1\right) \\ &\quad +q^{1 - 5 s} \left(q - 1\right) \left(q^{5} + 7 q^{4} - 2 q^{3} - 9 q^{2} + 3 q + 1\right) + q^{1 - 6 s} \left(q - 1\right)^{2} \left(4 q^{3} + 7 q^{2} - 3 q - 1\right) \\ &\quad +q^{1 - 8 s} \left(q - 1\right)^{2} \left(3 q - 2\right) + q^{- 7 s} \left(q - 1\right)^{2} \left(5 q^{3} - 3 q + 1\right) + q^{- 9 s} \left(q - 1\right)^{3} + q^{6} \\
\zeta_{\UU_{8}}(s) &= q^{4 - 2 s} \left(q - 1\right) \left(3 q + 2\right) \left(q^{2} + 2 q - 2\right) + q^{5 - s} \left(q - 1\right) \left(5 q + 1\right) \\ &\quad +q^{3 - 3 s} \left(q - 1\right) \left(q^{5} + 5 q^{4} + 10 q^{3} - 7 q^{2} - 8 q + 3\right) \\ &\quad +q^{3 - 6 s} \left(q - 1\right) \left(q^{5} + 7 q^{4} + 16 q^{3} - 24 q^{2} - 14 q + 15\right) \\ &\quad +q^{2 - 4 s} \left(q - 1\right) \left(12 q^{5} + 9 q^{4} - 16 q^{3} - 9 q^{2} + 6 q + 1\right) \\ &\quad +q^{1 - 5 s} \left(q - 1\right) \left(2 q^{7} + 8 q^{6} + 13 q^{5} - 23 q^{4} - 9 q^{3} + 12 q^{2} - 1\right) \\ &\quad +q^{1 - 7 s} \left(q - 1\right)^{2} \left(6 q^{5} + 18 q^{4} + 4 q^{3} - 19 q^{2} + q + 3\right) \\ &\quad +q^{1 - 8 s} \left(q - 1\right)^{2} \left(q^{5} + 13 q^{4} + 8 q^{3} - 14 q^{2} - 4 q + 3\right) + q^{1 - 11 s} \left(q - 1\right)^{3} \left(3 q + 1\right) \\ &\quad +q^{- 9 s} \left(q - 1\right)^{2} \left(4 q^{5} + 10 q^{4} - 7 q^{3} - 8 q^{2} + 3 q + 1\right) + q^{- 10 s} \left(q - 1\right)^{2} \left(5 q^{4} + q^{3} - 6 q^{2} + 1\right) \\ &\quad +q^{1 - 12 s} \left(q - 1\right)^{3} + q^{7} \\
\zeta_{\UU_{9}}(s) &= q^{5 - 2 s} \left(q - 1\right) \left(2 q + 1\right) \left(2 q^{2} + 5 q - 5\right) + q^{6 - s} \left(q - 1\right) \left(6 q + 1\right) \\ &\quad +q^{4 - 3 s} \left(q - 1\right) \left(2 q^{5} + 9 q^{4} + 14 q^{3} - 15 q^{2} - 11 q + 6\right) \\ &\quad +q^{4 - 4 s} \left(q - 1\right) \left(4 q^{5} + 19 q^{4} + 11 q^{3} - 34 q^{2} - 10 q + 14\right) \\ &\quad +q^{2 - 5 s} \left(q - 1\right) \left(q^{8} + 5 q^{7} + 29 q^{6} + q^{5} - 53 q^{4} - 2 q^{3} + 27 q^{2} - 3 q - 2\right) \\ &\quad +q^{2 - 6 s} \left(q - 1\right) \left(10 q^{7} + 33 q^{6} - 9 q^{5} - 68 q^{4} + 10 q^{3} + 38 q^{2} - 11 q - 1\right) \\ &\quad +q^{1 - 7 s} \left(q - 1\right) \left(2 q^{9} + 8 q^{8} + 27 q^{7} + 2 q^{6} - 87 q^{5} + 20 q^{4} + 46 q^{3} - 15 q^{2} - 3 q + 1\right) \\ &\quad +q^{1 - 8 s} \left(q - 1\right)^{2} \left(9 q^{7} + 33 q^{6} + 40 q^{5} - 45 q^{4} - 40 q^{3} + 21 q^{2} + 5 q - 1\right) \\ &\quad +q^{1 - 9 s} \left(q - 1\right)^{2} \left(2 q^{7} + 30 q^{6} + 42 q^{5} - 44 q^{4} - 48 q^{3} + 25 q^{2} + 7 q - 1\right) \\ &\quad +q^{1 - 11 s} \left(q - 1\right)^{2} \left(4 q^{6} + 25 q^{5} + 5 q^{4} - 48 q^{3} + 7 q^{2} + 9 q + 1\right) \\ &\quad +q^{1 - 12 s} \left(q - 1\right)^{2} \left(10 q^{5} + 18 q^{4} - 32 q^{3} - 10 q^{2} + 18 q - 3\right) + q^{1 - 15 s} \left(q - 1\right)^{3} \left(4 q - 3\right) \\ &\quad +q^{- 10 s} \left(q - 1\right)^{2} \left(2 q^{8} + 13 q^{7} + 38 q^{6} - 24 q^{5} - 49 q^{4} + 20 q^{3} + 11 q^{2} - 3 q - 1\right) \\ &\quad +q^{- 13 s} \left(q - 1\right)^{3} \left(12 q^{4} + 10 q^{3} - 13 q^{2} + q + 1\right) + q^{- 14 s} \left(q - 1\right)^{3} \left(9 q^{3} - 2 q^{2} - 5 q + 2\right) \\ &\quad +q^{- 16 s} \left(q - 1\right)^{4} + q^{8} \\
\zeta_{\UU_{10}}(s) &= q^{6 - 2 s} \left(q - 1\right) \left(5 q + 2\right) \left(q^{2} + 3 q - 3\right) + q^{7 - s} \left(q - 1\right) \left(7 q + 1\right) \\ &\quad +q^{5 - 3 s} \left(q - 1\right) \left(3 q^{5} + 15 q^{4} + 19 q^{3} - 28 q^{2} - 13 q + 10\right) \\ &\quad +q^{4 - 4 s} \left(q - 1\right) \left(q^{7} + 7 q^{6} + 32 q^{5} + 12 q^{4} - 65 q^{3} - 6 q^{2} + 27 q - 3\right) \\ &\quad +q^{3 - 5 s} \left(q - 1\right) \left(2 q^{8} + 21 q^{7} + 42 q^{6} - 16 q^{5} - 103 q^{4} + 24 q^{3} + 50 q^{2} - 13 q - 3\right) \\ &\quad +q^{2 - 6 s} \left(q - 1\right) \left(6 q^{9} + 27 q^{8} + 64 q^{7} - 73 q^{6} - 118 q^{5} + 64 q^{4} + 70 q^{3} - 39 q^{2} + q + 1\right) \\ &\quad +q^{2 - 7 s} \left(q - 1\right) \left(2 q^{10} + 5 q^{9} + 39 q^{8} + 74 q^{7} - 130 q^{6} - 133 q^{5} + 128 q^{4} + 74 q^{3} - 60 q^{2} + 2 q + 1\right) \\ &\quad +q^{2 - 8 s} \left(q - 1\right) \left(q^{10} + 12 q^{9} + 39 q^{8} + 67 q^{7} - 137 q^{6} - 172 q^{5} + 200 q^{4} + 63 q^{3} - 80 q^{2} + 2 q + 6\right) \\ &\quad +q^{2 - 9 s} \left(q - 1\right)^{2} \left(10 q^{8} + 65 q^{7} + 117 q^{6} - 36 q^{5} - 221 q^{4} + 18 q^{3} + 98 q^{2} - 11 q - 6\right) \\ &\quad +q^{1 - 11 s} \left(q - 1\right)^{2} \left(6 q^{9} + 31 q^{8} + 109 q^{7} + 8 q^{6} - 240 q^{5} - 10 q^{4} + 135 q^{3} - 17 q^{2} - 8 q - 1\right) \\ &\quad +q^{1 - 12 s} \left(q - 1\right)^{2} \left(2 q^{9} + 22 q^{8} + 77 q^{7} + 46 q^{6} - 217 q^{5} - 48 q^{4} + 156 q^{3} - 12 q^{2} - 20 q + 1\right) \\ &\quad +q^{1 - 13 s} \left(q - 1\right)^{2} \left(10 q^{8} + 50 q^{7} + 60 q^{6} - 138 q^{5} - 110 q^{4} + 146 q^{3} + 8 q^{2} - 25 q + 2\right) \\ &\quad +q^{1 - 15 s} \left(q - 1\right)^{3} \left(4 q^{6} + 42 q^{5} + 46 q^{4} - 51 q^{3} - 44 q^{2} + 23 q + 5\right) + q^{1 - 19 s} \left(q - 1\right)^{4} \left(4 q + 1\right) \\ &\quad +q^{- 10 s} \left(q - 1\right)^{2} \left(2 q^{11} + 8 q^{10} + 50 q^{9} + 112 q^{8} - 29 q^{7} - 227 q^{6} + 17 q^{5} + 123 q^{4} - 24 q^{3} - 12 q^{2} + q + 1\right) \\ &\quad +q^{- 14 s} \left(q - 1\right)^{2} \left(2 q^{9} + 24 q^{8} + 53 q^{7} - 52 q^{6} - 127 q^{5} + 84 q^{4} + 49 q^{3} - 32 q^{2} - 3 q + 3\right) \\ &\quad +q^{- 16 s} \left(q - 1\right)^{3} \left(10 q^{6} + 37 q^{5} - 9 q^{4} - 42 q^{3} + 6 q^{2} + 10 q - 1\right) \\ &\quad +q^{- 17 s} \left(q - 1\right)^{3} \left(12 q^{5} + 14 q^{4} - 21 q^{3} - 8 q^{2} + 6 q + 1\right) \\ &\quad +q^{- 18 s} \left(q - 1\right)^{3} \left(9 q^{4} - q^{3} - 9 q^{2} + q + 1\right) + q^{1 - 20 s} \left(q - 1\right)^{4} + q^{9} . \tag*{\qed}
    \end{align*}
\end{theorem}

\begin{theorem}
    \label{thm:representation_zeta_functions_Tn}
    The representation zeta functions $\zeta_{\TT_n}(s)$ for $1 \le n \le 10$ are given by
    \allowdisplaybreaks
    \footnotesize
    \begin{align*}
\zeta_{\TT_{1}}(s) &= q - 1 \\
\zeta_{\TT_{2}}(s) &= \left(q - 1\right)^{1 - s} + \left(q - 1\right)^{2} \\
\zeta_{\TT_{3}}(s) &= q^{- s} \left(q - 1\right)^{2 - s} + 2 \left(q - 1\right)^{2 - s} + \left(q - 1\right)^{1 - 2 s} + \left(q - 1\right)^{3} \\
\zeta_{\TT_{4}}(s) &= 3 q^{- s} \left(q - 1\right)^{2 - 2 s} + 2 q^{- s} \left(q - 1\right)^{3 - s} + q^{- 2 s} \left(q - 1\right)^{3 - s} + q^{- 2 s} \left(q - 1\right)^{2 - 2 s} + q^{- s} \left(q - 1\right)^{1 - 3 s} \\ &\quad +3 \left(q - 1\right)^{3 - s} + 3 \left(q - 1\right)^{2 - 2 s} + \left(q - 1\right)^{1 - 3 s} + \left(q - 1\right)^{4} \\
\zeta_{\TT_{5}}(s) &= 8 q^{- s} \left(q - 1\right)^{3 - 2 s} + 7 q^{- 2 s} \left(q - 1\right)^{3 - 2 s} + 7 q^{- 2 s} \left(q - 1\right)^{2 - 3 s} + 7 q^{- s} \left(q - 1\right)^{2 - 3 s} \\ &\quad +3 q^{- 3 s} \left(q - 1\right)^{3 - 2 s} + 3 q^{- s} \left(q - 1\right)^{4 - s} + 2 q^{- 2 s} \left(q - 1\right)^{4 - s} + 2 q^{- 2 s} \left(q - 1\right)^{1 - 4 s} \\ &\quad +2 q^{- 3 s} \left(q - 1\right)^{2 - 3 s} + 2 q^{- s} \left(q - 1\right)^{1 - 4 s} + q^{- 3 s} \left(q - 1\right)^{4 - s} + q^{- 4 s} \left(q - 1\right)^{3 - 2 s} + 6 \left(q - 1\right)^{3 - 2 s} \\ &\quad +4 \left(q - 1\right)^{4 - s} + 4 \left(q - 1\right)^{2 - 3 s} + \left(q - 1\right)^{1 - 4 s} + \left(q - 1\right)^{5} \\
\zeta_{\TT_{6}}(s) &= q^{- 2 s} \left(q - 1\right)^{1 - 5 s} \left(q + 7\right) + 29 q^{- 2 s} \left(q - 1\right)^{3 - 3 s} + 24 q^{- 3 s} \left(q - 1\right)^{3 - 3 s} + 23 q^{- 2 s} \left(q - 1\right)^{2 - 4 s} \\ &\quad +21 q^{- s} \left(q - 1\right)^{3 - 3 s} + 17 q^{- 3 s} \left(q - 1\right)^{2 - 4 s} + 16 q^{- 2 s} \left(q - 1\right)^{4 - 2 s} + 15 q^{- 4 s} \left(q - 1\right)^{3 - 3 s} \\ &\quad +15 q^{- s} \left(q - 1\right)^{4 - 2 s} + 13 q^{- 3 s} \left(q - 1\right)^{4 - 2 s} + 13 q^{- s} \left(q - 1\right)^{2 - 4 s} + 10 q^{- 4 s} \left(q - 1\right)^{2 - 4 s} \\ &\quad +7 q^{- 4 s} \left(q - 1\right)^{4 - 2 s} + 5 q^{- 5 s} \left(q - 1\right)^{3 - 3 s} + 4 q^{- 3 s} \left(q - 1\right)^{1 - 5 s} + 4 q^{- s} \left(q - 1\right)^{5 - s} \\ &\quad +3 q^{- 2 s} \left(q - 1\right)^{5 - s} + 3 q^{- 5 s} \left(q - 1\right)^{4 - 2 s} + 3 q^{- s} \left(q - 1\right)^{1 - 5 s} + 2 q^{- 3 s} \left(q - 1\right)^{5 - s} \\ &\quad +2 q^{- 4 s} \left(q - 1\right)^{1 - 5 s} + 2 q^{- 5 s} \left(q - 1\right)^{2 - 4 s} + q^{- 4 s} \left(q - 1\right)^{5 - s} + q^{- 6 s} \left(q - 1\right)^{4 - 2 s} \\ &\quad +q^{- 6 s} \left(q - 1\right)^{3 - 3 s} + 10 \left(q - 1\right)^{4 - 2 s} + 10 \left(q - 1\right)^{3 - 3 s} + 5 \left(q - 1\right)^{5 - s} + 5 \left(q - 1\right)^{2 - 4 s} + \left(q - 1\right)^{1 - 5 s} \\ &\quad +\left(q - 1\right)^{6} \\
\zeta_{\TT_{7}}(s) &= 2 q^{- 4 s} \left(q - 1\right)^{1 - 6 s} \left(q + 8\right) + q^{- 2 s} \left(q - 1\right)^{2 - 5 s} \left(2 q + 53\right) + q^{- 2 s} \left(q - 1\right)^{1 - 6 s} \left(2 q + 13\right) \\ &\quad +q^{- 3 s} \left(q - 1\right)^{2 - 5 s} \left(2 q + 71\right) + q^{- 3 s} \left(q - 1\right)^{1 - 6 s} \left(3 q + 19\right) + q^{- 4 s} \left(q - 1\right)^{2 - 5 s} \left(3 q + 67\right) \\ &\quad +q^{- 5 s} \left(q - 1\right)^{1 - 6 s} \left(q + 12\right) + 107 q^{- 3 s} \left(q - 1\right)^{3 - 4 s} + 104 q^{- 4 s} \left(q - 1\right)^{3 - 4 s} + 87 q^{- 2 s} \left(q - 1\right)^{3 - 4 s} \\ &\quad +79 q^{- 3 s} \left(q - 1\right)^{4 - 3 s} + 73 q^{- 4 s} \left(q - 1\right)^{4 - 3 s} + 73 q^{- 5 s} \left(q - 1\right)^{3 - 4 s} + 71 q^{- 2 s} \left(q - 1\right)^{4 - 3 s} \\ &\quad +49 q^{- 5 s} \left(q - 1\right)^{4 - 3 s} + 48 q^{- 5 s} \left(q - 1\right)^{2 - 5 s} + 46 q^{- s} \left(q - 1\right)^{4 - 3 s} + 44 q^{- s} \left(q - 1\right)^{3 - 4 s} \\ &\quad +42 q^{- 6 s} \left(q - 1\right)^{3 - 4 s} + 30 q^{- 6 s} \left(q - 1\right)^{4 - 3 s} + 28 q^{- 2 s} \left(q - 1\right)^{5 - 2 s} + 27 q^{- 3 s} \left(q - 1\right)^{5 - 2 s} \\ &\quad +24 q^{- s} \left(q - 1\right)^{5 - 2 s} + 23 q^{- 6 s} \left(q - 1\right)^{2 - 5 s} + 22 q^{- 4 s} \left(q - 1\right)^{5 - 2 s} + 21 q^{- s} \left(q - 1\right)^{2 - 5 s} \\ &\quad +15 q^{- 7 s} \left(q - 1\right)^{3 - 4 s} + 13 q^{- 5 s} \left(q - 1\right)^{5 - 2 s} + 12 q^{- 7 s} \left(q - 1\right)^{4 - 3 s} + 7 q^{- 6 s} \left(q - 1\right)^{5 - 2 s} \\ &\quad +5 q^{- 7 s} \left(q - 1\right)^{2 - 5 s} + 5 q^{- s} \left(q - 1\right)^{6 - s} + 4 q^{- 2 s} \left(q - 1\right)^{6 - s} + 4 q^{- 6 s} \left(q - 1\right)^{1 - 6 s} \\ &\quad +4 q^{- 8 s} \left(q - 1\right)^{4 - 3 s} + 4 q^{- s} \left(q - 1\right)^{1 - 6 s} + 3 q^{- 3 s} \left(q - 1\right)^{6 - s} + 3 q^{- 7 s} \left(q - 1\right)^{5 - 2 s} \\ &\quad +3 q^{- 8 s} \left(q - 1\right)^{3 - 4 s} + 2 q^{- 4 s} \left(q - 1\right)^{6 - s} + q^{- 5 s} \left(q - 1\right)^{6 - s} + q^{- 8 s} \left(q - 1\right)^{5 - 2 s} + q^{- 9 s} \left(q - 1\right)^{4 - 3 s} \\ &\quad +20 \left(q - 1\right)^{4 - 3 s} + 15 \left(q - 1\right)^{5 - 2 s} + 15 \left(q - 1\right)^{3 - 4 s} + 6 \left(q - 1\right)^{6 - s} + 6 \left(q - 1\right)^{2 - 5 s} + \left(q - 1\right)^{1 - 6 s} \\ &\quad +\left(q - 1\right)^{7} \\
\zeta_{\TT_{8}}(s) &= 14 q^{- 3 s} \left(q - 1\right)^{2 - 6 s} \left(q + 14\right) + 6 q^{- 7 s} \left(q - 1\right)^{1 - 7 s} \left(q + 7\right) + 4 q^{- 3 s} \left(q - 1\right)^{3 - 5 s} \left(q + 87\right) \\ &\quad +2 q^{- 2 s} \left(q - 1\right)^{2 - 6 s} \left(3 q + 50\right) + 2 q^{- 7 s} \left(q - 1\right)^{2 - 6 s} \left(3 q + 89\right) + q^{- 2 s} \left(q - 1\right)^{3 - 5 s} \left(3 q + 208\right) \\ &\quad +q^{- 2 s} \left(q - 1\right)^{1 - 7 s} \left(3 q + 20\right) + q^{- 3 s} \left(q - 1\right)^{1 - 7 s} \left(q^{2} + 11 q + 47\right) \\ &\quad +q^{- 4 s} \left(q - 1\right)^{3 - 5 s} \left(9 q + 457\right) + q^{- 4 s} \left(q - 1\right)^{2 - 6 s} \left(20 q + 261\right) \\ &\quad +q^{- 4 s} \left(q - 1\right)^{1 - 7 s} \left(12 q + 61\right) + q^{- 5 s} \left(q - 1\right)^{3 - 5 s} \left(6 q + 485\right) \\ &\quad +q^{- 5 s} \left(q - 1\right)^{2 - 6 s} \left(24 q + 305\right) + q^{- 5 s} \left(q - 1\right)^{1 - 7 s} \left(2 q^{2} + 20 q + 79\right) \\ &\quad +q^{- 6 s} \left(q - 1\right)^{3 - 5 s} \left(6 q + 415\right) + q^{- 6 s} \left(q - 1\right)^{2 - 6 s} \left(19 q + 250\right) \\ &\quad +q^{- 6 s} \left(q - 1\right)^{1 - 7 s} \left(q^{2} + 13 q + 60\right) + q^{- 8 s} \left(q - 1\right)^{2 - 6 s} \left(3 q + 85\right) + q^{- 8 s} \left(q - 1\right)^{1 - 7 s} \left(q + 15\right) \\ &\quad +410 q^{- 4 s} \left(q - 1\right)^{4 - 4 s} + 398 q^{- 5 s} \left(q - 1\right)^{4 - 4 s} + 340 q^{- 6 s} \left(q - 1\right)^{4 - 4 s} + 332 q^{- 3 s} \left(q - 1\right)^{4 - 4 s} \\ &\quad +297 q^{- 7 s} \left(q - 1\right)^{3 - 5 s} + 238 q^{- 7 s} \left(q - 1\right)^{4 - 4 s} + 229 q^{- 2 s} \left(q - 1\right)^{4 - 4 s} + 192 q^{- 4 s} \left(q - 1\right)^{5 - 3 s} \\ &\quad +174 q^{- 3 s} \left(q - 1\right)^{5 - 3 s} + 171 q^{- 5 s} \left(q - 1\right)^{5 - 3 s} + 168 q^{- 8 s} \left(q - 1\right)^{3 - 5 s} + 147 q^{- 8 s} \left(q - 1\right)^{4 - 4 s} \\ &\quad +139 q^{- 2 s} \left(q - 1\right)^{5 - 3 s} + 136 q^{- 6 s} \left(q - 1\right)^{5 - 3 s} + 110 q^{- s} \left(q - 1\right)^{4 - 4 s} + 90 q^{- 7 s} \left(q - 1\right)^{5 - 3 s} \\ &\quad +85 q^{- s} \left(q - 1\right)^{5 - 3 s} + 80 q^{- s} \left(q - 1\right)^{3 - 5 s} + 73 q^{- 9 s} \left(q - 1\right)^{3 - 5 s} + 71 q^{- 9 s} \left(q - 1\right)^{4 - 4 s} \\ &\quad +56 q^{- 8 s} \left(q - 1\right)^{5 - 3 s} + 45 q^{- 3 s} \left(q - 1\right)^{6 - 2 s} + 43 q^{- 2 s} \left(q - 1\right)^{6 - 2 s} + 42 q^{- 4 s} \left(q - 1\right)^{6 - 2 s} \\ &\quad +35 q^{- s} \left(q - 1\right)^{6 - 2 s} + 34 q^{- 5 s} \left(q - 1\right)^{6 - 2 s} + 31 q^{- s} \left(q - 1\right)^{2 - 6 s} + 30 q^{- 9 s} \left(q - 1\right)^{2 - 6 s} \\ &\quad +27 q^{- 10 s} \left(q - 1\right)^{4 - 4 s} + 26 q^{- 9 s} \left(q - 1\right)^{5 - 3 s} + 22 q^{- 6 s} \left(q - 1\right)^{6 - 2 s} + 21 q^{- 10 s} \left(q - 1\right)^{3 - 5 s} \\ &\quad +13 q^{- 7 s} \left(q - 1\right)^{6 - 2 s} + 11 q^{- 10 s} \left(q - 1\right)^{5 - 3 s} + 7 q^{- 8 s} \left(q - 1\right)^{6 - 2 s} + 7 q^{- 11 s} \left(q - 1\right)^{4 - 4 s} \\ &\quad +6 q^{- s} \left(q - 1\right)^{7 - s} + 5 q^{- 2 s} \left(q - 1\right)^{7 - s} + 5 q^{- 10 s} \left(q - 1\right)^{2 - 6 s} + 5 q^{- s} \left(q - 1\right)^{1 - 7 s} \\ &\quad +4 q^{- 3 s} \left(q - 1\right)^{7 - s} + 4 q^{- 9 s} \left(q - 1\right)^{1 - 7 s} + 4 q^{- 11 s} \left(q - 1\right)^{5 - 3 s} + 3 q^{- 4 s} \left(q - 1\right)^{7 - s} \\ &\quad +3 q^{- 9 s} \left(q - 1\right)^{6 - 2 s} + 3 q^{- 11 s} \left(q - 1\right)^{3 - 5 s} + 2 q^{- 5 s} \left(q - 1\right)^{7 - s} + q^{- 6 s} \left(q - 1\right)^{7 - s} \\ &\quad +q^{- 10 s} \left(q - 1\right)^{6 - 2 s} + q^{- 12 s} \left(q - 1\right)^{5 - 3 s} + q^{- 12 s} \left(q - 1\right)^{4 - 4 s} + 35 \left(q - 1\right)^{5 - 3 s} + 35 \left(q - 1\right)^{4 - 4 s} \\ &\quad +21 \left(q - 1\right)^{6 - 2 s} + 21 \left(q - 1\right)^{3 - 5 s} + 7 \left(q - 1\right)^{7 - s} + 7 \left(q - 1\right)^{2 - 6 s} + \left(q - 1\right)^{1 - 7 s} + \left(q - 1\right)^{8} \\
\zeta_{\TT_{9}}(s) &= 10 q^{- 8 s} \left(q - 1\right)^{4 - 5 s} \left(q + 186\right) + 10 q^{- 8 s} \left(q - 1\right)^{3 - 6 s} \left(7 q + 204\right) \\ &\quad +6 q^{- 3 s} \left(q - 1\right)^{4 - 5 s} \left(q + 182\right) + 6 q^{- 11 s} \left(q - 1\right)^{3 - 6 s} \left(q + 81\right) + 4 q^{- 2 s} \left(q - 1\right)^{4 - 5 s} \left(q + 149\right) \\ &\quad +4 q^{- 2 s} \left(q - 1\right)^{1 - 8 s} \left(q + 7\right) + 4 q^{- 6 s} \left(q - 1\right)^{3 - 6 s} \left(27 q + 598\right) \\ &\quad +4 q^{- 9 s} \left(q - 1\right)^{3 - 6 s} \left(11 q + 377\right) + 3 q^{- 7 s} \left(q - 1\right)^{4 - 5 s} \left(4 q + 739\right) \\ &\quad +2 q^{- 3 s} \left(q - 1\right)^{2 - 7 s} \left(q^{2} + 23 q + 212\right) + 2 q^{- 4 s} \left(q - 1\right)^{3 - 6 s} \left(31 q + 754\right) \\ &\quad +2 q^{- 7 s} \left(q - 1\right)^{2 - 7 s} \left(4 q^{2} + 103 q + 730\right) + 2 q^{- 8 s} \left(q - 1\right)^{2 - 7 s} \left(3 q^{2} + 72 q + 586\right) \\ &\quad +2 q^{- 10 s} \left(q - 1\right)^{3 - 6 s} \left(7 q + 491\right) + 2 q^{- 11 s} \left(q - 1\right)^{1 - 8 s} \left(2 q + 19\right) \\ &\quad +2 q^{- 12 s} \left(q - 1\right)^{2 - 7 s} \left(q + 38\right) + q^{- 2 s} \left(q - 1\right)^{3 - 6 s} \left(12 q + 425\right) \\ &\quad +q^{- 2 s} \left(q - 1\right)^{2 - 7 s} \left(12 q + 167\right) + q^{- 3 s} \left(q - 1\right)^{3 - 6 s} \left(31 q + 912\right) \\ &\quad +q^{- 3 s} \left(q - 1\right)^{1 - 8 s} \left(2 q^{2} + 21 q + 85\right) + q^{- 4 s} \left(q - 1\right)^{4 - 5 s} \left(15 q + 1658\right) \\ &\quad +q^{- 4 s} \left(q - 1\right)^{2 - 7 s} \left(2 q^{2} + 89 q + 758\right) + q^{- 4 s} \left(q - 1\right)^{1 - 8 s} \left(4 q^{2} + 45 q + 165\right) \\ &\quad +q^{- 5 s} \left(q - 1\right)^{4 - 5 s} \left(16 q + 2111\right) + q^{- 5 s} \left(q - 1\right)^{3 - 6 s} \left(92 q + 2099\right) \\ &\quad +q^{- 5 s} \left(q - 1\right)^{2 - 7 s} \left(9 q^{2} + 157 q + 1138\right) \\ &\quad +q^{- 5 s} \left(q - 1\right)^{1 - 8 s} \left(q^{3} + 12 q^{2} + 83 q + 262\right) + q^{- 6 s} \left(q - 1\right)^{4 - 5 s} \left(21 q + 2302\right) \\ &\quad +q^{- 6 s} \left(q - 1\right)^{2 - 7 s} \left(6 q^{2} + 180 q + 1339\right) + q^{- 6 s} \left(q - 1\right)^{1 - 8 s} \left(10 q^{2} + 97 q + 316\right) \\ &\quad +q^{- 7 s} \left(q - 1\right)^{3 - 6 s} \left(101 q + 2451\right) \\ &\quad +q^{- 7 s} \left(q - 1\right)^{1 - 8 s} \left(2 q^{3} + 22 q^{2} + 131 q + 369\right) + q^{- 8 s} \left(q - 1\right)^{1 - 8 s} \left(9 q^{2} + 81 q + 277\right) \\ &\quad +q^{- 9 s} \left(q - 1\right)^{2 - 7 s} \left(3 q^{2} + 80 q + 823\right) + q^{- 9 s} \left(q - 1\right)^{1 - 8 s} \left(2 q^{2} + 39 q + 181\right) \\ &\quad +q^{- 10 s} \left(q - 1\right)^{2 - 7 s} \left(39 q + 536\right) + q^{- 10 s} \left(q - 1\right)^{1 - 8 s} \left(2 q^{2} + 25 q + 119\right) \\ &\quad +q^{- 11 s} \left(q - 1\right)^{2 - 7 s} \left(11 q + 222\right) + 1395 q^{- 9 s} \left(q - 1\right)^{4 - 5 s} + 1254 q^{- 6 s} \left(q - 1\right)^{5 - 4 s} \\ &\quad +1224 q^{- 5 s} \left(q - 1\right)^{5 - 4 s} + 1123 q^{- 7 s} \left(q - 1\right)^{5 - 4 s} + 1061 q^{- 4 s} \left(q - 1\right)^{5 - 4 s} \\ &\quad +923 q^{- 8 s} \left(q - 1\right)^{5 - 4 s} + 911 q^{- 10 s} \left(q - 1\right)^{4 - 5 s} + 776 q^{- 3 s} \left(q - 1\right)^{5 - 4 s} + 670 q^{- 9 s} \left(q - 1\right)^{5 - 4 s} \\ &\quad +505 q^{- 11 s} \left(q - 1\right)^{4 - 5 s} + 494 q^{- 2 s} \left(q - 1\right)^{5 - 4 s} + 436 q^{- 10 s} \left(q - 1\right)^{5 - 4 s} + 394 q^{- 5 s} \left(q - 1\right)^{6 - 3 s} \\ &\quad +381 q^{- 4 s} \left(q - 1\right)^{6 - 3 s} + 369 q^{- 6 s} \left(q - 1\right)^{6 - 3 s} + 319 q^{- 3 s} \left(q - 1\right)^{6 - 3 s} + 299 q^{- 7 s} \left(q - 1\right)^{6 - 3 s} \\ &\quad +251 q^{- 11 s} \left(q - 1\right)^{5 - 4 s} + 242 q^{- 12 s} \left(q - 1\right)^{4 - 5 s} + 239 q^{- 2 s} \left(q - 1\right)^{6 - 3 s} + 230 q^{- 8 s} \left(q - 1\right)^{6 - 3 s} \\ &\quad +230 q^{- s} \left(q - 1\right)^{5 - 4 s} + 225 q^{- s} \left(q - 1\right)^{4 - 5 s} + 208 q^{- 12 s} \left(q - 1\right)^{3 - 6 s} + 154 q^{- 9 s} \left(q - 1\right)^{6 - 3 s} \\ &\quad +141 q^{- s} \left(q - 1\right)^{6 - 3 s} + 132 q^{- s} \left(q - 1\right)^{3 - 6 s} + 126 q^{- 12 s} \left(q - 1\right)^{5 - 4 s} + 98 q^{- 10 s} \left(q - 1\right)^{6 - 3 s} \\ &\quad +89 q^{- 13 s} \left(q - 1\right)^{4 - 5 s} + 67 q^{- 3 s} \left(q - 1\right)^{7 - 2 s} + 67 q^{- 4 s} \left(q - 1\right)^{7 - 2 s} + 61 q^{- 2 s} \left(q - 1\right)^{7 - 2 s} \\ &\quad +61 q^{- 5 s} \left(q - 1\right)^{7 - 2 s} + 58 q^{- 13 s} \left(q - 1\right)^{3 - 6 s} + 53 q^{- 13 s} \left(q - 1\right)^{5 - 4 s} + 51 q^{- 11 s} \left(q - 1\right)^{6 - 3 s} \\ &\quad +50 q^{- 6 s} \left(q - 1\right)^{7 - 2 s} + 48 q^{- s} \left(q - 1\right)^{7 - 2 s} + 43 q^{- s} \left(q - 1\right)^{2 - 7 s} + 34 q^{- 7 s} \left(q - 1\right)^{7 - 2 s} \\ &\quad +25 q^{- 12 s} \left(q - 1\right)^{6 - 3 s} + 25 q^{- 14 s} \left(q - 1\right)^{4 - 5 s} + 22 q^{- 8 s} \left(q - 1\right)^{7 - 2 s} + 18 q^{- 14 s} \left(q - 1\right)^{5 - 4 s} \\ &\quad +13 q^{- 9 s} \left(q - 1\right)^{7 - 2 s} + 12 q^{- 13 s} \left(q - 1\right)^{2 - 7 s} + 11 q^{- 13 s} \left(q - 1\right)^{6 - 3 s} + 9 q^{- 14 s} \left(q - 1\right)^{3 - 6 s} \\ &\quad +8 q^{- 12 s} \left(q - 1\right)^{1 - 8 s} + 7 q^{- 10 s} \left(q - 1\right)^{7 - 2 s} + 7 q^{- s} \left(q - 1\right)^{8 - s} + 6 q^{- 2 s} \left(q - 1\right)^{8 - s} \\ &\quad +6 q^{- s} \left(q - 1\right)^{1 - 8 s} + 5 q^{- 3 s} \left(q - 1\right)^{8 - s} + 5 q^{- 15 s} \left(q - 1\right)^{5 - 4 s} + 4 q^{- 4 s} \left(q - 1\right)^{8 - s} \\ &\quad +4 q^{- 14 s} \left(q - 1\right)^{6 - 3 s} + 4 q^{- 15 s} \left(q - 1\right)^{4 - 5 s} + 3 q^{- 5 s} \left(q - 1\right)^{8 - s} + 3 q^{- 11 s} \left(q - 1\right)^{7 - 2 s} \\ &\quad +2 q^{- 6 s} \left(q - 1\right)^{8 - s} + q^{- 7 s} \left(q - 1\right)^{8 - s} + q^{- 12 s} \left(q - 1\right)^{7 - 2 s} + q^{- 15 s} \left(q - 1\right)^{6 - 3 s} \\ &\quad +q^{- 16 s} \left(q - 1\right)^{5 - 4 s} + 70 \left(q - 1\right)^{5 - 4 s} + 56 \left(q - 1\right)^{6 - 3 s} + 56 \left(q - 1\right)^{4 - 5 s} + 28 \left(q - 1\right)^{7 - 2 s} \\ &\quad +28 \left(q - 1\right)^{3 - 6 s} + 8 \left(q - 1\right)^{8 - s} + 8 \left(q - 1\right)^{2 - 7 s} + \left(q - 1\right)^{1 - 8 s} + \left(q - 1\right)^{9} \\
\zeta_{\TT_{10}}(s) &= q^{- 6 s} \left(q - 1\right)^{1 - 9 s} \left(3 q + 17\right) \left(2 q^{2} + 13 q + 63\right) + 12 q^{- 13 s} \left(q - 1\right)^{3 - 7 s} \left(14 q + 349\right) \\ &\quad +10 q^{- 14 s} \left(q - 1\right)^{3 - 7 s} \left(7 q + 219\right) + 7 q^{- 5 s} \left(q - 1\right)^{4 - 6 s} \left(30 q + 1201\right) \\ &\quad +7 q^{- 12 s} \left(q - 1\right)^{4 - 6 s} \left(12 q + 967\right) + 5 q^{- 2 s} \left(q - 1\right)^{5 - 5 s} \left(q + 282\right) \\ &\quad +5 q^{- 10 s} \left(q - 1\right)^{5 - 5 s} \left(3 q + 1366\right) + 4 q^{- 2 s} \left(q - 1\right)^{4 - 6 s} \left(5 q + 333\right) \\ &\quad +4 q^{- 15 s} \left(q - 1\right)^{1 - 9 s} \left(q + 11\right) + 3 q^{- 2 s} \left(q - 1\right)^{3 - 7 s} \left(10 q + 259\right) \\ &\quad +3 q^{- 7 s} \left(q - 1\right)^{4 - 6 s} \left(118 q + 4475\right) + 3 q^{- 8 s} \left(q - 1\right)^{3 - 7 s} \left(9 q^{2} + 385 q + 4636\right) \\ &\quad +3 q^{- 11 s} \left(q - 1\right)^{3 - 7 s} \left(4 q^{2} + 205 q + 3159\right) + 2 q^{- 3 s} \left(q - 1\right)^{4 - 6 s} \left(27 q + 1496\right) \\ &\quad +2 q^{- 4 s} \left(q - 1\right)^{2 - 8 s} \left(8 q^{2} + 139 q + 889\right) + 2 q^{- 5 s} \left(q - 1\right)^{5 - 5 s} \left(13 q + 3104\right) \\ &\quad +2 q^{- 5 s} \left(q - 1\right)^{2 - 8 s} \left(q^{3} + 25 q^{2} + 293 q + 1620\right) + 2 q^{- 8 s} \left(q - 1\right)^{5 - 5 s} \left(19 q + 4337\right) \\ &\quad +2 q^{- 12 s} \left(q - 1\right)^{3 - 7 s} \left(3 q^{2} + 176 q + 3381\right) + 2 q^{- 13 s} \left(q - 1\right)^{4 - 6 s} \left(13 q + 2180\right) \\ &\quad +2 q^{- 14 s} \left(q - 1\right)^{4 - 6 s} \left(5 q + 1223\right) + 2 q^{- 15 s} \left(q - 1\right)^{2 - 8 s} \left(11 q + 170\right) \\ &\quad +2 q^{- 16 s} \left(q - 1\right)^{3 - 7 s} \left(2 q + 161\right) + q^{- 2 s} \left(q - 1\right)^{2 - 8 s} \left(20 q + 257\right) \\ &\quad +q^{- 2 s} \left(q - 1\right)^{1 - 9 s} \left(5 q + 37\right) + q^{- 3 s} \left(q - 1\right)^{5 - 5 s} \left(8 q + 2717\right) \\ &\quad +q^{- 3 s} \left(q - 1\right)^{3 - 7 s} \left(3 q^{2} + 117 q + 2039\right) + q^{- 3 s} \left(q - 1\right)^{2 - 8 s} \left(6 q^{2} + 104 q + 791\right) \\ &\quad +q^{- 3 s} \left(q - 1\right)^{1 - 9 s} \left(3 q^{2} + 33 q + 134\right) + q^{- 4 s} \left(q - 1\right)^{5 - 5 s} \left(21 q + 4418\right) \\ &\quad +q^{- 4 s} \left(q - 1\right)^{4 - 6 s} \left(122 q + 5413\right) + q^{- 4 s} \left(q - 1\right)^{3 - 7 s} \left(4 q^{2} + 273 q + 4096\right) \\ &\quad +q^{- 4 s} \left(q - 1\right)^{1 - 9 s} \left(q^{3} + 15 q^{2} + 108 q + 342\right) + q^{- 5 s} \left(q - 1\right)^{3 - 7 s} \left(19 q^{2} + 545 q + 6943\right) \\ &\quad +q^{- 5 s} \left(q - 1\right)^{1 - 9 s} \left(2 q^{3} + 35 q^{2} + 230 q + 659\right) + q^{- 6 s} \left(q - 1\right)^{5 - 5 s} \left(41 q + 7704\right) \\ &\quad +q^{- 6 s} \left(q - 1\right)^{4 - 6 s} \left(299 q + 11188\right) + q^{- 6 s} \left(q - 1\right)^{3 - 7 s} \left(21 q^{2} + 809 q + 9883\right) \\ &\quad +q^{- 6 s} \left(q - 1\right)^{2 - 8 s} \left(2 q^{3} + 84 q^{2} + 953 q + 4932\right) + q^{- 7 s} \left(q - 1\right)^{5 - 5 s} \left(36 q + 8593\right) \\ &\quad +q^{- 7 s} \left(q - 1\right)^{3 - 7 s} \left(37 q^{2} + 1103 q + 12627\right) \\ &\quad +q^{- 7 s} \left(q - 1\right)^{2 - 8 s} \left(8 q^{3} + 144 q^{2} + 1413 q + 6663\right) \\ &\quad +q^{- 7 s} \left(q - 1\right)^{1 - 9 s} \left(2 q^{4} + 21 q^{3} + 143 q^{2} + 654 q + 1524\right) + q^{- 8 s} \left(q - 1\right)^{4 - 6 s} \left(356 q + 14191\right) \\ &\quad +q^{- 8 s} \left(q - 1\right)^{2 - 8 s} \left(6 q^{3} + 135 q^{2} + 1587 q + 7639\right) \\ &\quad +q^{- 8 s} \left(q - 1\right)^{1 - 9 s} \left(q^{4} + 20 q^{3} + 165 q^{2} + 790 q + 1818\right) + q^{- 9 s} \left(q - 1\right)^{5 - 5 s} \left(20 q + 8073\right) \\ &\quad +q^{- 9 s} \left(q - 1\right)^{4 - 6 s} \left(322 q + 13751\right) + q^{- 9 s} \left(q - 1\right)^{3 - 7 s} \left(30 q^{2} + 1125 q + 13721\right) \\ &\quad +q^{- 9 s} \left(q - 1\right)^{2 - 8 s} \left(4 q^{3} + 125 q^{2} + 1513 q + 7536\right) \\ &\quad +q^{- 9 s} \left(q - 1\right)^{1 - 9 s} \left(10 q^{3} + 131 q^{2} + 723 q + 1771\right) + q^{- 10 s} \left(q - 1\right)^{4 - 6 s} \left(227 q + 12026\right) \\ &\quad +q^{- 10 s} \left(q - 1\right)^{3 - 7 s} \left(21 q^{2} + 934 q + 12376\right) \\ &\quad +q^{- 10 s} \left(q - 1\right)^{2 - 8 s} \left(8 q^{3} + 142 q^{2} + 1407 q + 6977\right) \\ &\quad +q^{- 10 s} \left(q - 1\right)^{1 - 9 s} \left(2 q^{4} + 22 q^{3} + 154 q^{2} + 704 q + 1669\right) + q^{- 11 s} \left(q - 1\right)^{4 - 6 s} \left(136 q + 9425\right) \\ &\quad +q^{- 11 s} \left(q - 1\right)^{2 - 8 s} \left(61 q^{2} + 889 q + 5143\right) \\ &\quad +q^{- 11 s} \left(q - 1\right)^{1 - 9 s} \left(6 q^{3} + 73 q^{2} + 433 q + 1181\right) + q^{- 12 s} \left(q - 1\right)^{2 - 8 s} \left(32 q^{2} + 526 q + 3607\right) \\ &\quad +q^{- 12 s} \left(q - 1\right)^{1 - 9 s} \left(2 q^{3} + 36 q^{2} + 255 q + 805\right) + q^{- 13 s} \left(q - 1\right)^{2 - 8 s} \left(10 q^{2} + 249 q + 2095\right) \\ &\quad +q^{- 13 s} \left(q - 1\right)^{1 - 9 s} \left(10 q^{2} + 110 q + 431\right) + q^{- 14 s} \left(q - 1\right)^{2 - 8 s} \left(5 q^{2} + 101 q + 983\right) \\ &\quad +q^{- 14 s} \left(q - 1\right)^{1 - 9 s} \left(2 q^{2} + 33 q + 171\right) + q^{- 15 s} \left(q - 1\right)^{3 - 7 s} \left(20 q + 929\right) \\ &\quad +q^{- 16 s} \left(q - 1\right)^{2 - 8 s} \left(2 q + 91\right) + 5331 q^{- 11 s} \left(q - 1\right)^{5 - 5 s} + 3802 q^{- 12 s} \left(q - 1\right)^{5 - 5 s} \\ &\quad +3288 q^{- 7 s} \left(q - 1\right)^{6 - 4 s} + 3213 q^{- 6 s} \left(q - 1\right)^{6 - 4 s} + 3128 q^{- 8 s} \left(q - 1\right)^{6 - 4 s} \\ &\quad +2802 q^{- 5 s} \left(q - 1\right)^{6 - 4 s} + 2724 q^{- 9 s} \left(q - 1\right)^{6 - 4 s} + 2472 q^{- 13 s} \left(q - 1\right)^{5 - 5 s} \\ &\quad +2228 q^{- 4 s} \left(q - 1\right)^{6 - 4 s} + 2216 q^{- 10 s} \left(q - 1\right)^{6 - 4 s} + 1657 q^{- 11 s} \left(q - 1\right)^{6 - 4 s} \\ &\quad +1544 q^{- 3 s} \left(q - 1\right)^{6 - 4 s} + 1442 q^{- 14 s} \left(q - 1\right)^{5 - 5 s} + 1192 q^{- 15 s} \left(q - 1\right)^{4 - 6 s} \\ &\quad +1149 q^{- 12 s} \left(q - 1\right)^{6 - 4 s} + 937 q^{- 2 s} \left(q - 1\right)^{6 - 4 s} + 759 q^{- 15 s} \left(q - 1\right)^{5 - 5 s} + 757 q^{- 6 s} \left(q - 1\right)^{7 - 3 s} \\ &\quad +729 q^{- 5 s} \left(q - 1\right)^{7 - 3 s} + 725 q^{- 13 s} \left(q - 1\right)^{6 - 4 s} + 700 q^{- 7 s} \left(q - 1\right)^{7 - 3 s} + 655 q^{- 4 s} \left(q - 1\right)^{7 - 3 s} \\ &\quad +607 q^{- 8 s} \left(q - 1\right)^{7 - 3 s} + 525 q^{- s} \left(q - 1\right)^{5 - 5 s} + 524 q^{- 3 s} \left(q - 1\right)^{7 - 3 s} + 492 q^{- 16 s} \left(q - 1\right)^{4 - 6 s} \\ &\quad +480 q^{- 9 s} \left(q - 1\right)^{7 - 3 s} + 427 q^{- s} \left(q - 1\right)^{6 - 4 s} + 426 q^{- 14 s} \left(q - 1\right)^{6 - 4 s} + 413 q^{- s} \left(q - 1\right)^{4 - 6 s} \\ &\quad +377 q^{- 2 s} \left(q - 1\right)^{7 - 3 s} + 365 q^{- 10 s} \left(q - 1\right)^{7 - 3 s} + 346 q^{- 16 s} \left(q - 1\right)^{5 - 5 s} \\ &\quad +249 q^{- 11 s} \left(q - 1\right)^{7 - 3 s} + 225 q^{- 15 s} \left(q - 1\right)^{6 - 4 s} + 217 q^{- s} \left(q - 1\right)^{7 - 3 s} + 203 q^{- s} \left(q - 1\right)^{3 - 7 s} \\ &\quad +163 q^{- 12 s} \left(q - 1\right)^{7 - 3 s} + 155 q^{- 17 s} \left(q - 1\right)^{4 - 6 s} + 133 q^{- 17 s} \left(q - 1\right)^{5 - 5 s} \\ &\quad +105 q^{- 16 s} \left(q - 1\right)^{6 - 4 s} + 97 q^{- 4 s} \left(q - 1\right)^{8 - 2 s} + 94 q^{- 5 s} \left(q - 1\right)^{8 - 2 s} + 93 q^{- 3 s} \left(q - 1\right)^{8 - 2 s} \\ &\quad +92 q^{- 13 s} \left(q - 1\right)^{7 - 3 s} + 85 q^{- 6 s} \left(q - 1\right)^{8 - 2 s} + 82 q^{- 2 s} \left(q - 1\right)^{8 - 2 s} + 74 q^{- 17 s} \left(q - 1\right)^{3 - 7 s} \\ &\quad +70 q^{- 7 s} \left(q - 1\right)^{8 - 2 s} + 63 q^{- s} \left(q - 1\right)^{8 - 2 s} + 57 q^{- s} \left(q - 1\right)^{2 - 8 s} + 50 q^{- 8 s} \left(q - 1\right)^{8 - 2 s} \\ &\quad +50 q^{- 14 s} \left(q - 1\right)^{7 - 3 s} + 43 q^{- 17 s} \left(q - 1\right)^{6 - 4 s} + 42 q^{- 18 s} \left(q - 1\right)^{5 - 5 s} + 35 q^{- 18 s} \left(q - 1\right)^{4 - 6 s} \\ &\quad +34 q^{- 9 s} \left(q - 1\right)^{8 - 2 s} + 25 q^{- 15 s} \left(q - 1\right)^{7 - 3 s} + 22 q^{- 10 s} \left(q - 1\right)^{8 - 2 s} + 16 q^{- 18 s} \left(q - 1\right)^{6 - 4 s} \\ &\quad +13 q^{- 11 s} \left(q - 1\right)^{8 - 2 s} + 12 q^{- 17 s} \left(q - 1\right)^{2 - 8 s} + 11 q^{- 16 s} \left(q - 1\right)^{7 - 3 s} + 9 q^{- 18 s} \left(q - 1\right)^{3 - 7 s} \\ &\quad +9 q^{- 19 s} \left(q - 1\right)^{5 - 5 s} + 8 q^{- 16 s} \left(q - 1\right)^{1 - 9 s} + 8 q^{- s} \left(q - 1\right)^{9 - s} + 7 q^{- 2 s} \left(q - 1\right)^{9 - s} \\ &\quad +7 q^{- 12 s} \left(q - 1\right)^{8 - 2 s} + 7 q^{- s} \left(q - 1\right)^{1 - 9 s} + 6 q^{- 3 s} \left(q - 1\right)^{9 - s} + 5 q^{- 4 s} \left(q - 1\right)^{9 - s} \\ &\quad +5 q^{- 19 s} \left(q - 1\right)^{6 - 4 s} + 4 q^{- 5 s} \left(q - 1\right)^{9 - s} + 4 q^{- 17 s} \left(q - 1\right)^{7 - 3 s} + 4 q^{- 19 s} \left(q - 1\right)^{4 - 6 s} \\ &\quad +3 q^{- 6 s} \left(q - 1\right)^{9 - s} + 3 q^{- 13 s} \left(q - 1\right)^{8 - 2 s} + 2 q^{- 7 s} \left(q - 1\right)^{9 - s} + q^{- 8 s} \left(q - 1\right)^{9 - s} \\ &\quad +q^{- 14 s} \left(q - 1\right)^{8 - 2 s} + q^{- 18 s} \left(q - 1\right)^{7 - 3 s} + q^{- 20 s} \left(q - 1\right)^{6 - 4 s} + q^{- 20 s} \left(q - 1\right)^{5 - 5 s} \\ &\quad +126 \left(q - 1\right)^{6 - 4 s} + 126 \left(q - 1\right)^{5 - 5 s} + 84 \left(q - 1\right)^{7 - 3 s} + 84 \left(q - 1\right)^{4 - 6 s} + 36 \left(q - 1\right)^{8 - 2 s} \\ &\quad +36 \left(q - 1\right)^{3 - 7 s} + 9 \left(q - 1\right)^{9 - s} + 9 \left(q - 1\right)^{2 - 8 s} + \left(q - 1\right)^{1 - 9 s} + \left(q - 1\right)^{10} . \tag*{\qed}
    \end{align*}
\end{theorem}

The computation times\footnote{As performed on an Intel\textregistered Xeon\textregistered CPU E5-4640 0 @ 2.40GHz.} for the representation zeta functions $\zeta_{\UU_n}(s)$ are given by
\begin{center}
    \begin{tabular}{c|c|c|c|c|c|c|c|c|c|c}
         $n$ & 1 & 2 & 3 & 4 & 5 & 6 & 7 & 8 & 9 & 10 \\ \hline
         time & 0.00s & 0.01s & 0.08s & 0.25s & 0.77s & 2.41s & 8.18s & 27.43s & 1m44s & 6m52s
    \end{tabular}
\end{center}
and those for the representation zeta functions $\zeta_{\TT_n}(s)$ are given by
\begin{center}
    \begin{tabular}{c|c|c|c|c|c|c|c|c|c|c}
         $n$ & 1 & 2 & 3 & 4 & 5 & 6 & 7 & 8 & 9 & 10 \\ \hline
         time & 0.00s & 0.03s & 0.15s & 0.51s & 1.81s & 6.32s & 26.16s & 1m53s & 6m39s & 55m46s
    \end{tabular}
\end{center}

As an example, applying Theorem \ref{thm:katz_theorem} and Corollary \ref{cor:point_count_from_zeta_function} with $G = \TT_6$, we obtain
\begin{align*}
    e(R_{\TT_6}(\Sigma_g))
        &= q^{18 g - 3} \left(q - 1\right)^{6 g + 3} + q^{18 g - 3} \left(q - 1\right)^{8 g + 2} + 2 q^{20 g - 5} \left(q - 1\right)^{4 g + 4} + 5 q^{20 g - 5} \left(q - 1\right)^{6 g + 3} \\ &\quad + 3 q^{20 g - 5} \left(q - 1\right)^{8 g + 2} + 2 q^{22 g - 7} \left(q - 1\right)^{2 g + 5} + 10 q^{22 g - 7} \left(q - 1\right)^{4 g + 4} + 15 q^{22 g - 7} \left(q - 1\right)^{6 g + 3} \\ &\quad + 7 q^{22 g - 7} \left(q - 1\right)^{8 g + 2} + q^{22 g - 7} \left(q - 1\right)^{10 g + 1} + 4 q^{24 g - 9} \left(q - 1\right)^{2 g + 5} + 17 q^{24 g - 9} \left(q - 1\right)^{4 g + 4} \\ &\quad + 24 q^{24 g - 9} \left(q - 1\right)^{6 g + 3} + 13 q^{24 g - 9} \left(q - 1\right)^{8 g + 2} + 2 q^{24 g - 9} \left(q - 1\right)^{10 g + 1} \\ &\quad + q^{26 g - 11} \left(q - 1\right)^{2 g + 5} \left(q + 7\right) + 23 q^{26 g - 11} \left(q - 1\right)^{4 g + 4} + 29 q^{26 g - 11} \left(q - 1\right)^{6 g + 3} \\
        &\quad + 16 q^{26 g - 11} \left(q - 1\right)^{8 g + 2} + 3 q^{26 g - 11} \left(q - 1\right)^{10 g + 1} + 3 q^{28 g - 13} \left(q - 1\right)^{2 g + 5} \\ &\quad + 13 q^{28 g - 13} \left(q - 1\right)^{4 g + 4} + 21 q^{28 g - 13} \left(q - 1\right)^{6 g + 3} + 15 q^{28 g - 13} \left(q - 1\right)^{8 g + 2} \\ &\quad + 4 q^{28 g - 13} \left(q - 1\right)^{10 g + 1} + q^{30 g - 15} \left(q - 1\right)^{12 g} + q^{30 g - 15} \left(q - 1\right)^{2 g + 5} + 5 q^{30 g - 15} \left(q - 1\right)^{4 g + 4} \\ &\quad + 10 q^{30 g - 15} \left(q - 1\right)^{6 g + 3} + 10 q^{30 g - 15} \left(q - 1\right)^{8 g + 2} + 5 q^{30 g - 15} \left(q - 1\right)^{10 g + 1} .
\end{align*}

%% file: conclusion.tex
\section{Conclusion}

\textbf{Comparison of methods}. We have seen two very different methods to compute very similar invariants of the $G$-representation variety $R_G(\Sigma_g)$.

An obvious advantages of the TQFT method is that it produces a more refined invariant, the virtual class in $\K(\Var_\CC)$. However, in practice the virtual class and $E$-polynomial seem to hold the same information, as the virtual classes in our results are all contained in the subring $\ZZ[q] \subset \K(\Var_\CC)$, and are therefore completely determined by the corresponding $E$-polynomials. This appears also to be the case for $G = \SL_2$ \cite{GonzalezPrieto2020}, so it would be interesting to see if there are linear algebraic groups for which this is not the case.

A possibly more significant advantage of the TQFT method is that it provides deeper geometric insight. From the TQFT we can directly understand the fibration
\[ G^{2g} \to G, \quad (A_1, B_1, \ldots, A_g, B_g) \mapsto \prod_{i = 1}^{g} [A_i, B_i] , \]
in $\K(\Var/G)$, and not only the fiber over the identity, that is, the $G$-representation variety. This allows the method to be more flexible. As seen in Subsection \ref{subsec:twisted_representation_varieties}, the TQFT method can be slightly modified to also study the twisted representation varieties. In a similar spirit, the TQFT can naturally be lifted to a $G$-equivariant setting in order to study the $G$-character stacks $[R_G(\Sigma_g) / G]$ \cite{GonzalezPrietoHablicsekVogel2022}.

On the other hand, the arithmetic method is orders of magnitude faster and the complexity of the computations remains quite manageable as the rank of the groups $\TT_n$ and $\UU_n$ increases. Once the representation zeta function has been computed, the $E$-polynomial of $R_G(\Sigma_g)$ can be obtained directly by evaluation, whereas for the TQFT method the map $Z(\bdgenus)$ must be diagonalized, which is a very expensive operation.

Of course, for general algebraic groups $G$, the arithmetic method is not as systematic as the TQFT method, since there is no recipe how to compute the irreducible representations of $G(\FF_q)$. 

\textbf{Extending the TQFT method}.
Naturally, one can ask whether the TQFT method can compute the virtual class of $R_{\TT_n}(\Sigma_g)$ for $n \ge 6$. Unfortunately, this is not possible without major modifications to the method for a number of reasons. First of all, we would like to take the conjugacy classes of $G$ as a generating set of a submodule of $\K(\Var/G)$, but for $n = 6$ it is shown \cite{Bhunia2019} that there are infinitely many unipotent conjugacy classes in $G = \TT_n$. For example, an infinite family of non-conjugate elements in $\TT_6$ is given by
\[ \smatrix{1 & 1 & \alpha & 0 & 0 & 0 \\ 0 & 1 & 0 & 0 & 1 & 0 \\ 0 & 0 & 1 & 1 & 1 & 0 \\ 0 & 0 & 0 & 1 & 0 & 0 \\ 0 & 0 & 0 & 0 & 1 & 1 \\ 0 & 0 & 0 & 0 & 0 & 1} \quad \textup{ for } \alpha \in k , \]
and moreover every such element appears as the commutator of elements in $\TT_6$. Secondly, from a practical point of view, we do not expect the computations to be feasible, both looking at the computation times and the prospect of diagonalizing a polynomial-valued matrix of size $\ge 250$.

\textbf{Motivic Higman's conjecture}. As mentioned, it is very remarkable that all virtual classes in our results are all contained in the subring $\ZZ[q] \subset \K(\Var_\CC)$. Therefore, in order to compute the virtual classes of $R_{\TT_n}(\Sigma_g)$ and $R_{\UU_n}(\Sigma_g)$ for $n \ge 6$, instead of using the TQFT method, it seems more natural to try to prove the Motivic Higman's Conjecture \ref{conj:motivic_higman}. Then it suffices to compute the $E$-polynomials, which can be done more efficient using the arithmetic method.

%% file: main.bib
@article{PakSoffer2015,
  title  = {On {H}igman's {$k(U_n(q))$} conjecture},
  author = {Pak, I. and Soffer, A.},
  year          = {2015},
  archiveprefix = {arXiv},
  primaryclass  = {math.CO},
  eprint        = {1507.00411}
}

@book{Serre1997,
  author    = {Serre, J.P.},
  isbn      = {978-3-540-90190-7},
  publisher = {Springer},
  series    = {Graduate texts in mathematics},
  title     = {Linear representations of finite groups},
  volume    = 42,
  year      = 1977
}

@article{Bhunia2019,
  author     = {Bhunia, S.},
  title      = {Conjugacy classes of centralizers in the group of upper
                triangular matrices},
  journal    = {Journal of Algebra and its Applications},
  volume     = {19},
  year       = {2020},
  number     = {1},
  issn       = {0219-4988},
  doi        = {10.1142/S0219498820500085},
  url        = {https://doi.org/10.1142/S0219498820500085}
}

@misc{Milne2015,
  author = {Milne, J.S.},
  title  = {Algebraic Groups (v2.00)},
  year   = {2015},
  note   = {Available at \url{www.jmilne.org/math/}},
  pages  = {528}
}

@article{GonzalezPrietoLogaresMunoz2020,
  author  = {González-Prieto, Á. and Logares, M. and Muñoz, V.},
  title   = {A Lax Monoidal Topological Quantum Field Theory For Representation Varieties},
  journal = {B. Sci. Math.},
  volume  = {161},
  year    = {2020},
  note    = {102871}
}

@article{GonzalezPrieto2020,
  author  = {González-Prieto, Á.},
  title   = {Virtual Classes of Parabolic $\text{SL}_2(\mathbb{C})$-Character Varieties},
  journal = {Adv. Math.},
  volume  = {368},
  year    = {2020},
  pages   = {107--148}
}

@article{HauselRodriguezVillegas2008,
  author  = {Hausel, T. and Rodriguez-Villegas, F.},
  title   = {Mixed Hodge polynomials of character varieties},
  journal = {Invent. Math.},
  volume  = {174},
  number  = {3},
  year    = {2008},
  pages   = {555-624}
}

@phdthesis{Mereb2018,
  author = {Mereb, M.},
  title  = {{On the $E$-polynomials of a family of character varieties}},
  year   = {2018},
  school = {The University of Texas at Austin}
}

@article{LogaresMunozNewstead2013,
  author  = {Logares, M. and Muñoz, V. and Newstead, P.E.},
  title   = {Hodge Polynomials of $\text{SL}_2(\mathbb{C})$-Character Varieties for Curves of Small Genus},
  journal = {{Rev. Mat. Complut.}},
  volume  = {26},
  number  = {2},
  year    = {2013},
  pages   = {635--703}
}

@article{GonzalezPrietoHablicsekVogel2022,
  author        = {González-Prieto, Á. and Hablicsek, M. and Vogel, J.T.},
  title         = {Virtual Classes of Character Stacks},
  year          = {2022},
  archiveprefix = {arXiv},
  primaryclass  = {math.AG},
  eprint        = {2201.08699}
}

@article{Martinez2017,
  author        = {Martínez, J.},
  title         = {$E$-polynomials of $\textup{PGL}(2, \mathbb{C})$ character varieties of surface groups},
  year          = {2017},
  archiveprefix = {arXiv},
  primaryclass  = {math.AG},
  eprint        = {1705.04649}
}

@article{HablicsekVogel2020,
  author  = {Hablicsek, M. and Vogel, J.T.},
  title   = {Virtual Classes of Representation Varieties of Upper Triangular Matrices via Topological Quantum Field Theories},
  year    = {2022},
  journal = {Symmetry, Integrability and Geometry: Methods and Applications},
  volume  = {18},
  number  = {095}
}

@misc{GitHubMathCode,
  author = {Vogel, J.T.},
  title  = {Math Code},
  year   = {2022},
  note   = {Available at \url{https://github.com/jessetvogel/math-code}}
}

@book{AdamsLoustaunau1994,
  author    = {Adams, W.W. and Loustaunau, P.},
  publisher = {American Mathematical Soc.},
  title     = {An Introduction to Gröbner Bases},
  year      = {1994}
}

@article{Higman1960a,
  author     = {Higman, G.},
  title      = {Enumerating {$p$}-groups. {I}. {I}nequalities},
  journal    = {Proc. London Math. Soc. (3)},
  fjournal   = {Proceedings of the London Mathematical Society. Third Series},
  volume     = {10},
  year       = {1960},
  pages      = {24--30},
  issn       = {0024-6115},
  mrclass    = {20.00},
  mrnumber   = {113948},
  mrreviewer = {H. A. Thurston},
  doi        = {10.1112/plms/s3-10.1.24},
  url        = {https://doi.org/10.1112/plms/s3-10.1.24}
}

@article{Simpson1994a,
  author  = {Simpson, C.T.},
  title   = {Moduli of representations of the fundamental group of a smooth projective variety {I}},
  journal = {Inst. Hautes Études Sci. Publ. Math},
  volume  = {79},
  pages   = {47--129},
  year    = {1994}
}

@article{Simpson1994b,
  author  = {Simpson, C.T.},
  title   = {Moduli of representations of the fundamental group of a smooth projective variety {II}},
  journal = {Inst. Hautes Études Sci. Publ. Math},
  volume  = {79},
  pages   = {5--79},
  year    = {1994}
}

@article{BaragliaHekmati2017,
  author     = {Baraglia, D. and Hekmati, P.},
  title      = {Arithmetic of singular character varieties and their {$E$}-polynomials},
  journal    = {Proc. Lond. Math. Soc. (3)},
  fjournal   = {Proceedings of the London Mathematical Society. Third Series},
  volume     = {114},
  year       = {2017},
  number     = {2},
  pages      = {293--332},
  issn       = {0024-6115},
  mrclass    = {14D20 (14G15 14L30 20C15 32S35)},
  mrnumber   = {3653231},
  mrreviewer = {Andr\'{e} Oliveira},
  doi        = {10.1112/plms.12008},
  url        = {https://doi.org/10.1112/plms.12008}
}

@article{VeraLopezArregi2003,
  author     = {Vera-L\'{o}pez, A. and Arregi, J.M.},
  title      = {Conjugacy classes in unitriangular matrices},
  journal    = {Linear Algebra Appl.},
  fjournal   = {Linear Algebra and its Applications},
  volume     = {370},
  year       = {2003},
  pages      = {85--124},
  issn       = {0024-3795},
  mrclass    = {20G40 (20E45)},
  mrnumber   = {1994321},
  mrreviewer = {Thomas Weigel},
  doi        = {10.1016/S0024-3795(03)00371-9},
  url        = {https://doi.org/10.1016/S0024-3795(03)00371-9}
}

@article{HalasiPalfy2011,
  url         = {https://doi.org/10.1515/jgt.2010.081},
  title       = {The number of conjugacy classes in pattern groups is not a polynomial function},
  author      = {Halasi, Z. and Pálfy, P.P.},
  pages       = {841--854},
  volume      = {14},
  number      = {6},
  journal     = {Journal of Group Theory},
  doi         = {doi:10.1515/jgt.2010.081},
  year        = {2011},
  lastchecked = {2022-12-13}
}

@article{Serre1958,
  author    = {Serre, J.P.},
  title     = {Espaces fibr\'es alg\'ebriques},
  journal   = {S\'eminaire Claude Chevalley},
  publisher = {Secr\'etariat math\'ematique},
  volume    = {3},
  year      = {1958},
  language  = {fr}
}

@book{Milne1980,
  title     = {Etale Cohomology (PMS-33)},
  author    = {Milne, J.S.},
  publisher = {Princeton University Press},
  address   = {Princeton},
  doi       = {doi:10.1515/9781400883981},
  isbn      = {9781400883981},
  year      = {1980}
}

@article{HauselMellitMinetsSchiffman2022,
  title         = {{$P = W$} via {$\mathcal{H}_2$}},
  author        = {Hausel, T. and Mellit, A. and Minets, A. and Schiffmann, O.},
  year          = {2022},
  archiveprefix = {arXiv},
  primaryclass  = {math.AG},
  eprint        = {2209.05429}
}

@article{MaulikShen2022,
  title         = {The {$P = W$} conjecture for $\textup{GL}_n$},
  author        = {Maulik, D. and Shen, J.},
  year          = {2022},
  archiveprefix = {arXiv},
  primaryclass  = {math.AG},
  eprint        = {2209.02568}
}

@article{MartinezMunoz2016,
  author     = {Martínez, J. and Muñoz, V.},
  title      = {$E$-polynomials of the {$\textup{SL}(2, \mathbb{C})$}-character
                varieties of surface groups},
  journal    = {Int. Math. Res. Not. IMRN},
  fjournal   = {International Mathematics Research Notices. IMRN},
  year       = {2016},
  number     = {3},
  pages      = {926--961},
  issn       = {1073-7928},
  mrnumber   = {3493438}
}
